\documentclass[aos,preprint]{imsart}

\RequirePackage[OT1]{fontenc}
\usepackage{amsthm,amsmath,multirow,graphicx,booktabs,url,mathtools,wrapfig,mathrsfs,bm,relsize,caption,xcolor, enumitem, float,bbm,subcaption}

\RequirePackage[authoryear]{natbib}
\RequirePackage[psamsfonts]{amssymb}
 
\usepackage[colorlinks=true,
linkcolor=blue,
urlcolor=blue,
citecolor=blue]{hyperref}

\numberwithin{equation}{section}
\newtheorem{conj}{Conjecture}
\newtheorem{thm}[conj]{Theorem}
\newtheorem{cor}[conj]{Corollary}
\newtheorem{prop}[conj]{Proposition}
\newtheorem{lemma}[conj]{Lemma}

\newtheorem{innercustomgeneric}{\customgenericname}
\providecommand{\customgenericname}{}
\newcommand{\newcustomtheorem}[2]{%
	\newenvironment{#1}[1] 
	{%
		\renewcommand\customgenericname{#2}%
		\renewcommand\theinnercustomgeneric{##1}%
		\innercustomgeneric
	}
	{\endinnercustomgeneric}
}

\newcustomtheorem{customAss}{Assumption}

\theoremstyle{remark}\newtheorem{remark}{Remark}

\def\Cov{{\rm Cov}}   
\def\PP{\mathbb{P}}
\def\EE{\mathbb{E}}
\def\RR{\mathbb{R}}
\def\wh{\widehat}
\def\eps{\varepsilon}
\def\bI{\bm{I}}
\def\wt{\widetilde}

\def\a{\alpha}
\def\T{\top}
\def\i{\infty}
\def\sz{\Sigma_Z}

\def\szy{\Sigma_{Z|Y}}

\def\C{\szy}
\def\sw{\Sigma_W}
\def\errw{\delta_W}

\def\tr{{\rm tr}}
\def\op{{\rm op}}
\def\diag{{\rm diag}}
\def\KL{{\rm KL}}
\def\rI{\textrm{I}}
\def\rII{\textrm{II}}
\def\rIII{\textrm{III}}

\def\Dt{\Delta} 
\def\ba{\bar \alpha}

\def\X{{\bm X}}
\def\Y{{\bm Y}}
\def\Z{{\bm Z}}
\def\W{{\bm W}}
\def\G{{\bm G}}
\def\U{{\bm U}}
\def\be{\bm e}
\def\R{\bm R}
\def\D{{\bm D}}
\def\V{{\bm V}}

\def\E{\mathcal{E}}

\def\cO{\mathcal O}
\def\cS{\mathcal {S}}
\def\cN{\mathcal {N}}
\def\cL{\mathcal{L}}

\def\0{\bm 0}
\def\1{\mathbbm 1}
\def\b1{\bm 1}

\def\mut{\mu_{\theta}}
\def\mutp{\mu_{\theta'}}
\def\Sigt{\Sigma_{\theta}}
\def\Sigtp{\Sigma_{\theta'}}
\def\tX{\tilde \X}

\def\e{\bm  e}
\DeclareMathOperator*{\argmin}{arg\,min}

\def\vec{{\rm vec}}

\DeclareMathOperator*{\argmax}{arg\,max}

\begin{document}
	
	\begin{frontmatter}
		\title{Optimal Discriminant Analysis in High-Dimensional Latent Factor Models}
		\runtitle{Optimal Discriminant Analysis}
		
		\begin{aug}
			\author[A]{\fnms{Xin} \snm{Bing}  \ead[label=e1]{xin.bing@utoronto.ca}}
			\and
			\author[B]{\fnms{Marten} \snm{Wegkamp} %\thanksref{t1} 
   \ead[label=e2]{mhw73@cornell.edu}}
			
			\runauthor{Xin Bing and Marten Wegkamp}
			
		%%	\thankstext{t1}{Corresponding author. }  
			
			\address[A]{
				Department of Statistical Sciences,
				University of Toronto,
				\printead{e1}
			}
			\address[B]{
				Department of Mathematics \& Department of Statistics and Data Science,
				Cornell University,
				\printead{e2}}
		\end{aug}
		
		\begin{abstract}
			In high-dimensional classification problems, a commonly used approach is to first project the high-dimensional features into a lower dimensional space, and base the classification on the resulting lower dimensional projections. In this paper, we formulate a latent-variable model with a hidden low-dimensional structure to justify this two-step procedure and to guide which projection to choose. We propose a computationally efficient classifier that takes certain principal components (PCs) of the observed features as projections, with the number of retained PCs selected in a data-driven way. A general theory is established for analyzing such two-step classifiers based on any  
			projections. We derive explicit rates of convergence of the excess risk of the proposed PC-based classifier. The obtained rates are further shown to be optimal up to logarithmic factors in the minimax sense. Our theory allows  
			the lower dimension to grow with the sample size and is also valid even when the feature dimension (greatly) exceeds the sample size. Extensive simulations corroborate our theoretical findings. The proposed method also performs favorably  relative to  other existing discriminant methods on three real data examples. 
		\end{abstract}
		
		
		
	\end{frontmatter}

	\section{Introduction}
	
	In high-dimensional classification problems, a widely used technique is to first project the high-dimensional features    into a lower dimensional space, and base the classification on the resulting lower dimensional projections \cite{antoniadis2003effective,BBW2003,boulesteix2004pls,chiaromonte2002dimension,dai2006dimension,ghosh2001singular,hadef2019proposal,jin2021classification,li2016accurate,ma2020unsupervised,mallary2022acoustic,Nguyen2002}.  Despite having been widely used for years, theoretical understanding of this approach is scarce, and what kind of low-dimensional projection to choose remains unknown. In this paper we formulate a latent-variable model with
	a hidden low-dimensional structure to justify the two-step procedure that takes leading principal components of the observed features as projections. 
	
	Concretely,  suppose our data consists of independent copies of the pair $(X,Y)$ with features $X\in \RR^p$ according to 
	\begin{equation}\label{model_X}
		X = AZ + W
	\end{equation}
	and labels $Y\in\{0,1\}$. 
	Here $A$ is a  deterministic, unknown $p\times K$  loading matrix,
	$Z\in \RR^K$ are unobserved, latent factors and $W$ is random noise.
	We assume that
	\begin{itemize}
		\item [(i)] $W$ is independent of both $Z$ and $Y$,
		\item[(ii)] $\EE[W]=\0_p$,
		\item[(iii)] $A$ has rank $K$. 
	\end{itemize}
	This  mathematical framework allows for a substantial dimension reduction in classification for $K\ll p$. Indeed, in terms of the Bayes' misclassification errors,
	we prove in Lemma \ref{lem:RxvsRz} of Section \ref{sec_benchmark} the inequality 
	\begin{align}\label{eq_RxRz}
		R_x^*  := \inf_{g} \PP\{ g(X) \ne Y \} & ~\ge~  R_z^*:=\inf_{h} \PP\{ h(Z) \ne Y \} ,
	\end{align}  
	that is,
	it is easier to classify in the latent space $\RR^K$ than in 
	the observed feature space $\RR^p$.  
	In this work, we further assume that
	\begin{itemize}
		\item[(iv)] 
		$Z$ is a mixture of two Gaussians 
		\begin{equation}\label{model_YZ}
			Z \mid  Y = k \sim N_K(\a_k, \szy),\qquad \PP(Y=k) = \pi_k, \qquad k\in \{0,1\}
		\end{equation}
		with different means $\a_0:= \EE [Z \mid Y=0] $ and $\a_1:= \EE[Z \mid Y=1] $, but with the same covariance matrix 
		\begin{equation}\label{def_szy}
			\szy:= \Cov{(Z \mid Y=0)}= \Cov{(Z \mid  Y=1)},
		\end{equation}	
		assumed to be strictly positive definite. 
	\end{itemize}
	We emphasize that the distributions of $X$ given $Y$ are not necessarily Gaussian as the distribution of $W$ could be arbitrary.

	Within the above modelling framework, parameters related with the moments of $X$ and $Y$, such as 
	$\pi_k$, $\EE[X|Y]$ and $\Cov(X|Y)$, are identifiable, while
	$A$, $\szy$, $\a_k$, and $\sw := \Cov(W)$ are not. For instance, we can always replace $Z$ by $Z' = QZ$ for any invertible $K\times K$ matrix $Q$ and write $\a_k' = Q\a_k$, $\szy'= Q\szy Q^\T$ and $A' = AQ^{-1}$.  
	Since we focus on classification, there is no need to impose any  conditions on the latter group of parameters that render them identifiable.  
	Although our discussion throughout this paper is based on a fixed notation of $A$, $\szy$, $\sw$ and $\a_k$, it should be understood that our results are valid for all possible choices of these parameters such that model (\ref{model_X}) and (\ref{model_YZ}) holds, including sub-models under which such parameters are (partially) identifiable.\\

	Our goal is to construct a classification rule $\wh g_x:\RR^p \to\{0,1\}$
	based on the training data $\D := \{\X, \Y\}$  that consists of independent pairs $(X_1,Y_1),\ldots,  (X_n,Y_n)$ from model (\ref{model_X}) and (\ref{model_YZ}) such that  the resulting rule has small   missclassification error $\PP\{ \wh g_x(X)\ne Y\} $
	for a new pair of $(X,Y)$ from the same model that is independent of $\D$.  
	In this paper, we are particularly interested in  $\wh g_x$ that is linear in $X$, motivated by the fact that the restriction of equal covariance in (\ref{def_szy}) leads to a Bayes rule that is linear in $Z$ when we observe $Z$ (see display (\ref{rule_Z_0}) below). 
	
	Linear classifiers have been popular for decades, especially in high-dimensional classification problems, due to their interpretability and  computational simplicity.  
	One strand of the existing literature imposes sparsity on the   coefficients $\beta\in\RR^p$ in linear classifiers $g(x)= \1\{ \beta ^\T x +\beta_0\ge 0\}$ for large $p$ ($p\ge n$), see, for instance,  \cite{CaiLiu2011,caizhang2019, FanFan2008,mai2012,Shao2011,Tibshirani2002,Witten2011} for sparse linear discriminant analysis (LDA) and  \cite{TG,WY}  for sparse support vector machines.
	For instance, in the classical LDA-setting, when $X$ itself is a mixture of Gaussians
	\begin{equation}\label{model_Gaussian}
		X \mid Y=k \sim N_p(\mu_k, \Sigma),\qquad \PP(Y=k) = \pi_k,\qquad   k\in \{0,1\}
	\end{equation}
	with $\Sigma$  strictly positive definite, the Bayes classifier is linear with  $p$-dimensional vector $\beta=\Sigma^{-1}(\mu_1 - \mu_0)$. 
	Sparsity of $\beta$ is then a reasonable assumption 
	when $\Sigma$ is close to diagonal, so that   sparsity of $\beta$ gets translated to that of the difference between the mean vectors $\mu_1 - \mu_0$. However, in the high-dimensional regime, many features are highly correlated and any sparsity assumption on $\beta$ is no longer intuitive and becomes in fact questionable. This serves as a main motivation for this work, in which we study a class of linear classifiers that no longer requires the sparsity assumption on $\beta$, for neither construction of the classifier, nor its analysis.
	
	\subsection{Contributions} We summarize our contributions below. 
	
	\subsubsection{Minimax lower bounds of rate of convergence of the excess risk}
	
	Our first contribution in this paper is to establish minimax lower bounds of rate of convergence of the excess risk for any classifier under model (\ref{model_X}) and (\ref{model_YZ}). The excess risk is defined relative to $R_z^*$ in (\ref{eq_RxRz}) which we view 
	as 
	a more natural benchmark than $R_x^*$ because our proposed classifier is designed to 
		adapt to the underlying low-dimensional structure in (\ref{model_X}). The relation in (\ref{eq_RxRz}) suggests $R_z^*$ is also a more ambitious benchmark than $R_x^*$. 
	
	Since the gap between $R_x^*$ and $R_z^*$ quantifies the irreducible error for not observing $Z$, we start in Lemma \ref{lem_risk} of Section \ref{sec_benchmark} by
	characterizing how $R_x^*-R_z^*$  depends on $\xi^* = \lambda_K(A\szy A^\T)/\lambda_1(\sw)$, the signal-to-noise ratio for predicting $Z$ from $X$ (conditioned on $Y$), and $\Dt^2 = (\a_1-\a_0)^\T\szy^{-1}(\a_1-\a_0)$, the Mahalanobis distance between random vectors $Z \mid Y = 1$ and $Z \mid Y=0$. 
	Interestingly, it turns out that $R_x^*- R_z^*$ is small when either $\xi^*$ or $\Dt$ is large, a phenomenon that is different from the setting when $Y$  is linear in $Z$. Indeed, for the latter case, the excess risk of predicting $Y$ by using the best linear predictor of $X$ relative to the risk of predicting $Y$ from $\EE[Y|Z]$ is small only when $\xi^*$ is large \citep{bing2020prediction}. 
	
	In Theorem \ref{thm_lowerbound} of Section \ref{sec_lower_bound}, we derive the minimax lower bounds of the excess risk for any classifier with explicit dependency on the signal-to-noise ratio $\xi^*$, the separation distance $\Dt$, the dimensions $K$ and $p$ and the sample size $n$. Our results also fully capture the phase transition of the excess risk as the magnitude of $\Dt$ varies.  Specifically, when $\Dt$ is of constant order, the established lower bounds  are \[
	(\omega_n^*)^2  ~ = ~ {K\over n} + {\Dt^2 \over \xi^*} + {\Dt^2 \over \xi^*}{p\over \xi^* n}.
	\]
	The first term is  the optimal rate of the excess risk even when $Z$ were observable; the second term corresponds to the irreducible error of not observing $Z$ in $R_x^* - R_z^*$ and 
	the last term reflects the minimal price to pay for estimating the column space of $A$. When $\Dt \to \infty$ as $n\to\infty$, the lower bounds  become $(\omega_n^*)^2 \exp(-\Dt^2/8)$ and get exponentially faster in $\Dt^2$. When $\Dt \to 0$ as $n\to \infty$, the lower bounds  get slower as $\omega_n^* \min\{\omega_n^*/\Dt, 1\} $, implying a more difficult scenario for classification. In Section \ref{sec_optimality_LDA}, the lower bounds are further shown to be tight in the sense that the excess risk of the proposed PC-based classifiers have a matching upper bound, up to some  logarithmic factors.

	To the best of our knowledge, our minimax lower bounds are both new in the literature of factor models and the classical LDA. In the factor model literature, even in linear factor regression models, there is no known minimax lower bound of the prediction risk with respect to the quadratic loss function. In the LDA literature, our results cover the minimax lower bound of the excess risk in the classical LDA as a special case and are the first to fully characterize the phase transition in $\Dt$ (see Remark \ref{rem_comp_lowerbounds} for details). 
	The analysis of  establishing Theorem \ref{thm_lowerbound} is highly non-trivial and encounters several challenges. Specifically, since the excess risk is not a semi-distance, as required by the standard techniques of proving minimax lower bounds, the first challenge is to develop a reduction scheme based on a surrogate loss function that satisfies a local triangle inequality-type bound. The second challenge of our analysis is to allow a fully non-diagonal structure of $\Cov(X | Y)$ under model (\ref{model_X}), as opposed to the existing literature on the classical LDA that assumes $\Cov(X | Y)$ to be diagonal or even proportional to the identity matrix. To characterize the effect of estimating the column space of $A$ on the excess risk in deriving the third term of the lower bounds, our proof is based on constructing a suitable subset of the parameter space via the hypercube construction that is used for proving the optimal rates of the sparse PCA \citep{vu2013minimax} (see the paragraph after Theorem \ref{thm_lowerbound} for a full discussion). Since the statistical distance (such as the KL-divergence) between thus constructed hypotheses could diverge as $p/n\to\i$, this leads to the third challenge of providing a meaningful and sharp lower bound that is valid for both $p<n$ and $p>n$.
	
	\subsubsection{A general two-step classification approach and the PC-based classifier}
	
	Our second contribution in this paper is to propose a computationally efficient linear classifier in Section \ref{sec_PCR_method} that uses leading principal components (PCs) of the high-dimensional feature, with the number of retained PCs selected in a data-driven way. This PC-based classifier is one instance of a general two-step classification approach proposed in Section \ref{sec_general_method}. To be clear, it differs from naively applying standard  LDA, using plug-in estimates of the Bayes rule, on the leading PCs. 
	
	To motivate our approach, suppose that the factors $Z$ were observable. Then the optimal Bayes rule is to classify a new point $z\in \RR^K$ as
	\begin{equation}\label{rule_Z_0}
		g_z^*(z) = \1 \{z^\T \eta + \eta_0 \ge 0\}
	\end{equation}
	where
	\begin{equation}\label{rule_Z}
		\eta = \szy^{-1}(\a_1 - \a_0) ,\qquad \eta_0 = - {1\over 2}(\a _0 + \a _1 )^\T \eta + \log {\pi_1 \over \pi_0}.
	\end{equation}
	This rule is optimal in the sense that it has the smallest possible misclassification error. 
	Our approach  in Section \ref{sec_general_method} utilizes 
	an intimate connection	between the linear discriminant analysis and regression to reformulate the Bayes rule $g_z^*(z)$ as $\1\{z^\T\beta+\beta_0\ge 0\}$ with $\beta= \sz^{-1} \Cov(Z, Y)$ (and  $\beta_0$ is given in (\ref{eq_beta}) of Section \ref{sec_method}). The key difference is  the use of the {\em unconditional} covariance matrix $\sz$, as opposed to the {\em conditional} one $\szy$ in (\ref{rule_Z}). As a result, $\beta$  can be interpreted as the coefficient of regressing $Y$ on $Z$, suggesting to estimate  $z^\T \beta$ by 
	$z^\T (\Z^\T \Pi_n\Z)^+\Z^\T \Pi_n\Y$ via the method of least squares, again, in case   $\Z=(Z_1,\ldots,Z_n)^\T\in \RR^{n\times K}$ and $z\in\RR^K$ had been observed. Here 	$\Y=(Y_1,\ldots,Y_n)^\T\in \{0,1\}^n$, $\Pi_n = \bI_n - n^{-1}\b1_n\b1_n^\T$ is the centering projection matrix and $M^+$ denotes the Moore-Penrose inverse of any matrix $M$ throughout of this paper. 
	
	Since we only have access to $x\in\RR^p$, a realization of $X$,  $\X =  [ X_1 \cdots X_n ]^\T\in \RR^{n\times p}$, and $\Y\in\{0,1\}^n$,  it is natural to estimate the span of $z$ by   $B^\T x$ and to predict the span of  $\Pi_n \Z$ by $\Pi_n  \X B$, for some appropriate matrix $B$. This motivates us
	to estimate the inner-product $z^\T \beta$ by 
	\begin{align}
		(B^\T x)^\T (B^\T \X^\T  \Pi_n \X B)^+B^\T \X^\T \Pi_n \Y   ~ := ~  x^\T \wh\theta.
		\label{general B}  
	\end{align}
	By using a plug-in estimator $\wh \beta_0$ of $\beta_0$, the resulting rule $\wh g_x(x) = \1\{x^\T\wh\theta+\wh \beta_0\ge 0\}$ is a general two-step, regression-based classifier and the choice of $B$ is up to the practitioner.

	In this paper, we advocate the choice $B = \U_r\in\RR^{p\times r}$ where $\U_r$ contains the first $r$ right-singular vectors of $\Pi_n \X$, such that the projections $\Pi_n \X B$ become the first $r$ principal components of $\X$. Intuitively, this method has promise as 
	\cite{SW2002_JASA} proves that when $r$ is chosen as $K$, the projection $\Pi_n \X\U_K$ accurately predicts the span of $\Pi_n \Z$ under model (\ref{model_X}).  
	Since in practice $K$ is oftentimes unknown, we further use a data-driven selection of $K$ in Section \ref{sec_K} to construct our final PC-based classifier.
	The proposed procedure is computationally efficient. 
	Its only computational burden is that of computing the singular value decomposition (SVD) of $\X$.
	Guided by our theory, we also discuss  a cross-fitting strategy in Section \ref{sec_PCR_method} that improves the PC-based classifier by removing the dependence from using the data twice (one for constructing $\U_r$ and one for computing $\wh\theta$ in (\ref{general B})) when $p > n$ and the signal-to-noise ratio $\xi^*$ is weak.

	Retaining only a few principal components of the observed features and using them in subsequent regressions is known as principal component regression (PCR) \citep{SW2002_JASA}. It is a popular method for predicting $Y\in\RR$ from a high-dimensional feature vector $X\in\RR^p$ when both $X$ and $Y$ are generated via a low-dimensional latent factor $Z$. Most of the existing literature analyzes the performance of PCR when both $Y$ and $X$ are linear in $Z$, for instance, \cite{Bai-Ng-forecast, Bair_JASA, bing2020prediction, Partial_Factor_Modeling, SW2002_JASA,SW2002_JB}, just to name a few. When $Y$ is not linear in $Z$, little is known. An exception is  \cite{fan2017}, which studies the model $Y = h(\xi_1 Z, \cdots, \xi_q Z; \eps)$ and $X = AZ+W$ for some unknown general link function $h(\cdot)$. Their focus is only on estimation of $\xi_1, \ldots, \xi_q$, the sufficient predictive indices of $Y$, rather than analysis of the risk of predicting $Y$. As  $\EE[Y|Z]$ is not linear in $Z$ under our model (\ref{model_X}) and  (\ref{model_YZ}),  to the best of our knowledge,   analysis of the misclassifcation error under model  (\ref{model_X})  and  (\ref{model_YZ})  for a general linear classifier has not been studied elsewhere.

	\subsubsection{A general strategy of analyzing the  excess risk of $\wh g_x$ based on any matrix $B$}
	
	Our third contribution in this paper is  to provide a general theory for analyzing the  excess risk of the type of classifiers $\wh g_x$ that uses a generic matrix $B$ in (\ref{general B}). 	In Section \ref{sec_theory_general} we state our result in Theorem \ref{thm_general_risk}, a general bound for the excess risk of the classifier $\wh g_x$ based on a generic matrix $B$. 
	It depends on (i) how well we estimate $z^\T \beta + \beta_0$ and (ii) a margin condition on the conditional distributions $Z\mid Y=k$, $k\in \{0,1\}$, nearby the hyperplane $\{z\mid z^\T \beta + \beta_0=0\}$. This  is a different 
	approach than the usual one in the literature \cite{DGL} that provides  bounds on the excess risk $\PP\{\wh g(X)\ne Y \mid \D \}- R_z^*$ of a classifier $\wh g:\RR^p\to\{0,1\}$ by the expression $2\EE[| \eta(Z)- 1/2|\1\{ \wh g(X) \ne g_z^*(Z)\}\mid \D~]$, with
	$\eta(z) = \PP(Y = 1|Z=z)$,
	and
	involves analyzing the behavior of $\eta(Z) $ near $1/2$ (see our detailed discussion in Remark \ref{rem_margin}).
	The analysis of Theorem \ref{thm_general_risk} is powerful in that it can easily be generalized to any distribution of $Z\mid Y$, as explained in Remark \ref{rem_extension}.  
Our second main result in  Theorem \ref{thm_general_risk_explicit} of Section \ref{sec_theory_general} provides explicit rates of convergence of the excess risk of $\wh g_x$ for a generic $B$ and clearly delineates three key quantities that need to be controlled as introduced therein. The established rates of convergence reveal the same phase transition in $\Dt$ from the lower bounds. It is worth mentioning that the analysis of Theorem \ref{thm_general_risk_explicit} is more challenging under model (\ref{model_X}) and (\ref{model_YZ}) than the classical LDA setting (\ref{model_Gaussian}) in which the excess risk of any linear classifier in $X$ has a closed-form expression.

\subsubsection{Optimal rates of convergence of the PC-based classifier}

Our fourth contribution is to apply
the general theory in Section \ref{sec_theory_general}
to analyze the PC-based classifiers. Consistency of our proposed estimator of $K$ is established in Theorem \ref{thm_PCR_K_hat} of Section \ref{sec_PCLDA_K_hat}.  In Theorem \ref{thm_PCR} of Section \ref{sec_PCLDA_K}, we
derive  explicit  rates of convergence of the excess risk of the PC-based classifier that uses $B = \U_K$. 
The obtained rate of convergence exhibits an interesting interplay between the sample size $n$ and the dimensions $K$ and $p$ through the quantities $K/n$,  $\xi^*$ and $\Dt$. Our analysis also covers the low signal setting  $\Dt = o(1)$, a regime that has not been analyzed even in the existing literature of  classical LDA. Our theoretical results are valid for both fixed and growing $K$ and are also valid even when $p$ is much lager than $n$.
In Theorem \ref{thm_PCR_indep} of Section \ref{sec_PCLDA_K}, we also show that a PC-based LDA that uses either auxiliary data or sample splitting could surprisingly yield faster rates of convergence of the excess risk by removing the dependence between $\U_K$ and $\X$. These faster rates are further shown to be minimax optimal, up to a logarithmic factor, in Corollary \ref{cor_PCR} of Section \ref{sec_optimality_LDA}. The benefit of using auxiliary data or sample splitting has also been recognized in other problems, such as the problem of estimating the optimal instrument in sparse high-dimensional instrumental variable model \citep{belloni2012sparse} and the problem of inference on a low-dimensional parameter in the presence of high-dimensional nuisance parameters \citep{DDML}.

\subsubsection{Extension to multi-class classification}

Our fifth contribution is to extend the general two-step classification procedure in Section \ref{sec_method} to handle multi-class classification problems in Section \ref{sec_multi_level}. Rates of convergence of the excess risk of the proposed multi-class classifier 
are derived in Theorem \ref{thm_risk_multi}. PC-based classifiers are analyzed subsequently in Corollary \ref{cor_risk_multi}. Our theory is the first to explicitly characterize dependence of the excess risk on the number of classes, and to cover the weak separation case when $\Dt \to 0$.\\

The paper is organized as follows. 
In Section \ref{sec_benchmark}, we provide 
an oracle benchmark that quantifies the excess risk of the optimal classifier based on $X$. We state the minimax lower bounds of the excess risk for any classifier in Section \ref{sec_lower_bound}. 
In Section \ref{sec_method},
we present a connection between the linear discriminant classifier by using $Z$ and regression of $Y$ onto $Z$. This key observation leads to our proposed PC-based classifier. Furthermore, we propose
a data-driven selection of the number of retained principal components. 
A general theory is stated in Section \ref{sec_theory_general} for analyzing the excess risk of the classifier $\wh g_x$ that uses any $B$ for the estimate $\wh \theta$ in (\ref{general B}). In Section \ref{sec_theory_application} we apply the general result to analyze the PC-based classifiers. 
Main simulation results are presented in Section \ref{sec_sim} and a  real data analysis is given  in Section \ref{sec_real_data}. Extension to multi-class classification is studied in Section \ref{sec_multi_level}. All the proofs and additional simulation results are deferred to the Appendix.\\

\noindent{\bf Notation:} 
We use the common notation  $\varphi(x)=\exp(-x^2/2) / \sqrt{2\pi}$ for the standard normal density, and denote by $\Phi(x)=\int  \varphi(t)\1\{ t\le x\} \, {\rm d} t$  its c.d.f.. For any positive integer $d$, we write $[d] := \{1,\ldots, d\}$.
For any vector $v$, we use $\|v\|_q$ to denote its $\ell_q$ norm for $0\le q\le \i$.  We also write $\|v\|_Q^2 = v^\T Q^{-1}v$ for any commensurate, invertible square matrix $Q$. For any real-valued matrix $M\in \RR^{r\times q}$, we use $M^+$ to denote the Moore-Penrose inverse of $M$, and $\sigma_1(M)\ge \sigma_2(M)\ge \cdots \ge \sigma_{\min(r,q)}(M)$ to denote the singular values of $M$ in non-increasing order. We define the operator norm $\|M\|_{\op}=\sigma_1(M)$. 
For a symmetric positive semi-definite matrix $Q\in \RR^{p\times p}$, we use $\lambda_1(Q)\ge \lambda_2(Q)\ge \cdots \ge \lambda_p(Q)$ to denote the eigenvalues of $Q$ in non-increasing order.  We write $Q \succ 0$ if $Q$ is strictly positive definite. 
For any two sequences $a_n$ and $b_n$, we write $a_n\lesssim b_n$ if there exists some constant $C$ such that $a_n \le Cb_n$. The notation $a_n\asymp b_n$ stands for $a_n \lesssim b_n$ and $b_n \lesssim a_n$. For two numbers $a$ and $b$, we write $a\wedge b = \min\{a, b\}$ and $a\vee b =\max\{a,b\}$. 
We use $\bI_d$ to denote the $d\times d$ identity matrix and use $\b1_d$ ($\0_d$) to denote the vector with all ones (zeroes). For $d_1\ge d_2$, we use $\cO_{d_1\times d_2}$ to denote the set of all $d_1\times d_2$ matrices with orthonormal columns.
Lastly, we use $c,c',C,C'$ to denote positive and finite absolute constants that unless otherwise indicated can change from line to line.

\section{Excess risk and its minimax optimal rates of convergence}\label{sec_excess_risk_lower_bound}

We start in Section \ref{sec_benchmark} by introducing the oracle benchmark relative to which the excess risk is defined. Minimax optimal rates of convergence of the excess risk are derived in Section \ref{sec_lower_bound}.

\subsection{Oracle benchmark}\label{sec_benchmark}

Since our goal is to predict the Bayes rule $\1\{z^\T \eta +\eta_0\ge 0\}$ under model (\ref{model_YZ}), it is natural to choose the oracle risk $R_z^*$ in (\ref{eq_RxRz}) as our benchmark, as opposed to $R_x^*$. Furthermore, we always have the explicit expression
\begin{equation}\label{eqn_Bayes_risk}
	R_z^* =  
	1 - \pi_1\Phi \left({\Dt\over 2} + {\log {\pi_1 \over \pi_0} \over \Dt}\right) - \pi_0 \Phi \left({\Dt\over 2} - {\log {\pi_1 \over \pi_0} \over \Dt}\right),
\end{equation}
see, for instance, \cite[Section 8.3, pp 241--244]{Izenman-book}. 
Here,
\begin{equation}\label{def_Dt}
	\Dt^2 := (\a_0-\a_1)^\T\C^{-1}(\a_0-\a_1)
\end{equation}
is the Mahalanobis distance between the conditional distributions $Z \mid Y=1\sim N_K(\a_1, \szy)$ and $Z \mid Y=0\sim N_K(\a_0, \C)$. In particular, when $\pi_0 = \pi_1$, the expression in (\ref{eqn_Bayes_risk})  simplifies to
$
R_z^* = 1 - \Phi\left(\Dt/2 \right).
$

\begin{remark}
	\label{rem_cases}    
	It is immediate from (\ref{eqn_Bayes_risk}) that $\Dt\to\i$ implies  $R_z^*\to0$. The case of zero Bayes error $R_z^*$ represents the easiest classification problem and 
	we can expect fast rates of the excess risk.
	If $\Dt\to0$,  the Bayes risk $R_z^*$ converges to $\min\{\pi_0, \pi_1\}$. 
	When $\pi_0=\pi_1=  1/2$, the limit reduces to random guessing, which represents  the hardest classification problem and slow rates are to be expected. When $\pi_0\ne \pi_1$, we can expect fast rates, too, since the asymptotic  Bayes rule always votes for the same label, to wit, the one  with the largest unconditional probability.
	Thus,  in a way, $\Dt\asymp1$ is the most interesting case to investigate.
\end{remark}

The lemma  below shows that $R_x^* \ge R_z^*$, implying that $R_z^*$ is also an ambitious benchmark.

\begin{lemma}\label{lem:RxvsRz}
	Under model (\ref{model_X}) and {\rm (i) -- (iii)}, we have
	$$  R_x^*   = \inf_{g:\ \RR^p\to\{0,1\}} \PP\{ g(AZ+W) \ne Y \} ~ \ge ~  R_z^* =\inf_{h:\ \RR^K\to\{0,1\}} \PP\{ h(Z) \ne Y \} .$$
\end{lemma} 
\begin{proof}
	See Appendix \ref{app_proof_RxvsRz}.
\end{proof}

If $W=\0_p$, the inequality in Lemma \ref{lem:RxvsRz} obviously becomes an equality. More generally, if
the signal for predicting $Z$ from $X$ under model (\ref{model_X}) 
is large, we expect  the gap between $R_x^*$ and $R_z^*$ to be small. To characterize such dependence, we introduce the following parameter space of $\theta := (A, \szy, \sw, \a_1, \a_0, \pi_1, \pi_0)$, 
\begin{align}\label{def_Theta}
	\Theta(\lambda, \sigma, \Dt) &= \left\{\theta:  \lambda_j(\sw)  \asymp \sigma^2, \forall j\in [p],~  \lambda_k(A\szy A^\T) \asymp \lambda,\forall k\in [K],~  \pi_0 = \pi_1\right\}
\end{align}
and  recall $\Dt$ from  (\ref{def_Dt}).
For any $\theta \in \Theta(\lambda,\sigma,\Dt)$, the quantity $\lambda/\sigma^2$ can be treated as the signal-to-noise ratio for predicting $Z$ from $X$ given $Y$ under model (\ref{model_X}). The following lemma shows how the gap between $R_x^*$ and $R_z^*$ depends on $\lambda/\sigma^2$ and $\Dt$ in the special  case   $W\sim N_p(\0_p,\sw)$.

\begin{lemma}\label{lem_risk}
	Under model (\ref{model_X}) and {\rm (i) -- (iv)}, suppose $W\sim N_p(\0_p,\sw)$ with $\sw \succ 0 $. For any $\theta \in \Theta(\lambda,\sigma,\Dt)$, we have
	\[
	{\Dt \over 1+(\lambda/\sigma^2)}\exp\left\{-{\Dt^2 \over 8}\right\} \lesssim  ~ R_x^* - R_z^* ~ \lesssim  {\Dt \over 1+(\lambda/\sigma^2)}\exp\left\{-{\Dt^2 \over 8} + {\Dt^2 \over 8(1+\lambda/\sigma^2)}\right\}.
	\]
\end{lemma}
\begin{proof}
	See Appendix \ref{app_proof_benchmark}.
\end{proof}

\begin{remark}\label{rem_risk}
	The upper bound of 
	Lemma \ref{lem_risk} reveals that $\lambda/\sigma^2 \to\infty$ implies $R_x^*-R_z^* \to0$ irrespective of the magnitude of $\Dt$. 
	Regarding to $\Dt$, we also find that  $R_x^*-R_z^* \to 0$ in the following scenarios: (1) if $\Dt\to0$, irrespective of  $\lambda/\sigma^2$,  (2) if $\Dt\to\infty$  and  $\lambda/\sigma^2\not\to0$, (3) if  $\Dt\asymp1 $ and $\lambda/\sigma^2\to \infty$.

	The lower bound of Lemma \ref{lem_risk}, on the other hand, establishes the irreducible error for not observing $Z$. This term will naturally appear in the minimax lower bounds of the excess risk derived in the next section.

\end{remark}

\subsection{Minimax lower bounds of the excess risk}\label{sec_lower_bound}

In this section, we establish minimax lower bounds of the excess risk $R_x(\wh g) - R_z^*$ under model (\ref{model_X})  and  (\ref{model_YZ}) for any classifier $\wh g$. Here,
\begin{align}
	R_x(\wh g) &:= \PP\left\{ \wh g(X) \ne Y  \mid \D \right\}
\end{align}
is the (conditional) misclassification error, given the training data 
$$\D:= (\X, \Y)=\left\{ (X_1,Y_1),\ldots(X_n,Y_n)\right\}.$$
The results are established over the parameter space $\Theta(\lambda, \sigma, \Dt)$ in (\ref{def_Theta}) which is characterized by three quantities: $\lambda$, $\sigma^2$ and $\Dt$, all of which are allowed to grow with the sample size $n$. Our minimax lower bounds of the excess risk fully characterize the dependence on these quantities, in addition to the dimensions $K$ and $p$ and the sample size $n$.

We use $\PP^{\D}_{\theta}$ to denote the set of all distributions of $\D$ parametrized by $\theta\in \Theta(\lambda, \sigma, \Dt)$ under model (\ref{model_X}) and (\ref{model_YZ}). For simplicity, we drop the dependence on $\theta$ for both $R_x(\wh g)$ and $R_z^*$.
Define 
\begin{equation}\label{def_omega_star}
	\omega_n^* = \sqrt{{K\over n}+ {\sigma^2\over \lambda} \Dt^2 + {\sigma^2\over \lambda} {\sigma^2 p \over \lambda n}\Dt^2}.
\end{equation}
The following theorem states the minimax lower bounds of the excess risk for any classifier over the parameter space $\Theta(\lambda, \sigma, \Dt)$. 

\begin{thm}\label{thm_lowerbound}
	Under model (\ref{model_X}), assume {\rm (i) -- (iv)}, $K\ge 2$, 
	$K /(n \wedge p) \le c_1 $, $\sigma^2/\lambda \le c_2$ and $\sigma^2 p/(\lambda n) \le c_3$ for some sufficiently small constants $c_1, c_2, c_3>0$.  There exists some constants $c_0\in (0, 1)$ and $C>0$ such that
	\begin{enumerate}
		\item If $\Dt \asymp 1$, then 
		\[
		\inf_{\wh g}\sup_{\theta \in \Theta(\lambda, \sigma, \Dt) } \PP^{\D}_{\theta}\left\{
		R_x(\wh g) - R^*_z \ge C \left(\omega_n^* \right)^2
		\right\}\ge c_0.
		\]
		\item If $\Dt \to \i$ and $\sigma^2/\lambda = o(1)$ as $n\to \i$, then
		\[
		\inf_{\wh g}\sup_{\theta \in \Theta(\lambda,\sigma, \Dt)} \PP^{\D}_{\theta}\left\{
		R_x(\wh g) - R^*_z \ge C \left(\omega_n^* \right)^2 \exp\left\{- \left[{1\over 8} + o(1)\right]\Dt^2\right\}
		\right\}\ge c_0.
		\]
				\item If $\Dt \to 0$ as $n\to \i$, then 
		\[
		\inf_{\wh g}\sup_{\theta \in \Theta(\lambda,\sigma, \Dt)} \PP^{\D}_{\theta}\left\{
		R_x(\wh g) - R^*_z \ge C \min\left\{{\omega_n^*\over \Dt},1\right\} \omega_n^*
		\right\}\ge c_0.
		\]
	\end{enumerate} 
	The infima in all statements are taken over all classifiers. 
\end{thm}
\begin{proof}	
	The proof of Theorem \ref{thm_lowerbound} is deferred to Appendix \ref{app_proof_thm_lowerbound}.
\end{proof}

The lower bounds in Theorem \ref{thm_lowerbound} consist of three terms: the one related with $K/n$ is  the optimal rate of the excess risk
even when $Z$ were observable; the second one related with $\sigma^2 / \lambda$ is the irreducible error for not observing $Z$ (see, Lemma \ref{lem:RxvsRz}); the last one involving $\sigma^2 p / (\lambda n)$ is the price to pay for estimating the column space of $A$.
Although the third term could get absorbed by the second term as $\sigma^2p/(\lambda n) \le c_3$, we incorporate it here for transparent interpretation. The lower bounds in Theorem \ref{thm_lowerbound} are tight as we show in Section \ref{sec_optimality_LDA} that there exists a classifier whose excess risk has a matching upper bound.

\begin{remark}[Phase transition in $\Dt$]\label{rem_Dt}
	Recall from (\ref{def_Dt}) that $\Dt$ quantifies the separation between $N(\a_0, \szy)$ and $N(\a_1, \szy)$.
	We see in Theorem \ref{thm_lowerbound} a phase transition of the rates of convergence of the excess risk as $\Dt$ varies.  When $\Dt$ is of constant order, the  excess risk has minimax convergence rate 
	\[
	{K\over n}+ {\sigma^2\over \lambda} + {\sigma^2\over \lambda} {\sigma^2 p \over \lambda n}.
	\]
	When $\Dt\to\infty$, we see that the minimax rate of convergence of the excess risk gets faster exponentially in $\Dt^2$. For instance, if $\Dt^2 \ge  C_0\log n$ for some constant $C_0>0$, then the minimax rate already becomes {\em polynomially faster in $n$} as
	\[
	\left[{K\over n}+ {\sigma^2\over \lambda} + {\sigma^2\over \lambda} {\sigma^2 p \over \lambda n}\right]{1\over n^{C_1}}
	\]
	for some $C_1>0$ depending on $C_0$. 
	 The condition $\sigma^2/\lambda=o(1)$ for $\Dt\to \i$ can be removed, and the lower bound remains the same except the factor $(1/8)$ gets replaced by $(1/8)(1/(1+\lambda/\sigma^2))$.
	Finally, when $\Dt \to 0$, a more challenging, yet important case, the minimax convergence rate of the excess risk gets slower. It is worth noting that although the oracle Bayes risk $R_z^* \to 1/2$ when $\Dt\to 0$, the minimax excess risk still converges to zero at least in $\omega_n^*$-rate.
	If $\omega_n^* \lesssim \Dt$, the convergence gets faster as
	\[
	{K\over n}{1\over \Dt}+ {\sigma^2\over \lambda}\Dt + {\sigma^2\over \lambda} {\sigma^2 p \over \lambda n}\Dt.
	\]
\end{remark}

\begin{remark}[Proof technique]
	To prove Theorem \ref{thm_lowerbound}, the three terms in the lower bound are derived separately in the setting where $X\mid Y$ is Gaussian. Since, for any  classifier $\wh g$, 
	\[
	R_x(\wh g) - R_z^* = \left(R_x(\wh g) - R_x^*\right) + \left(R_x^* - R_z^*\right),
	\]
	in view of Lemma \ref{lem:RxvsRz}, it suffices to prove
	the two terms related with $K/n$ and $\sigma^2 p  / (\lambda n)$ constitute the lower bounds of $R_x(\wh g) - R_x^*$. In fact, as a byproduct of our result, we also derive  minimax lower bounds of the excess risk relative to $R_x^*$. This derivation is based on constructing subsets of $\Theta(\lambda, \sigma, \Dt)$ by fixing either $A$ or $\a_0$ and $\a_1$ separately. The choice of $A$ is based on the hypercube construction for matrices with orthonormal columns \citep[Lemma A.5]{vu2013minimax}.  
	The analyses of both terms are non-standard as the 
	excess risk is not a semi-distance, as required by standard techniques of proving minimax lower bounds. Based on a reduction scheme established in Appendix \ref{app_proof_thm_lowerbound}, we show that proving Theorem \ref{thm_lowerbound} suffices to establish a minimax lower bound of the following loss function 
	\[
	L_\theta(\wh g) := \PP_\theta \left\{
	\wh g(X) \ne g_\theta^*(X) \mid \D  
	\right\}.
	\]
	Here $\PP_{\theta}$ is taken with respect to $X$ and $g_\theta^*(X)$ is the Bayes rule based on  $X$ that minimizes $R_x(g)$ over $g:\RR^p\to \{0,1\}$. Since $L_\theta(\wh g)$ is shown to satisfy a local triangle inequality-type bound such that a variant of Fano's lemma can be applied \citep[Proposition 2]{Martin_2013}, we proved a crucial result, in Lemmas \ref{lem_lb_L} and \ref{lem_lb_L_prime} of Appendix \ref{app_proof_thm_lowerbound}, that
	\begin{equation}\label{display_lb_L}
		\inf_{\wh g}\sup_{\theta \in \Theta(\lambda,\sigma,\Dt)} \PP^{\D}_{\theta}\left\{
		L_\theta(\wh g) \ge C
		\left(\sqrt{K\over n}  {1\over \Dt}+ \sqrt{{\sigma^2 \over \lambda}{\sigma^2 p \over \lambda n}}
		\right) e^{-{\Dt^2 \over 8}}
		\right\}\ge c_0
	\end{equation}
	for some constant $c_0\in (0,1)$ and $C>0$.
\end{remark}

\begin{remark}[Comparison with the existing literature]\label{rem_comp_lowerbounds}
	As mentioned above, a byproduct of our proof of Theorem \ref{thm_lowerbound} is the minimax lower bounds of $R_x(\wh g) - R_x^*$ in the setting where $X\mid Y$ is Gaussian, which have exactly the same form as Theorem \ref{thm_lowerbound} but without the second term related with $\sigma^2/\lambda$.
		It is informative to put this lower bound of $R_x(\wh g) - R_x^*$ in comparison to the existing literature in this special setting. 
	
	Under the classical LDA model (\ref{model_Gaussian}), \cite{cai2019convex} derives the minimax lower bounds of $R_x(\wh g) - R_x^*$ over a suitable parameter space for $\Dt \gtrsim 1$, which have the same form as ours  with 
	$K/n + \sigma^4 p \Dt^2 / (\lambda^2 n)$
	replaced by $s / n$ for $s := \|\Sigma^{-1}(\mu_1 - \mu_0)\|_0$. 
	In contrast, our lower bounds reflect the benefit of considering an approximate lower-dimensional structure of $X \mid Y$ under (\ref{model_X})  and (\ref{model_Gaussian}) instead of directly assuming sparsity on $\Sigma^{-1}(\mu_1 - \mu_0)$.  These two lower bounds coincide in the low-dimensional setting $(p < n)$ when there is no sparsity in $\Sigma^{-1}(\mu_1 - \mu_0)$, that  is $s = p$, and
	when there is no low-dimensional hidden factor model (that is, $X=Z$ with $K=p$, $A=\bI_p$ and $W=\0_p$). 
	On the other hand, \cite{caizhang2019} only established the phase transition between $\Dt\asymp 1$ and $\Dt \to \i$ whereas we are able to derive the minimax lower bound for $\Dt \to 0$,  a case that has not even been analyzed in the classical LDA literature.

Technically, it is also worth mentioning that the latent model structure on $X$ via (\ref{model_X}) brings considerable additional difficulties for establishing the lower bounds of $R_x(\wh g) - R_x^*$. Indeed, for any $\theta\in \Theta(\lambda, \sigma, \Dt)$, the covariance matrix of $X \mid Y$ is $\Sigma(\theta) =  A\szy A^\T + \sw$ which cannot be chosen as a diagonal matrix to simplify the analysis as done by \cite{cai2019convex}. Furthermore, to derive the term $\sigma^4  p \Dt^2 /(\lambda^2n)$ in the lower bound for quantifying the error of estimating the column space of $A$, we need to carefully choose the subset of  $\Theta(\lambda, \sigma, \Dt)$ via the hypercube construction \citep[Lemma A.5]{vu2013minimax} that has been used for proving the optimal rates of the sparse PCA. Since the statistical distance (such as KL-divergence) between any two of thus constructed hypotheses of $\Theta(\lambda, \sigma, \Dt)$  is diverging whenever $p/n
\to \i$ (see, Lemma \ref{lem_KL} in Appendix \ref{app_proof_thm_lowerbound}), a different analysis than the standard one (for instance, in \cite{Martin_2013}) has to be used to allow $p > n$ and a  large amount of work is devoted to provide a meaningful and sharp lower bound that is valid for both $p<n$ and $p>n$ (see Lemma \ref{lem_lb_L} for details).
\end{remark}

\section{Methodology}\label{sec_method}

In this section, we describe our classification method based on $n$ i.i.d. observations from model (\ref{model_X}) and (\ref{model_YZ}). 
We first state a general method in Section \ref{sec_general_method} which is motivated by the optimal oracle rule $g_z^*$ in (\ref{rule_Z_0}) and (\ref{rule_Z}), and is based on prediction of the unobserved factors $Z_1,\ldots,Z_n,Z$ in the features $X_1,\ldots,X_n,X$ by projections. In Section \ref{sec_PCR_method} we state our proposed methods via principal component projections as well as a cross-fitting strategy for high-dimensional scenarios. Selection of the number of principal components is further discussed in Section \ref{sec_K}.

\subsection{General approach}\label{sec_general_method}

The first idea is to change the classification problem into a regression problem, at the population level. The close relationship between LDA and regression has been observed before, see, for instance, Section 8.3.3 in \cite{Izenman-book},
\cite{hastie_09_esl} and \cite{mai2012}.
Let   $\sz = \Cov(Z)$ be the unconditional covariance matrix of $Z$. Define
\begin{align}\label{eq_beta}
\beta &= \pi_0\pi_1  \sz^{-1}(\a_1 - \a_0),\\
\beta_0 &= -{1\over 2}(\a_0+\a_1)^\T \beta + \pi_0\pi_1\left[
1 - (\a_1-\a_0)^\T \beta
\right]\log{\pi_1 \over \pi_0}.\nonumber
\end{align} 
\begin{prop}\label{prop_ls_rule}
Let $\eta,\eta_0$ and $\beta,\beta_0$ be defined in (\ref{rule_Z}) and (\ref{eq_beta}), respectively. Under model (\ref{model_YZ}) and assumption {\rm (iv)}, we have
\[
z^\T \eta+\eta_0 \ge 0 \quad \iff\quad z^\T \beta +\beta_0\ge0.
\]
Furthermore, 
\[
\beta = \sz^{-1}\Cov(Z,Y).
\]
\end{prop}
\begin{proof}
The proof of Proposition \ref{prop_ls_rule} can be found in Appendix \ref{app_proof_sec_method}. 
\end{proof}

\begin{remark}
In fact, our proof  shows that the first statement of Proposition \ref{prop_ls_rule} still holds if we replace $\pi_0\pi_1$ in the definition of $\beta$ by any positive value coupled with corresponding modification of $\beta_0$ (see Lemma \ref{lem_rule} in Appendix \ref{app_proof_sec_method} for the precise statement). The advantage of using $\pi_0\pi_1$ in (\ref{eq_beta}) is that $\beta$ can be obtained by simply regressing $Y$ on $Z$. For this choice of $\beta$, our proof also reveals 	
\begin{equation}\label{eq_Gz_Gzstar}
	z^\T \eta+\eta_0 = {1\over \pi_0\pi_1[1 - (\a_1 - \a_0)^\T \beta]} \left( z^\T \beta +\beta_0\right) =  {1 + \pi_0\pi_1 \Dt^2 \over \pi_0\pi_1} \left( z^\T \beta +\beta_0\right),
\end{equation}
a key identity that will used later in Section \ref{sec_multi_level} to extend our approach for handling multi-class classification problems.
\end{remark}

Proposition \ref{prop_ls_rule} implies the equivalence between the linear  rules 
$
g_z^*(z) 
$
in (\ref{rule_Z})
and
\begin{align}\label{LDAa}
g_z(z)
&:= \1 \{ z^\T \beta+ \beta_0 \ge 0 \}
\end{align}
based on, respectively, the halfspaces $\{z\mid z^\T \eta +\eta_0\ge 0\}$ and  $\{z\mid  z^\T \beta + \beta_0\ge 0\}$. 
According to Proposition \ref{prop_ls_rule},
if $\Z = (Z_1^\T, \ldots, Z_n^\T)^\T \in \RR^{n\times K}$  were observed,
it is natural to use 
the least squares estimator $(\Z^\T \Pi_n \Z)^{+}\Z^\T  \Pi_n \Y$ to estimate $\beta$.  Recall that $\Pi_n = \bI_n  - n^{-1}\b1_n\b1_n^\T$ is the centering matrix and $M^+$ is the Moore-Penrose inverse of any matrix $M$.
Since in practice  only $\X = (X_1^\T, \ldots, X_n^\T)^\T \in \RR^{n\times p}$ is observed, we propose to estimate $z^\T \beta$ by
\begin{equation}\label{def_theta_hat}
x^\T \wh\theta ~ := ~ x^\T B (\Pi_n \X B)^+\Y  = x^\T B(B^\T \X^\T \Pi_n X B)^+ B^\T \X^\T \Pi_n \Y
\end{equation}
with $x\in \RR^p$ being one realization of $X$ from model (\ref{model_X}). Here in principal $B\in\RR^{p\times q}$ could be any matrix  with any $q\in \{1,\ldots, p\}$.  Furthermore, we  estimate $\beta_0$
by
\begin{equation}\label{def_theta_0_hat}
\wh\beta_0  := -{1\over 2}(\wh \mu_0+\wh \mu_1)^\T \wh\theta+  \wh \pi_0\wh \pi_1 \left[
1   -(\wh \mu_1-\wh \mu_0 )^\T\wh\theta \ 
\right]\log{\wh \pi_1 \over \wh \pi_0}
\end{equation}
based on	standard non-parametric estimates
\begin{align}\label{def_pi_hat}
n_k = \sum_{i=1}^n \1{\{Y_i = k\}},\quad \wh\pi_k = {n_k \over n},\quad 
\wh \mu_k = {1\over n_k}\sum_{i=1}^n X_i\1\{ Y_i = k\}, 
\quad k\in \{0,1\}. 
\end{align}
Our final classifier is
\begin{equation}\label{eq_g_hat}
\wh g_x(x) :=  \1\{ x^\T \wh \theta + \wh \beta_0\ge 0\}.
\end{equation}
Notice that $\wh\theta$, $\wh \beta_0$ and $\wh g_x(x)$ all depend on $B$ implicitly. 

\subsection{Principal component (PC) based classifiers}\label{sec_PCR_method}

Though the classifier in (\ref{eq_g_hat}) can use any matrix $B$, in this paper we mainly consider the choice $B=\U_r\in \RR^{p\times r}$, for some $r\in \{1,\ldots, p\}$, where the matrix $\U_r$ consists of  the first $r$ right-singular vectors of $\Pi_n\X$, the centered $\X$. In this case, $x^\T \wh\theta$ is the famous principal component regression (PCR) predictor by using $r$ principal components \citep{Hotelling}. The optimal choice of $r$ would be $K$, the number of latent factors, when it is known in advance. 
We analyze the classifier with $B = \U_K$ in  Theorem \ref{thm_PCR} of Section \ref{sec_PCLDA_K}.

Suggested by our theory, in the high-dimensional setting $p>n$,  performance of the PC-based classifiers can be improved  either by using an additional dataset or via data-splitting.  

In several applications, such as semi-supervised learning, researchers also have access to an additional set of unlabelled data. Given an additional data matrix $\tX\in \RR^{n'\times p}$ with i.i.d. (unlabelled) observations from model (\ref{model_X}) with $n'\asymp n$ and independent of $\X$ in  (\ref{def_theta_hat}), it is often beneficial to use $B=\wt\U_K$ based on  the first $K$ right singular vectors of $\Pi_{n'}\tX$.  This classifier is analyzed in Theorem \ref{thm_PCR_indep} of Section \ref{sec_PCLDA_K}.

When additional data is not available, we advocate to use a sample splitting technique called $k$-fold cross-fitting \citep{DDML}. First, we randomly split the data into $k$ folds, and for each fold, we use it as $\tX$ to construct $\wt \U_r$ and use the remaining data as $\X$ and $\Y$ to obtain $\wh\theta$ and $\wh\beta_0$ from (\ref{def_theta_hat}) and (\ref{def_theta_0_hat}), respectively. In the end, the final classifier is constructed via (\ref{eq_g_hat}) based on the averaged $k$ pairs of $\wh\theta$ and $\wh\beta_0$. Theoretically, it is straightforward to show that the resulting classifiers share the same conclusions as Theorem \ref{thm_PCR_indep} for $k = \cO(1)$. Empirically, since this cross-fitting strategy ultimately uses all data points, it might mitigate the efficiency loss due to sample splitting. Standard choices of $k$ include $k = 2$ and $k = 5$ while the latter is reported to have smaller standard errors \citep{DDML}.

\subsection{Estimation of the number of retained PCs}\label{sec_K}
When $K$ is unknown, we propose to estimate it by 
\begin{equation}\label{def_K_hat}
\wh K := 
\argmin_{k\in\{0,1,\ldots,\bar K\}} {\sum_{j>k} \sigma_j^2  \over np - c_0 (n+p)  k}
,\qquad \textrm{with }\quad \bar K:= \left\lfloor {\nu \over 2c_0 ( 1+\nu)} (n \wedge p)  \right\rfloor ,
\end{equation}
for absolute constants $c_0$ and   $\nu >1$.  The latter is  introduced to avoid division by zero and can be set arbitrarily large.  The choice of $c_0 = 2.1$  is used in all of our simulations and has overall good performance.
The  sum  $\sum_{j}\sigma_j u_j v_j^\T$, with non-increasing $\sigma_j$, is the singular-value-decomposition (SVD) of $\Pi_n\X$ or $\Pi_n \tilde \X$.

Criterion (\ref{def_K_hat}) was originally proposed in \cite{rank19} for selecting  the rank of the coefficient of a multivariate response regression model and is further adopted by \cite{bing2020prediction} for selecting the number of retained principal components under the framework of factor regression models. It also has close connection to the well-known elbow method, but is more practical in terms of parameter tuning. 
The main computation of solving (\ref{def_K_hat}) is to compute the SVD of $\Pi_n\X$ once. In Section \ref{sec_PCLDA_K_hat} we show the consistentcy of $\wh K$, ensuring that the classifier with $B = \U_{\wh K}$ shares the same theoretical properties as the one with $B = \U_K$.

\section{A general 
strategy of bounding the excess classification error}

\label{sec_theory_general}

In this section, we establish a general theory for analyzing the excess risk of the classifier $\wh g_x$ in (\ref{eq_g_hat}) that uses {\em any} matrix $B$ for the estimate $\wh\theta$ in (\ref{def_theta_hat}). 
The main purpose is to establish high-level conditions that yield a consistent classifier constructed in Section \ref{sec_method} 
in the sense  
\[ 
R_x(\wh g_x) := \PP\{ \wh g_x(X)\ne Y   \mid \D \} \to R_z^*,\quad \text{in probability, as $n\to\i$}
\]
and further to  provide its rate of convergence. We recall that $\PP$ is taken with respect to $(X, Y)$.

For convenience, we introduce the notation
\begin{align}\label{eq_G_hat}
\wh G_x(x) &:= x^\T\wh \theta + \wh \beta_0,\qquad 
G_z(z) := z^\T \beta+ \beta_0
\end{align}
such that  
$
\wh g_x(x) = \1\{ \wh G_x(x)\ge0\}
$ 
from (\ref{eq_g_hat}) and, using the equivalence in Proposition \ref{prop_ls_rule},
\begin{align}\label{def_gz_star}
g_z^*(z)=  \1\{ G_z(z)\ge0\}.
\end{align}
Recall that $\wh g_x$ depends on the choice of $B$ via $\wh\theta$ and $\wh\beta_0$.

The following theorem provides a general bound for the excess risk of $\wh g_x$ that uses any $B$ in (\ref{def_theta_hat}). Its proof can be found in Appendix \ref{app_proof_thm_general_risk}.

\begin{thm}\label{thm_general_risk}
Under model (\ref{model_X}), assume {\rm (i) -- (iv)}.	For all $t>0$, we have
\begin{align}\label{disp_margin}
R_x(\wh g_x) -
R_z^*
&~ \le  ~ \PP \{| \wh G_x(X) - G_z(Z) | > t   ~ \mid \D  \} +  c_*  t  ~ P(t) 
\end{align}
where $c_* = \Dt^2 + (  \pi_0\pi_1)^{-1}$ and
\begin{align}\label{disp_margin_prob}
P(t) &=  \pi_0  \Bigl[\Phi\left(R\right) - \Phi\left(R - {t~ c_*/ \Dt}\right)\Bigr] +  \pi_1  \Bigl[\Phi\left(L+{t~c_*/ \Dt}\right) - \Phi\left(L\right)\Bigr]
\end{align}
with 
$$
L = -{\Dt\over 2} - {\log{\pi_1\over \pi_0} \over \Dt},\qquad R = {\Dt\over 2} - {\log{\pi_1\over \pi_0} \over \Dt}.
$$
\end{thm}

\begin{remark}\label{rem_margin}
The quantity $P(t)$ in (\ref{disp_margin_prob}) is in fact 
\[
\pi_0  \PP\{ - t< G_z(Z) <0 \mid Y=0\} + \pi_1   \PP\{ 0< G_z(Z) < t\mid Y=1\}
\]
which describes the probabilistic behavior of  the margin of the hyperplane $\{z: G_z(z) = 0\}$ that separates the distributions $Z\mid Y = 0$ and $Z\mid Y =1$.
Conditions that control the margin between  $Z\mid Y = 0$ and $Z\mid Y =1$ are more suitable in our current setting and have a different perspective than the usual margin condition in \cite{tsybakov2004optimal} that controls the probability $\PP\{|\eta(Z) - 1/2|<\delta \}$ for any $0\le \delta \le 1/2$, with $\eta(z) := \PP(Y=1 \mid Z=z)$.
\end{remark}

\begin{remark}[Extension to non-linear classifiers]\label{rem_extension}
The proof of Theorem \ref{thm_general_risk} also allows us to analyze more complex classifiers.
Indeed, let $\Lambda_z(z)$ be the logarithm of the ratio between $\PP(Z=z, Y = 1)$ and $\PP(Z=z, Y = 0)$, and let $\wh \Lambda_x(x)$ be an arbitrary estimate of $\Lambda_z(z)$. We can easily derive from our proof of Theorem \ref{thm_general_risk} the following excess risk bound for the classifier $\wh g_x(x) = \1\{\wh\Lambda_x(x)\ge 0\}$,
\begin{align} \label{bd_excess_risk}
R_x(\wh g_x) - R_z^*  & ~ \le ~  \PP \{| \wh \Lambda_x(X) - \Lambda_z(Z) | > t   ~ \mid \D \}\\\nonumber
& + t ~ 
\pi_0  \PP\{ -  t<  \Lambda_z(Z) <0 \mid Y=0\} + t ~
\pi_1 \PP\{ 0< \Lambda_z(Z) < t\mid Y=1\},
\end{align}
for any $t>0$.
Therefore, bound in (\ref{bd_excess_risk}) can be 
used as an initial step for analyzing any classification problems, particularly suitable for situations where conditional distributions $Z \mid  Y$ are specified. The remaining difficulty is to find a good estimator $\wh \Lambda_x(x)$ and to control $|\wh \Lambda_x(X) - \Lambda_z(Z)|$. For instance,  when $Z\mid Y =k$, for $k\in \{0,1\}$, have Gaussian distributions with different means and different covariances, the Bayes rule of using $Z$ (equivalently, $\Lambda_z(Z)$) becomes quadratic, leading to an estimator $\wh \Lambda_x(x)$ that is quadratic in $x$ as well. Since both the procedure and the analysis are different, we will study this setting in a separate paper.
\end{remark}

From (\ref{eq_G_hat}), we find the identity 
\begin{align}\label{decomp_G_x}
\wh G_x(X) - G_z(Z) &= Z ^\T(A^\T \wh \theta - \beta) + W^\T \wh \theta + \wh \beta_0- \beta_0.
\end{align}
To establish its deviation inequalities, our analysis uses the following distributional assumption on $W$ from (\ref{model_X}). 
We assume that
\begin{itemize}
\item[(v)]   $W = \sw^{1/2}\wt W$ and  $\wt W$ is a mean-zero $\gamma$-subGaussian random vector with $\EE[\wt W\wt W^\T] = \bI_p$ and
$\EE[\exp(u^\T \wt W)] \le \exp(\gamma^2/2)$, for all $\|u\|_2=1$.
\end{itemize}
We stress that the distributions of $X \mid Y$ need not be  Gaussian.
In addition, we require that
\begin{itemize}
\item[(vi)] $\pi_0$ and $\pi_1$ are fixed and bounded from below by some constant $c\in (0,1/2]$. 
\end{itemize}

The following proposition states a deviation inequality of $|\wh G_x(X) - G_z(Z)|$ which holds with high probability under the law $\PP^{\D}$. It depends on three quantities:
\begin{equation}\label{def_r_hat}
\wh r_1 := \|\sz^{1/2}(A^\T \wh \theta - \beta)\|_2,\quad  \wh r_2 := \|\wh\theta\|_2,  \quad \wh r_3 := {1\over \sqrt n}\|\W(P_B-P_A)\|_{\op}.
\end{equation} 
For any matrix $M$, let $P_M$ denote the projection onto its column space.
From (\ref{decomp_G_x}), appearance of the first two quantities in (\ref{def_r_hat}) is natural since  $Z$ and $W$ are independent of $\wh\theta$ and $\wh\beta_0$, and $Z$ and $W$  are subGaussian random vectors under the distributional  assumptions (iv) and (v). 
The third quantity $\|\W(P_B-P_A)\|_{\op}$ in (\ref{def_r_hat}) originates from $\wh\beta_0 - \beta_0$ and reflects the benefit of using a matrix $B$ that estimates the column space of $A$ well.

\begin{prop}\label{prop_risk}
Under model (\ref{model_X}), assume {\rm (i) -- (vi)} and $K\log n \le c n$ for some constant $c>0$. For any $a\ge 1$,  we have 
\begin{align}\label{bd_risk_deviation}
\PP^{\D} \left\{\PP\left\{ \left |\wh G_x(X) - G_z(Z) \right|\ge \wh \omega_n(a)   ~  \mid \D  \right\}     \lesssim  ~  n^{-a} \right\} = 1-\cO(n^{-1}).
\end{align}
Here, for some constant $C>0$ depending on $\gamma$ only,  
\begin{align}\label{def_hat_omega_n}
\wh\omega_n(a)  &= C  \left\{ \sqrt{a \log n} \left( \wh r_1 +\|\sw\|_\op^{1/2} ~ \wh r_2 \right)  +  \wh r_2 \wh r_3+ \sqrt{\log n\over n}\right\}.
\end{align}
\end{prop}
\begin{proof}
See Appendix \ref{app_proof_prop_risk}.
\end{proof}

Proposition \ref{prop_risk} implies that we need to control $\wh\omega_n(a)$ whose randomness solely depends on $\D$. 
In view of Theorem \ref{thm_general_risk} and Proposition \ref{prop_risk}, 
we have the following result.

\begin{thm}\label{thm_general_risk_explicit}
Under model (\ref{model_X}), assume {\rm (i) -- (vi)} and $K\log n \le c n$ for some constant $c>0$. For any $a\ge 1$ and any sequence $\omega_n>0$, on the event
$ \{ \wh \omega_n(a) \le \omega_n\}$, the following holds with probability $1-\cO(n^{-1})$ under the law $\PP^\D$, 
\begin{eqnarray*}
R_x(\wh g_x) - R_z^* ~\lesssim~ n^{-a} + 
\begin{cases} 
	\omega_n^2  & \text{ if $\Dt\asymp 1$}\\
	\omega_n^2 \exp\left\{-[c_\pi+o(1)]\Dt^2\right\} & \text{ if $\Dt\to\i$ and $\omega_n = o(1)$}\\
	\omega_n^2 \exp\left\{-[c' + o(1)] / \Dt^2\right\} & \text{ if $\Dt\to0$, $\pi_0\ne\pi_1$ and $\omega_n = o(1)$}\\
	\omega_n  \min\{1,\omega_n / \Dt\} & \text{ if $\Dt\to0$ and $\pi_0=\pi_1$}
\end{cases}
\end{eqnarray*}
Here $c_\pi$ and $c'$ are some absolute positive constants  and $c_\pi = 1/8$ if $\pi_0 = \pi_1$. 

\end{thm}

Hence, it remains to find a deterministic sequence $\omega_n\to0$
such that 
$ \PP^{\D}\{ \wh \omega_n(a) \le \omega_n\} \to1$  as $n\to\i$.  Further,  in view of (\ref{def_hat_omega_n}), all we need is to find deterministic upper bounds of $\wh r_1, \wh r_2$ and $\wh r_3$. In such way Theorem \ref{thm_general_risk_explicit} serves as a general tool for analyzing the excess risk of the classifier constructed via (\ref{def_theta_hat}) -- (\ref{eq_g_hat}) by using any matrix $B$. 

Later in Section \ref{sec_theory_application} we apply Theorem \ref{thm_general_risk_explicit} to analyze several classifiers, including the principal components based classifier by choosing $B = \U_K$ and  $B = \wt \U_K$ as well as their counterparts based on the  data-dependent choice $\wh K$.  For theses PC-based classifiers, we will find a sequence $\omega_n$ that closely matches the   sequence $\omega_n^*$ in (\ref{def_omega_star}) under suitable conditions, up to $\log(n)$,
for our procedure. 
In view of Theorem \ref{thm_lowerbound}, this rate turns out to be minimax-optimal over a subset of the parameter space considered in Theorem \ref{thm_lowerbound}, up to $\log(n)$ factors. 

Although not pursued in this paper, it is worth mentioning
some other reasonable choices of $B$ including, for instance, the identity matrix $\bI_p$ which leads to  the generalized least squares based classifier {\citep{BW22}}, the estimator of $A$ in \cite{LOVE}, the projection matrix from supervised PCA \citep{Bair_JASA,barshan2011supervised} and the projection matrix obtained via partial least squares regression \citep{barker2003partial,Nguyen2002}. 

\begin{remark}
We observe the same phase transition in Theorem \ref{thm_general_risk_explicit} for $\Dt \asymp 1$ amd $\Dt \to \i$ as discussed in Remark \ref{rem_Dt}. 
To the best of our knowledge, upper bounds of the excess risk in the regime $\Dt = o(1)$ are not known in the existing literature. Our result in this regime relies on a careful analysis which does not require any condition on $\Dt$, in contrast to the existing analysis of the classical high-dimensional LDA problems. For instance, under model (\ref{model_Gaussian}), \cite{caizhang2019} assumes $\Dt_x^2:=(\mu_1-\mu_0)^\T \Sigma^{-1}(\mu_1-\mu_0) \gtrsim 1$ and 
$
\Dt_x^2 (s\log n / n) = o(1)
$
to derive the convergence rate of their estimator of 
$
\Sigma^{-1}(\mu_1 - \mu_0)
$
with $s = \|\Sigma^{-1}(\mu_1 - \mu_0)\|_0$. As a result, their results of excess misclassification risk only hold for $\Dt_x \gtrsim 1$.
\end{remark}

\section{Rates of convergence of the PC-based classifier}
\label{sec_theory_application}

We apply our general theory in Section \ref{sec_theory_general} to several classifiers corresponding to different choices of $B=\U_K$, $B=\U_{\wh K}$, 
$B=\wt\U_K$ and $B=\wt\U_{\wh K}$ in (\ref{def_theta_hat}).
Since our  analysis is beyond the parameter space $\Theta(\lambda, \sigma, \Dt)$ in (\ref{def_Theta}), we first generalize  the  signal-to-noise ratio $\lambda/\sigma^2$  of predicting $Z$ from $X$ given $Y$ by introducing
\begin{equation}\label{def_xi_star}
\xi^* := 	{\lambda_K(A\szy A^\T) \over \lambda_1(\sw)}.
\end{equation}
We also need the related quantity
\begin{equation}\label{def_xi}
\xi  := {\lambda_K(A\szy A^\T) \over  \errw},
\end{equation}
that 
characterizes the signal-to-noise ratio of predicting $\Z$ from $\X=\Z A^\T + \W$.
Indeed, note that we
replaced $ \lambda_1(\sw)$ in \eqref{def_xi_star} by 
\begin{equation}\label{def_errw}
\errw=  \lambda_1(\sw) + {\tr(\sw) \over n} 
\end{equation} 
and  the largest eigenvalue of the random matrix $\W^\T \W/n$ is of order $\cO_\PP(\errw)$  
under assumption (v) (see, for instance, \cite[Lemma 22]{bing2020prediction}).

\subsection{Consistent estimation of the latent dimension $K$}\label{sec_PCLDA_K_hat}

Since in practice the true $K$ is often  unknown,  we analyze the estimated rank $\wh K$ selected from (\ref{def_K_hat}). 

Consistency of $\wh K$ under the factor model (\ref{model_X}) when $Z$ is a zero-mean subGaussian random vector
has been established in   \cite[Proposition 8]{bing2020prediction}.  Here we establish such property of $\wh K$ under  (\ref{model_X})  where $Z$ follows  a mixture of two Gaussian distributions. Let $r_e(\sw) = \tr(\sw) / \lambda_1(\sw) $ denote the effective rank of $\sw$.

\begin{thm}
\label{thm_PCR_K_hat}
Let $\wh K$ be defined in (\ref{def_K_hat})   for some absolute constant $c_0 > 0$. 
Under model (\ref{model_X}), assume {\rm (i) -- (vi)}, and, in addition,
\[
K\le \bar K,\ \xi \ge C\
\text{and}  \ r_e(\sw) \ge C'(n\wedge p)
\]
for some constants $C, C'>0$. Then, 
\[
\PP^{\D}\{\wh K=K\}=1 - \cO(n^{-1}).
\]
\end{thm}
\begin{proof}
The proof is deferred to Appendix \ref{app_proof_thm_PCR_K_hat}
\end{proof}

Theorem \ref{thm_PCR_K_hat} implies that the classifier that uses $B = \U_{\wh K}$ ($B = \wt\U_{\wh K}$) has the same excess risk bound as that uses $B = \U_{K}$ ($B = \wt\U_{K}$).  For this reason, we restrict our analysis in the remaining of this section to $B$ based on the first $K$ principal components of $\U$ and $\wt\U$.

The  condition $K\le \bar K$ holds, for instance, if $K\le c'(n\wedge p)$ with $c' \le \nu/(2c_0(1+\nu))$. 
Condition $r_e(\sw) \ge C'(n\wedge p)$ holds, for instance, in the commonly considered setting
\begin{equation*}
0 < c \le \lambda_p(\sw) \le \lambda_1(\sw) \le C< \i
\end{equation*}
while being more general.

 The condition that  $\xi \ge C$  
is also needed in our subsequent derivation of the rates  of the excess risks for the classifiers using $B = \U_K$ and $B = \wt\U_K$. This essentially requires $\xi^* \ge C$ in the low-dimensional settings, and $\xi^* \ge C(p/n)$ in the high-dimensional settings (see, Remark \ref{rem_cond_PCR}  below for details). 
Since the minimax lower bounds for the excess risk in Theorem \ref{thm_lowerbound} above contain  the term $\min(1, \Dt) / \xi^*$, it is imperative that the signal-to-noise ratio $\xi^*$ is large to guarantee good performance of the classifier, irrespective of the estimation of the latent dimension $K$.

We investigate in Appendix \ref{app_sim_K} the consequences of inconsistent estimates $\wh K$ and found that our proposed classifiers are robust against both under-estimation and over-estimation. This is corroborated  in our follow-up work \cite{BW22}, that proves that the classifier using $\wh \theta= (\Pi_n \X)^{+} \Y$ based on $B=\bI_p$ (in other words, $\wh K=p$), often is minimax optimal and performing slightly inferior to $B=\U_K$ in finite sample simulations.

\subsection{PC-based LDA by using the true dimension $K$}\label{sec_PCLDA_K}

The following theorem states the excess risk bounds of $\wh g_x$ that uses $B = \U_K$. Its proof can be found in Appendix \ref{app_proof_PCR}.  Denote by  $\kappa$ the condition number $\lambda_1(A\sz A^\T) / \lambda_K(A\sz A^\T)$
of the matrix $A\sz A^\T$.

\begin{thm}\label{thm_PCR}
Under model (\ref{model_X}), assume {\rm (i) -- (vi)}. If
$
K\log n \le  c n
$
and 
$\xi  \ge C\kappa^2$ for some constants $c, C>0$, 
then for any $a\ge 1$ and 
\begin{equation}\label{def_omega_n_PCR}
\omega_n(a) =  \left(\sqrt{K\log n\over n}+ \min\{1,\Dt\}\sqrt{1\over \xi^*} + \sqrt{\kappa\over \xi^2}\right) \sqrt{a\log n},
\end{equation}
we have 
$
\PP^{\D}\left\{
\wh\omega_n(a) \lesssim \omega_n(a)
\right\} = 1-\cO(n^{-1}).
$
Hence, with this probability,
the conclusion of Theorem \ref{thm_general_risk_explicit} holds for the classifier that uses $B = \U_K$ for $\omega_n(a)$ in (\ref{def_omega_n_PCR}).
\end{thm}

Theorem \ref{thm_PCR} requires $\xi \ge C\kappa^2$, which can be relaxed to $\xi \ge  C$, as shown in the proof (see, Remark 1 in Appendix A.4). However, the stronger condition can  lead to a faster rate when one has additional data set to construct $B=\wt \U_K$, as stated in the theorem below. Its proof can be found in Appendix \ref{app_proof_thm_PCR_indep}.

\begin{thm}\label{thm_PCR_indep}
Under the same conditions of Theorem \ref{thm_PCR}, 
for any $a>0$ and 
\begin{equation}\label{def_omega_n_PCR_indep}
\omega_n(a) =   \left(\sqrt{K\log n\over n} +  \min\{1,\Dt\}\sqrt{1\over \xi^*}\right)\sqrt{a\log n},
\end{equation} 
we have 
$
\PP^{\D}\left\{
\wh\omega_n(a) \lesssim \omega_n(a)
\right\} = 1-\cO(n^{-1}).
$
Hence, with this probability,
the conclusion of Theorem \ref{thm_general_risk_explicit} holds for the classifier that uses $B = \wt\U_K$ for $\omega_n(a)$ in (\ref{def_omega_n_PCR_indep}).
\end{thm}

\begin{remark}[Polynomially fast rates]\label{rem_fast_rates}
In view of Theorems \ref{thm_PCR} \& \ref{thm_PCR_indep},  fast rates (of the order $\cO(n^{-a})$ for arbitrary $a\ge 1$) are obtained for
both PC-based procedures, provided that (a) $\Dt^2\gg \log n$ or (b)  $1/\Dt^2 \gg \log n$ and $\pi_0\ne \pi_1$. 
\end{remark}

\begin{remark}[Advantage of using an independent dataset or data splitting]

Compared to (\ref{def_omega_n_PCR}) in Theorem \ref{thm_PCR}, the convergence rate of the excess risk of the classifier that uses $B=\wt \U_K$ does 
not have the third term $\sqrt{\kappa/\xi^2}$. This advantage only becomes evident when $p > n$ and $\xi^*$ is not sufficiently large. We refer to Remark \ref{rem_cond_PCR} below for detailed explanation.

To understand why using $\wt \U_K$, that is independent of $\X$, yields a smaller excess risk, recall that the third term in (\ref{def_omega_n_PCR}) originates from predicting $\Z$ from $\X$ and its derivation involves controlling $\|\W(P_{\U_K}-P_A)\|_\op$. Since $\U_K$ is constructed from $\X$, hence also depends on $\W$,  the dependence between $\W$ and $\U_K$ renders a slow rate for  $\|\W(P_{\U_K}-P_A)\|_\op$. 
The fact that auxiliary data can bring improvements (in terms of either smaller prediction / estimation error or weaker conditions) is a phenomenon that has been observed in other problems, such as the problem of estimating the optimal instrument in sparse high-dimensional instrumental variable model \citep{belloni2012sparse} and the problem of making inference on a low-dimensional parameter in the presence of high-dimensional nuisance parameters \citep{DDML}.
\end{remark}

\begin{remark}[Simplified rates within $\Theta(\lambda, \sigma, \Dt)$]\label{rem_cond_PCR}
To obtain more insight from the results of Theorems \ref{thm_PCR} \& \ref{thm_PCR_indep}, consider $\theta \in \Theta(\lambda, \sigma, \Dt)$  in (\ref{def_Theta}) with $\Dt\asymp 1$ such that 
$\pi_0 = \pi_1$, $1 / \xi^* \asymp \sigma^2 / \lambda$, $1/\xi \asymp (\sigma^2/\lambda)(1 + p/n)$ and $\kappa\asymp 1$.
In this case, combining Theorems \ref{thm_general_risk_explicit}, \ref{thm_PCR} and \ref{thm_PCR_indep} reveals that, with probability $1-\cO(n^{-1})$, 
\begin{alignat}{2}\label{bd_simp_PCR}
&R_x(\wh g_x) - R_z^* ~\lesssim~ \left[{K\log n\over n} + {\sigma^2 \over \lambda} + \left(
{p\over n}{\sigma^2 \over \lambda}
\right)^2\right]\log n, \quad  && \text{if $B = \U_K$;}\smallskip\\\label{bd_simp_PCR_indep}
&R_x(\wh g_x) - R_z^* ~\lesssim~	\left[{K\log n\over n} + {\sigma^2 \over \lambda}\right]\log n,  && \text{if $B = \wt\U_K$}.
\end{alignat}
We have the following conclusions. 
\begin{enumerate}
\item[(1)] If $p < n$, the  two rates above coincide and equal (\ref{bd_simp_PCR_indep}), whence consistency of both PC-based classifiers requires that $K\log^2 n / n \to 0$ and $\sigma^2 \log n  /\lambda \to 0$.

\item[(2)] If $p>n$,  it  depends on the signal-to-noise ratio  (SNR) $\lambda/\sigma^2$ whether or not consistency of the classifier with $B = \U_K$  requires an additional condition. 
\begin{enumerate}
\item  If  the SNR is large such that 
\begin{equation}\label{cond_SNR}
	{\lambda \over \sigma^2} ~ \gtrsim ~  \min\left\{
	\left(p\over n\right)^2, ~ {p\over \sqrt{nK\log n}}
	\right\},
\end{equation}
the two rates in (\ref{bd_simp_PCR}) and (\ref{bd_simp_PCR_indep}) also coincide and equal (\ref{bd_simp_PCR_indep}). In this case,  there is no apparent benefit of using an auxiliary data set.

\item For relatively smaller values of SNR that fail (\ref{cond_SNR}), the effect of using $B = \wt \U_K$ based on an independent data set $\tilde \X$  is real as evidenced in Figure \ref{fig_split} below where we keep $\lambda/\sigma^2$, $n$ and $K$ fixed but let $p$ grow.

\begin{figure}[ht]
	\centering
	\includegraphics[width=0.45\textwidth]{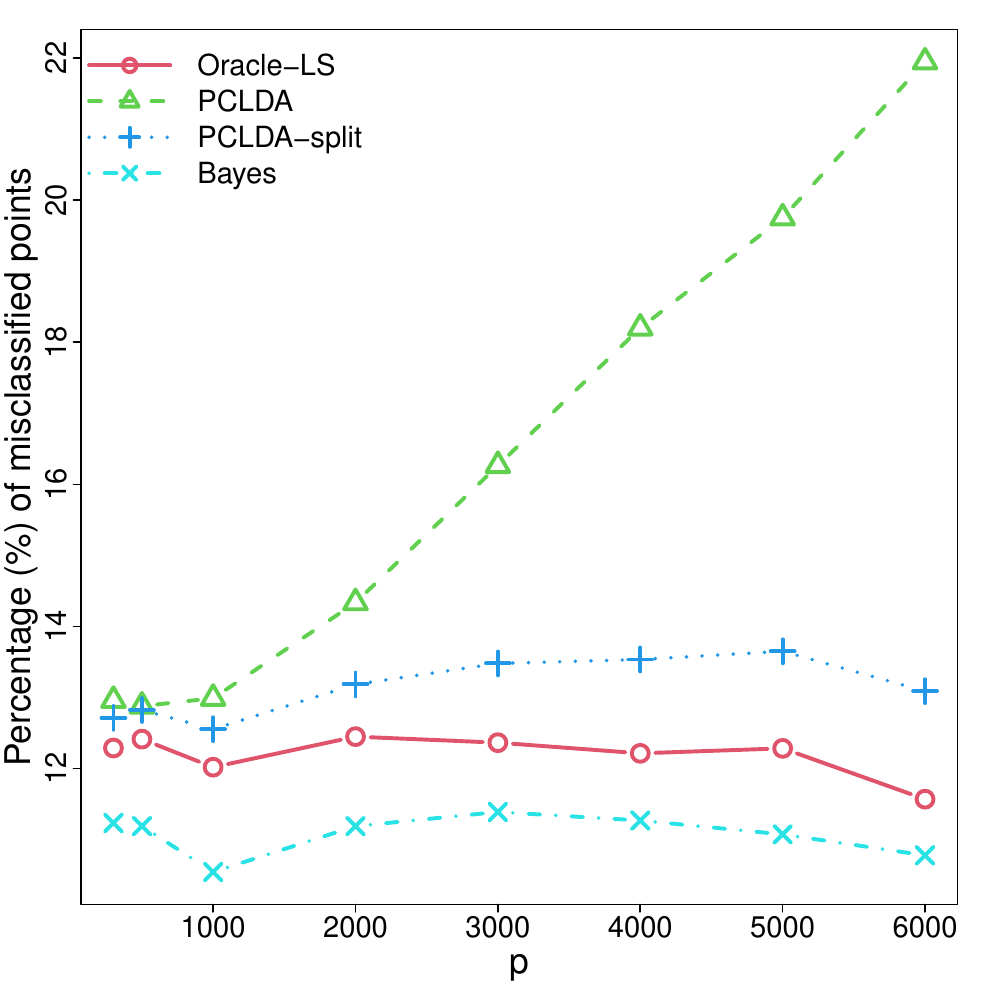}
	\caption{Illustration of the advantage of constructing $\wt\U_K$ from an independent dataset: PCLDA represents the PC-based classifier based on $B=\U_K$ while PCLDA-split uses $B=\wt\U_K$ that is constructed from an independent $\tX$. Oracle-LS is the oracle benchmark that uses both $Z$ and $\Z$ while Bayes represents the risk of using the oracle Bayes rule.  We fix $n = 100$ and $K = 5$ and keep $\lambda/\sigma^2$ fixed, while we let $p$ grow. We refer to Section \ref{sec_sim} for detailed data generating mechanism.}
	\label{fig_split}
\end{figure}

\item It is worth mentioning that if the SNR is sufficiently large such that 
\[
{\lambda \over \sigma^2} ~ \gtrsim ~  \max\left\{
\left(p\over n\right)^2, ~ {p\over \sqrt{nK\log n}}
\right\},
\]
both errors due to not observing $Z$ and estimation of the column space of the matrix $A$ are negligible compared to the parametric rate $K/n$, to wit,  both rates  in (\ref{bd_simp_PCR}) and (\ref{bd_simp_PCR_indep}) reduce to $K\log^2 n / n$.
\end{enumerate}
\end{enumerate}
Conditions $\lambda \gtrsim p$ and $\sigma^2 = \cO(1)$ are common  in the analysis of factor models with a diverging number of features \citep{bai2012,fan2013large,SW2002_JASA}. For instance,  $\lambda \gtrsim p$ holds when eigenvalues of $\szy$ are bounded and a fixed proportion of rows of $A$ are i.i.d. realizations of a sub-Gaussian random vector with covariance matrix having bounded eigenvalues as well. {In this case, the bounds in \eqref{bd_simp_PCR} and \eqref{bd_simp_PCR_indep} reduce to 
\[
{K\log^2 n \over n} + {\log n\over p},
\]
which decreases as $p$ increases. }
Nevertheless, consistency of the PC-based classifiers only requires $\lambda / \{\sigma^2 \log n(1 + p/n)\} \to \i$ for $B = \U_K$ and $\lambda / (\sigma^2\log n)\to \i$ for $B = \wt \U_K$,  which are
both 
much milder conditions.

\end{remark}

\subsection{Optimality of the PC-based LDA by sample splitting}\label{sec_optimality_LDA}

We now show that the PC-based LDA by sample splitting achieves the minimax lower bounds in Theorem \ref{thm_lowerbound}, up to multiplicative logarithmic factors of $n$. 
Recalling that (\ref{def_Theta}), for any $\theta \in \Theta(\lambda, \sigma, \Dt)$, one has $\pi_0 = \pi_1$, $1 / \xi^* \asymp \sigma^2 / \lambda$, $1/\xi \asymp (\sigma^2/\lambda)(1 + p/n)$ and  $1\lesssim \kappa\lesssim 1+\Dt^2$. Based on Theorem \ref{thm_PCR_indep},  we have the following corollary for the classifier that uses $B =\wt\U_K$. Its proof can be found in Appendix \ref{app_proof_cor_PCR}. We use the notation $\lessapprox$ for inequalities that hold up to a multiplicative logarithmic factor of $n$. Recall $\omega_n^*$ from (\ref{def_omega_star}).

\begin{cor}\label{cor_PCR}
Under model (\ref{model_X}), assume {\rm (i) -- (v)},
$
K\log n\le  cn,
$
$\kappa^2\sigma^2 / \lambda \le c'$ and  $\kappa^2 \sigma^2p/(\lambda n)\le c''$ for some constants $c, c',c''>0$.  For any $\theta \in \Theta(\lambda, \sigma, \Dt)$, with probability $1-\cO(n^{-1})$, the classifier that uses $B = \wt\U_K$ satisfies the following statements.
\begin{enumerate}				
\item[(1)] If 
$\Dt \asymp 1$,  then
\[
R_x(\wh g_x) - R_z^* ~\lessapprox~   (\omega_n^*)^2 .
\] 
\item[(2)] If $\Dt \to \i$,
and additionally, 
$(\log n + \Dt^2) K\log n /n \to 0$ and $(\log n + \Dt^2) \sigma^2/\lambda \to 0$ as $n\to\i$, 
then
\[
R_x(\wh g_x) - R_z^* ~\lessapprox~ (\omega_n^*)^2 \exp\left\{-\left[{1\over 8} + o(1)\right]\Dt^2\right\}.
\] 
\item[(3)] If $\Dt \to 0$ as $n\to\i$, then
\[R_x(\wh g_x) - R_z^* ~\lessapprox~ \min\left\{
{\omega_n^* \over \Dt}, 1 \right\} \omega_n^*.
\] 
\end{enumerate}
\end{cor}

In view of Theorem \ref{thm_lowerbound} and Corollary \ref{cor_PCR}, we conclude the optimality of PC-based procedure that uses $B = \wt\U_K$ over $\Theta(\lambda,\sigma, \Dt)$.  For $\Dt \to \i$, if conditions in (2)  are not met such as $\Dt^2 \gtrsim n / K$ or $\Dt^2 \gtrsim \lambda / \sigma^2$, the PC-based procedure still has $n^{-a}$ convergence rate of its excess risk, for arbitrary large $a\ge 1$, as commented in Remark \ref{rem_fast_rates}.

Regarding the PC-based classifier that does not resort to sample splitting, according to Theorems \ref{thm_lowerbound} \&  \ref{thm_PCR}, its excess risk also achieves optimal rates of convergence when $\lambda /\sigma^2$ is large  in the precise sense that 
\[
{\lambda \over \sigma^2} ~ \gtrsim ~  \min\left\{
{1\over \min\{1,\Dt\}}\left(p\over n\right)^2, ~ {p\over \sqrt{nK\log n}}
\right\}.
\]

\section{Simulation study}\label{sec_sim}

We conduct various simulation studies in this section to compare the performance of our proposed algorithm with other competitors. For our proposed algorithm, we call it PCLDA standing for the Principal Components based LDA. The name PCLDA-$K$ is reserved when the true $K$ is used as input. When $K$ is estimated by $\wh K$, we use PCLDA-$\wh K$ instead.  
We call PCLDA-CF-$k$ the PCLDA with $k$-fold cross-fitting. We consider $k = 5$ in our simulation as suggested by \cite{DDML}.  To set a benchmark for PCLDA-CF-$k$, we use PCLDA-split that uses an independent copy of $\X$ to compute $\wt \U_K$. 
On the other hand, we compare with the nearest shrunken centroids classifier (PAMR) \citep{Tibshirani2002}, the $\ell_1$-penalized linear discriminant (PenalizedLDA) \citep{Witten2011} and the direct sparse discriminant analysis (DSDA) \citep{mai2012}\footnote{PAMR, PenalizedLDA and DSDA are implemented in the R packages \textsf{pamr}, \textsf{penalizedLDA} and \textsf{TULIP}, respectively.}. Finally, we choose the performance of the oracle procedure (Oracle-LS) as benchmark in which Oracle-LS uses both $\Z$ and $Z$ to estimate $\beta$, $\beta_0$ and the classification rule $g_z$ in (\ref{LDAa}). 

We generate the data as follows. First, we set $\pi_0 = \pi_1 = 0.5$, $\alpha_0 = \0_K$ and $\alpha_1  = \b1_K\sqrt{\eta/K}$. The parameter $\eta$ controls the signal strength $\Delta$ in  (\ref{def_Dt}). We generate $\C$ by independently sampling its diagonal elements $[\C]_{ii}$ from Unif(1,3) and set its off-diagonal elements as 
$$[\C]_{ij} = \sqrt{[\C]_{ii}[\C]_{jj}}  (-1)^{i+j}(0.5)^{|i-j|},\quad  \text{ for each $i\ne j$}.$$ The covariance matrix $\sw$ is generated in the same way, except we set $\diag(\sw) = \b1_p$. The rows of  $\W\in \RR^{n\times p}$ are generated independently from $N_p(0, \sw)$. Entries of $A$ are generated independently from $N(0, 0.3^2)$. The training data $\Z$, $\X$ and $\Y$ are generated according to model  (\ref{model_X}) and (\ref{model_YZ}). In the same way, we generate 100 data points that serve as test data for calculating the (out-of-sample) misclassification error for each algorithm.

In the sequel, we vary the dimensions $n$ and $p$ as well as the signal strength $\Delta$ in (\ref{def_Dt}), one at a time. 
For each setting, we repeat the entire procedure 100 times and  averaged misclassification errors for each algorithm are reported. 

\subsection{Vary the sample size $n$}
We set $\eta = 5$, $K = 10$, $p = 300$ and vary $n$ within  $\{50, 100, 300, 500, 700\}$. The left-panel in Figure \ref{fig_error} shows the averaged  misclassification error (in percentage) of each algorithm on the test data sets. Since $\wh K$ consistently estimates $K$, we only report the performance of PCLDA-$K$. We also exclude the performance of PCLDA-split and PCLDA-CF-$5$ since they all have similar performance as PCLDA-$K$\footnote{This is as expected since our data generating mechanism ensures $\xi^* \asymp p$ in which case PCLDA-split has no clear advantage comparing to PCLDA-$K$ (see, discussions after Theorem \ref{thm_PCR_indep}).}. The blue line represents the optimal Bayes error. All algorithms perform better as the sample size $n$ increases. As expected, Oracle-LS is the best because it uses the true $\Z$ and $Z$. Among the other algorithms, PCLDA-$K$ has the closest performance to Oracle-LS in all settings. The gap between PCLDA-$K$ and Oracle-LS does not close as $n$ increases. According to Theorem \ref{thm_PCR}, this is because such a gap mainly depends on $1/\xi$ which does not vary with $n$.

\begin{figure}[ht]
\centering 
\includegraphics[width = .45\textwidth]{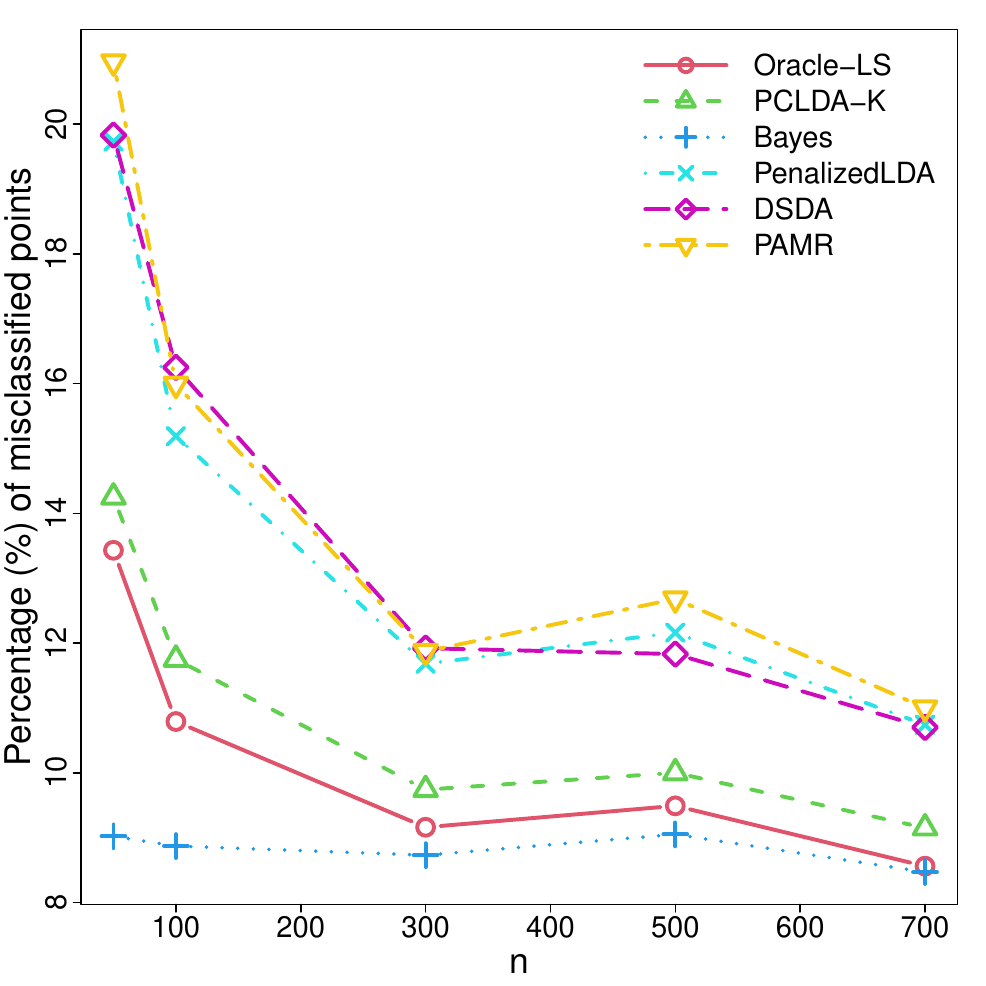} 
\includegraphics[width = .45\textwidth]{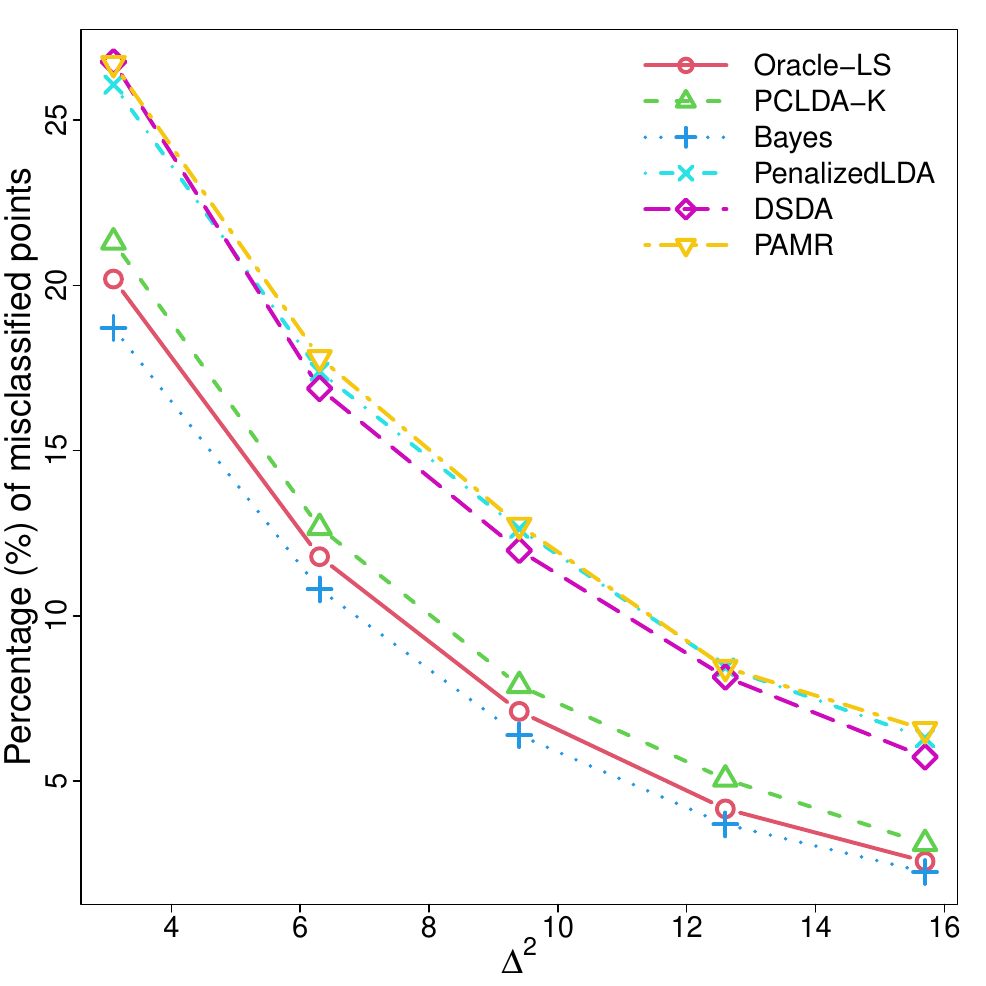}
\caption{The averaged misclassification errors of each algorithm. We vary $n$ in the left panel while vary $\Delta$ in the right one.  }
\label{fig_error} 
\vspace{-3mm} 
\end{figure}

\subsection{Vary the signal strength $\Delta^2$}

We fix $K = 5$, $n = 100$, $p = 300$ and vary $\eta$ within $\{2,4,6,8,10\}$. As a consequence, the signal strength $\Delta^2$ varies within $\{3.1, 6.3, 9.4, 12.6, 15.7\}$. The right-panel of Figure \ref{fig_error} depicts the averaged misclassification errors of each algorithm. For the same reasoning as before, we exclude PCLDA-$\wh K$, PCLDA-CF-$5$ and PCLDA-split. It is evident that all algorithms have better performance as the signal strength $\Delta$ increases. Among them, PCLDA-$K$ has the closest performance to Oracle-LS and Bayes in all settings.

\subsection{Vary the feature dimension $p$}

We examine the performance of each algorithm when the feature dimension $p$ varies across a wide range. Specifically, we fix $K=5$, $\eta = 5$, $n = 100$ and vary $p$ within $\{100, 300, 500, 700, 900\}$. Figure \ref{fig_error_p} shows the misclassification errors of each algorithm. The performance of PCLDA-$K$ improves  and gets closer to that of Oracle-LS as $p$ increases, in line with Theorem \ref{thm_PCR}. The gap between Oracle-LS and Bayes is due to the fact that  both $n$ and $\Dt$ are held fixed. 
\begin{figure}[ht]
\centering 
\includegraphics[width = .45\textwidth]{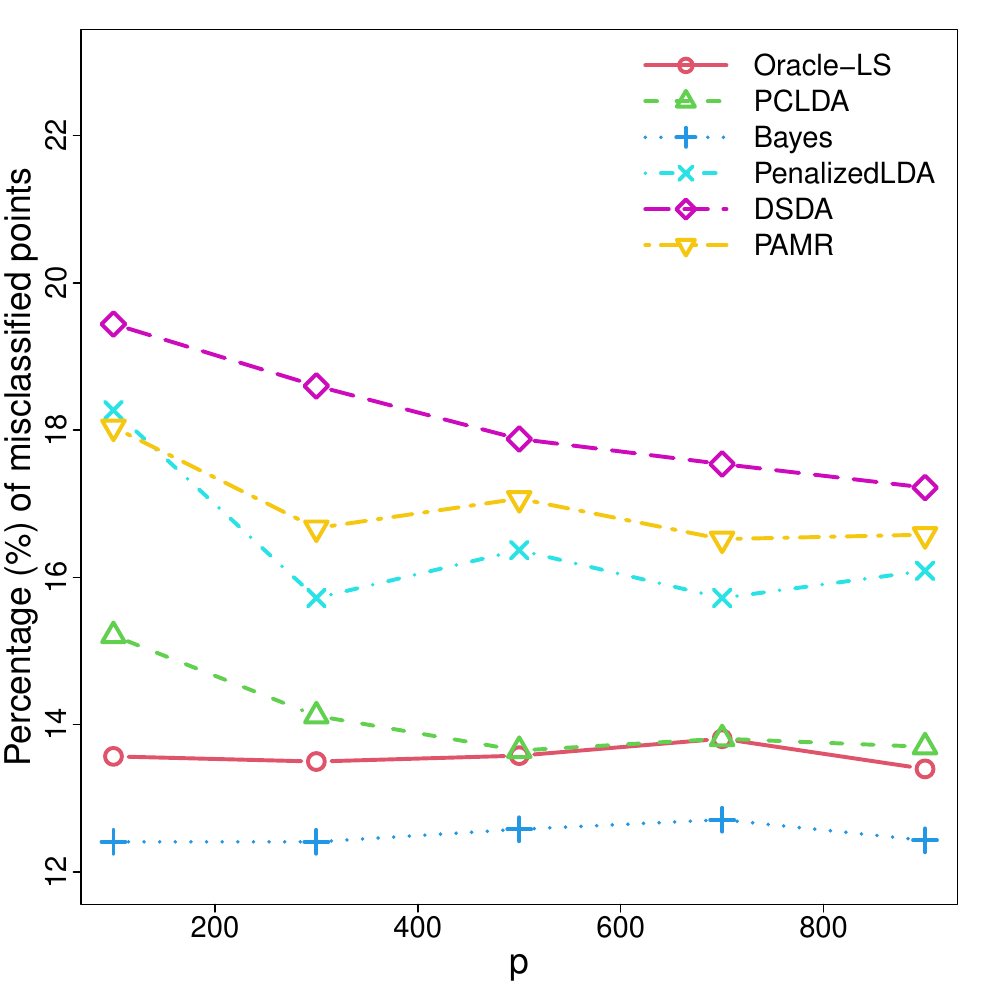}
\caption{The averaged misclassification errors of each algorithm for various choices of  $p$.}
\label{fig_error_p} 
\vspace{-5mm}
\end{figure}

\section{Real data analysis}\label{sec_real_data}

To further illustrate the effectiveness of our proposed method, we analyze three popular gene expression datasets (leukemia data, colon data  and lung cancer data)\footnote{Leukemia data is available at \url{www.broad.mit.edu/cgi-bin/cancer/datasets.cgi}. Colon data is available from the R package \textsf{plsgenomics}. Lung cancer data is available at \url{www.chestsurg.org}.}, which have been widely used to test classification methods, see, for instance,  \cite{Alon1999,Dettling2004,Nguyen2002,Singh2002} and also, the more recent literature, \cite{caizhang2019,FanFan2008, mai2012}. These datasets   contain thousands or even over ten-thousand features with around one hundred samples (see, Table \ref{tab_data}). In such challenging settings, LDA-based classifiers that are designed for high-dimensional data  not only are easy to interpret but also have competing and even superior performance than other, highly complex classifiers such as classifiers based on kernel support vector machines, random forests and boosting \citep{Dettling2004,mai2012}. 

\begin{table}[H]
\centering
\caption{Summary of three data sets.}
\label{tab_data}
{\renewcommand{\arraystretch}{1.2}{
\resizebox{\textwidth}{!}{
	\begin{tabular}{l|c|c|l|l}
		Data name  &  $p$ & $n$ & $n_0$ (category) & $n_1$ (category)\\
		\hline
		Leukemia   & 7129 &  72  & 47 (acute lymphoblastic leukemia) & 25 (acute myeloid leukemia)\\
		Colon   & 2000 & 62 & 22 (normal) & 40 (tumor)\\
		Lung cancer  & 12533 & 181 & 150 (adenocarcinoma) & 31 (malignant pleural mesothelioma)\\
		\hline
	\end{tabular}
}}}
\end{table}

Since the goal is to predict a dichotomous response, for instance, whether one sample is a tumor or normal tissue, we compare the classification performance of each algorithm. For all three data sets, the features are standardized to zero mean and unit standard deviation. For each dataset, we randomly split the data, within each category, into 70\% training set and 30\% test set. Different classifiers are fitted on the training set and their misclassification errors are computed on the test set. This whole procedure is repeated 100 times. The averaged misclassification errors (in percentage) as well as their standard deviations of each algorithm are reported in Table \ref{tab_data_error}. Our proposed PC-based LDA classifiers have the smallest misclassification errors over all datasets. Although PCLDA-CF-5 only has the second best performance in Colon and Lung cancer data sets, its performance is very close to that of  PCLDA-$\wh K$.

\begin{table}[H]
\centering
\caption{The averaged misclassification errors (in percentage). The numbers in parentheses are the standard deviations over 100 repetitions.}
\label{tab_data_error}
{\renewcommand{\arraystretch}{1.2}{
\begin{tabular}{l|c|c|c|c|c}
	& PCLDA-$\wh K$ & PCLDA-CF-5  & DSDA  & PenalizedLDA & PAMR \\
	\hline
	Leukemia  & 3.57 (0.036) & {\bf 3.04} (0.032)  & 5.52  (0.044)  & 3.91 (0.043) &  4.61     (0.039) \\
	Colon   & {\bf 16.37} (0.077) &   18.11 (0.082)  &  18.11 (0.07) &  33.95 (0.086) & 19.00  (0.089)\\
	Lung cancer & {\bf 0.55} (0.008)  &   0.60 (0.009)  &     1.69 ( 0.017)  &     1.80 (0.026)   &   0.91 (0.011)      \\
	\hline
\end{tabular}    
}}
\end{table}

\section{Extension to multi-class classification}\label{sec_multi_level}

In this section, we discuss how to extend the previously discussed procedure to multi-class classification problems in which $Y$ has $L$ classes, $\cL := \{0, 1, \ldots, L-1\}$, for some positive integer $L\ge 2$, and 
\begin{equation}\label{model_YZ_multi}
Z \mid  Y = k \sim N_K(\a_k, \szy),\qquad \PP(Y=k) = \pi_k, \qquad k\in \cL.
\end{equation}
In particular, the covariance matrices for the $L$ classes are the same.

For a new point $z\in \RR^K$,  the oracle Bayes rule assigns  it to $k \in \cL$ if and only if 
\begin{align}\label{oracle_rule}\nonumber
k 
= \argmax_{\ell\in \cL}  \PP(Y = \ell \mid Z = z)
&= \argmax_{\ell\in \cL}  \log{\PP(Z = z, Y = \ell) \over  \PP(Z = z, Y = 0) }\\
&= \argmax_{\ell\in \cL}  \left( z^\T \eta^{(\ell)} + \eta^{(\ell)}_0\right) := \argmax_{\ell\in \cL} ~ G_z^{(\ell | 0)}(z)
\end{align}
where 
\begin{equation}\label{def_etas_ell}
\eta^{(\ell)} =  \szy^{-1}(\a_\ell - \a_0),\quad \eta^{(\ell)}_0 = -{1\over 2}(\a _0 + \a _\ell )^\T \eta^{(\ell)} + \log {\pi_\ell \over \pi_0},\quad \forall ~ \ell \in \cL.
\end{equation}
Notice that $G_z^{(0 | 0)}(z) = 0$ and, for any $\ell \in \cL\setminus \{0\}$, the proof of (\ref{eq_Gz_Gzstar}) reveals that, 
\begin{equation}\label{eq_eta_beta_multi}
G_z^{(\ell | 0)}(z) = z^\T \eta^{(\ell)} + \eta^{(\ell)}_0 =  {1\over \bar \pi_0\bar \pi_\ell[1 - (\a_\ell - \a_0)^\T \beta^{(\ell)}]} \left( z^\T \beta^{(\ell)} +\beta_0^{(\ell)}\right) 
\end{equation}
with $\bar \pi_0 = \pi_0 / (\pi_0 + \pi_\ell)$, $\bar \pi_\ell = \pi_\ell / (\pi_0 + \pi_\ell)$,
\begin{align}\label{def_betas_ell}
&\beta^{(\ell)}  = \left[
\Cov(Z \mid  Y\in \{0,\ell\})
\right]^{-1}\Cov(Z, \1\{Y = \ell\} \mid Y \in \{0,\ell\}),\\\nonumber
&\beta^{(\ell)}_0 = -{1\over 2}(\a_0 + \a_\ell)^\T \beta^{(\ell)} + \bar \pi_0 \bar \pi_\ell \left(
1 - (\a_\ell - \a_0)^\T \beta^{(\ell)}
\right)\log{\bar\pi_\ell \over \bar\pi_0}.
\end{align}
In view of (\ref{oracle_rule}) and (\ref{eq_eta_beta_multi}), for a new point $x\in \RR^p$ and any matrix $B\in \RR^{p\times q}$ with $q\in [p]$, we propose the following  multi-class classifier 
\begin{equation}\label{def_gx_star_hat_multi}
\wh g_x^*(x) = \argmax_{\ell\in \cL} ~ \wh G_x^{(\ell | 0)}(x)
\end{equation}
where  $\wh G_x^{(0 | 0)}(x) = 0$ and, for any $\ell \in \cL\setminus \{0\}$,
\begin{align}\label{def_G_hat_ell_0}
\wh G_x^{(\ell | 0)}(x) = {1\over \wt\pi_0\wt\pi_\ell[1 - (\wh\mu_\ell - \wh \mu_0)^\T \wh \theta^{(\ell)}]}\left(x^\T \wh\theta^{(\ell)} + \wh\beta_0^{(\ell)}\right)
\end{align}
with
\begin{align*}
&\wt \pi_\ell = {n_\ell \over n_0 + n_\ell},\\
&\wh\theta^{(\ell)}  = B \left(\Pi_{(n_0+n_\ell)} \X^{(\ell)} B\right)^+ \Y^{(\ell)} ,\\
&\wh \beta^{(\ell)}_0 =   -{1\over 2}(\wh \mu_0+\wh \mu_\ell)^\T \wh\theta^{(\ell)} +  \wt \pi_0\wt \pi_\ell \left(
1   -(\wh \mu_\ell-\wh \mu_0 )^\T\wh\theta^{(\ell)}  \ 
\right)\log{\wt \pi_\ell \over \wt \pi_0}.
\end{align*}
Here $n_\ell$ and $\wh\mu_\ell$ are the  non-parametric estimates as (\ref{def_pi_hat}) and both the submatrix $\X^{(\ell)}\in \RR^{(n_0+n_\ell)\times p}$ of $\X$ and  the response vector $\Y^{(\ell)}  =  \{0,1\}^{(n_0+n_\ell)}$ correspond to  samples with label in $\{0,\ell\}$. Note that $\Y^{(\ell)}$ is encoded as $1$ for observations 
with label $\ell$ and $0$ otherwise.

To analyze the classifier $\wh g_x^*$ in (\ref{def_gx_star_hat_multi}),  its excess risk depends on 
\begin{equation}\label{def_r_hat_multi}
\wh r_1 =\max_{\ell \in \cL \setminus \{0\}} \left\|\bigl[\sz^{(\ell)}\bigr]^{1/2}\bigl(A^\T \wh \theta^{(\ell)} - \beta^{(\ell)}\bigr)\right\|_2, \qquad \wh r_2 = \max_{\ell \in \cL \setminus \{0\}}\bigl\|\wh\theta^{(\ell)}\bigr\|_2
\end{equation}
as well as $\wh r_3$ 
as defined in (\ref{def_r_hat}).  Here $\sz^{(\ell)} := \Cov(Z \mid  Y\in \{0,\ell\})$.  Analogous to 
(\ref{def_hat_omega_n}), for some constant $C=C(\gamma)>0$, define 
\begin{align}\label{def_wh_omega_n}
\wh \omega_n   = 	  C\sqrt{\log n}\left(\wh r_1 + \|\sw\|_\op^{1/2}\wh r_2 + \wh r_2 \wh r_3 +  \sqrt{L \over n} \right).
\end{align}
For ease of presentation, we also assume there exists some sequence $\Dt>0$  and some absolute constants $C>c>0$ such that 
\begin{equation}\label{cond_Dt_multi}
c~  \Dt  \le   \min_{k,\ell \in \cL, ~ k \ne \ell} \|\a_\ell - \a_k\|_{\szy} \le   \max_{k,\ell \in \cL, ~ k \ne \ell}  \|\a_\ell - \a_k\|_{\szy} \le C\Dt.
\end{equation}

The following theorem extends Theorem \ref{thm_general_risk_explicit} to multi-class classification by establishing rates of convergence of the excess risk of $\wh g_x^*$ in (\ref{def_gx_star_hat_multi})  for a general $B\in\RR^{p\times q}$.

\begin{thm}\label{thm_risk_multi}
Under model (\ref{model_X}) and (\ref{model_YZ_multi}), assume {\rm (i) -- (iii)} and (\ref{cond_Dt_multi}). Further assume 
$c/L \le \min_{k\in \cL} \pi_k \le \max_{k\in \cL} \pi_k \le C/L$ and $LK\log n \le c' n$
for some constants $c,c',C>0$.  Then, for any sequence $\omega_n >0$ satisfying $(1 + \Dt^2) \omega_n = o(1)$ as $n\to \infty$, on the event
$\{\wh\omega_n \le \omega_n\}$,  the following holds with probability at least $1-\cO(n^{-1})$ under the law $\PP^{\D}$.
\begin{enumerate}				
\item[(1)] If $\Dt \asymp 1$,  then
\[
R_x(\wh g_x^*) - R_z^* ~\lesssim~  L ~ \omega_n^2.
\] 
\item[(2)] If $\Dt \to \i$, 
then, for some constant $c''>0$,
\[
R_x(\wh g_x^*) - R_z^* ~\lesssim~ L ~ \omega_n^2~  \exp\left\{-\left[c'' + o(1)\right]\Dt^2 \right\}
\]
\item[(3)] If $\Dt = o(1)$, then, 
\[
R_x(\wh g_x^*) - R_z^* ~\lesssim~ L ~ \omega_n \min\left\{
{\omega_n \over \Dt}, 1
\right\} .
\]
\end{enumerate}
\end{thm}

\begin{proof}
The proof can be found in Appendix \ref{app_sec_multi_level}.
\end{proof}

Condition (\ref{cond_Dt_multi}) is only assumed to simplify the presentation. It is straightforward to derive results based on our analysis   when the separation $\|\a_\ell - \a_k\|_{\szy}$ is not of the same order for all $\ell, k \in \cL$. For the third case, $\Dt = o(1)$, our proof also allows to establish different convergence rates depending on whether or not $\pi_k$ and $\pi_\ell$ are distinct for each $k\ne \ell$,  analogous to the last two cases  of Theorem \ref{thm_general_risk_explicit}. However, we opt for the current  presentation for succinctness.

Theorem \ref{thm_risk_multi} immediately leads to the following corollary for the PC-based classifiers that use $B = \U_K$ and $B = \wt \U_K$.  Furthermore, Theorem \ref{thm_PCR_K_hat} also ensures that similar guarantees can be obtained for  the classifiers in (\ref{def_gx_star_hat_multi}) that use $B = \U_{\wh K}$ and $B = \wt \U_{\wt K}$.

\begin{cor}\label{cor_risk_multi}
Assume the conditions in Theorem \ref{thm_risk_multi} and $\xi \ge C\kappa^2$ for some constant $C>0$. Then, the conclusion of Theorem \ref{thm_risk_multi} holds for the classifier in (\ref{def_gx_star_hat_multi}) that uses
\begin{enumerate}
\item[(1)] $B = \U_K$
with  
\[
\omega_n=  \left(\sqrt{LK\log n\over n}+ \min\{1,\Dt\}\sqrt{1\over \xi^*} + \sqrt{\kappa\over \xi^2}\right)\sqrt{\log n},
\]
\item[(2)]  $B = \wt \U_K$ with 
\[
\omega_n =\left(\sqrt{LK\log n\over n}+ \min\{1,\Dt\}\sqrt{1\over \xi^*} \right)\sqrt{\log n}.
\]
\end{enumerate}
\end{cor}
\begin{proof}
See Appendix \ref{app_proof_cor_risk_multi}. 
\end{proof}

\begin{remark}\label{rem_multi}
Multi-class classification problems based on discriminant analysis have been studied, for instance, by \cite{cai2019convex,clemmensen2011sparse,mai2019multiclass,Witten2011}. Theoretical guarantees are only provided in \cite{mai2019multiclass} and \cite{cai2019convex} under the classical LDA setting for moderate / large separation scenarios, $\Dt \gtrsim 1$, and for fixed $L$, the number of classes. See also the work \cite{AP19} that derives bounds for the misclassification error (rather than excess risk) in a set-up similar to LDA, and reports a similar  phase transition phenomenon between $\Dt\asymp1$ and $\Dt\to\i$. 
Our results fully characterize dependence of the excess risk on $L$ and also cover the weak separation case, $\Dt \to 0$. \end{remark}

\begin{remark}
The classifier in (\ref{def_gx_star_hat_multi}) chooses $Y = 0$ as the baseline. In practice, we recommend taking each class as the baseline one at the time and averaging the predicted probabilities. Specifically, it is easy to see that, for any baseline choice $k\in \cL$ and for any $\ell\in \cL$,
\[
\PP\left(
Y = \ell \mid Z = z
\right) = {\PP\left(
Z = z, Y = \ell 
\right) \over  \sum_{k'\in \cL} \PP\left(
Z = z, Y = k'
\right)} = { \exp\left\{G_z ^{(\ell | k)}(z)\right\} \over  \sum_{k'\in \cL}\exp\left\{ G_z ^{(k' | k)}(z)\right\}}
\]
where $G_z^{(\ell | k)}(z)$ is defined analogous to (\ref{oracle_rule}) with $k$ in lieu of $0$.
Therefore, for any new data point $x\in \RR^p$, the averaged version of the classifier in  (\ref{def_gx_star_hat_multi}) is
\[
\argmax_{\ell\in \cL}{1\over L}\sum_{k \in \cL}  {\exp\left\{\wh G_x^{(\ell | k)}(x)\right\} \over \sum_{k' \in \cL}\exp\left\{\wh G_x^{(k' | k)}(x)\right\}}
\]
with $\wh G_x^{(\ell | k)}(x)$ defined analogous to (\ref{def_G_hat_ell_0}).
This classifier tends to have better 
finite sample performance, as revealed by the simulation study in Appendix \ref{app_sim_multi}.
\end{remark}

\begin{acks}[Acknowledgments]
The authors would like to thank the Editor, Associate Editor and two referees for their careful reading and very constructive suggestions.
\end{acks}

\begin{funding}
Wegkamp is supported in part by the National Science Foundation grants DMS 2015195 and 
  DMS 2210557. Bing is partially supported by a discovery grant from the Natural Sciences and Engineering Research Council of Canada.
\end{funding}

\begin{supplement}
\stitle{Supplement to  “OPTIMAL DISCRIMINANT ANALYSIS IN
HIGH-DIMENSIONAL LATENT FACTOR MODELS” }
\sdescription{Appendices A and B contain the main proofs for the results in Sections 2 – 5 and 8.
Technical lemmas and auxiliary lemmas are collected in Appendices
C and D. Appendix E contains additional simulation results. }
\end{supplement}

\bibliographystyle{imsart-nameyear.bst}

\bibliography{ref}

\begin{thebibliography}{56}
% BibTex style file: imsart-nameyear.bst, 2017-11-03
% Default style options (sort=1,type=nameyear).
% Used options (sort=1,type=nameyear).

\bibitem[\protect\citeauthoryear{Abramovich and Pensky}{2019}]{AP19}
\begin{barticle}[author]
\bauthor{\bsnm{Abramovich},~\bfnm{Felix}\binits{F.}} \AND
  \bauthor{\bsnm{Pensky},~\bfnm{Marianna}\binits{M.}}
(\byear{2019}).
\btitle{Classification with many classes: Challenges and pluses}.
\bjournal{Journal of Multivariate Analysis}
\bvolume{174}
\bpages{104536}.
\bdoi{10.1016/j.jmva.2019.104536}
\end{barticle}
\endbibitem

\bibitem[\protect\citeauthoryear{Alon et~al.}{1999}]{Alon1999}
\begin{barticle}[author]
\bauthor{\bsnm{Alon},~\bfnm{U.}\binits{U.}},
  \bauthor{\bsnm{Barkai},~\bfnm{N.}\binits{N.}},
  \bauthor{\bsnm{Notterman},~\bfnm{D.~A.}\binits{D.~A.}},
  \bauthor{\bsnm{Gish},~\bfnm{K.}\binits{K.}},
  \bauthor{\bsnm{Ybarra},~\bfnm{S.}\binits{S.}},
  \bauthor{\bsnm{Mack},~\bfnm{D.}\binits{D.}} \AND
  \bauthor{\bsnm{Levine},~\bfnm{A.~J.}\binits{A.~J.}}
(\byear{1999}).
\btitle{Broad patterns of gene expression revealed by clustering analysis of
  tumor and normal colon tissues probed by oligonucleotide arrays}.
\bjournal{Proceedings of the National Academy of Sciences}
\bvolume{96}
\bpages{6745--6750}.
\bdoi{10.1073/pnas.96.12.6745}
\end{barticle}
\endbibitem

\bibitem[\protect\citeauthoryear{Antoniadis, Lambert-Lacroix and
  Leblanc}{2003}]{antoniadis2003effective}
\begin{barticle}[author]
\bauthor{\bsnm{Antoniadis},~\bfnm{Anestis}\binits{A.}},
  \bauthor{\bsnm{Lambert-Lacroix},~\bfnm{Sophie}\binits{S.}} \AND
  \bauthor{\bsnm{Leblanc},~\bfnm{Fr{\'e}d{\'e}rique}\binits{F.}}
(\byear{2003}).
\btitle{Effective dimension reduction methods for tumor classification using
  gene expression data}.
\bjournal{Bioinformatics}
\bvolume{19}
\bpages{563--570}.
\end{barticle}
\endbibitem

\bibitem[\protect\citeauthoryear{Azizyan, Singh and
  Wasserman}{2013}]{Martin_2013}
\begin{binproceedings}[author]
\bauthor{\bsnm{Azizyan},~\bfnm{Martin}\binits{M.}},
  \bauthor{\bsnm{Singh},~\bfnm{Aarti}\binits{A.}} \AND
  \bauthor{\bsnm{Wasserman},~\bfnm{Larry}\binits{L.}}
(\byear{2013}).
\btitle{Minimax Theory for High-dimensional Gaussian Mixtures with Sparse Mean
  Separation}.
In \bbooktitle{Advances in Neural Information Processing Systems}
(\beditor{\bfnm{C.~J.~C.}\binits{C.~J.~C.}~\bsnm{Burges}},
  \beditor{\bfnm{L.}\binits{L.}~\bsnm{Bottou}},
  \beditor{\bfnm{M.}\binits{M.}~\bsnm{Welling}},
  \beditor{\bfnm{Z.}\binits{Z.}~\bsnm{Ghahramani}} \AND
  \beditor{\bfnm{K.~Q.}\binits{K.~Q.}~\bsnm{Weinberger}}, eds.)
\bvolume{26}.
\bpublisher{Curran Associates, Inc.}
\end{binproceedings}
\endbibitem

\bibitem[\protect\citeauthoryear{Bai and Li}{2012}]{bai2012}
\begin{barticle}[author]
\bauthor{\bsnm{Bai},~\bfnm{Jushan}\binits{J.}} \AND
  \bauthor{\bsnm{Li},~\bfnm{Kunpeng}\binits{K.}}
(\byear{2012}).
\btitle{Statistical analysis of factor models of high dimension}.
\bjournal{Ann. Statist.}
\bvolume{40}
\bpages{436--465}.
\bdoi{10.1214/11-AOS966}
\end{barticle}
\endbibitem

\bibitem[\protect\citeauthoryear{Bai and Ng}{2008}]{Bai-Ng-forecast}
\begin{barticle}[author]
\bauthor{\bsnm{Bai},~\bfnm{Jushan}\binits{J.}} \AND
  \bauthor{\bsnm{Ng},~\bfnm{Serena}\binits{S.}}
(\byear{2008}).
\btitle{Forecasting economic time series using targeted predictors}.
\bjournal{Journal of Econometrics}
\bvolume{146}
\bpages{304 - 317}.
\bnote{Honoring the research contributions of Charles R. Nelson}.
\end{barticle}
\endbibitem

\bibitem[\protect\citeauthoryear{Bair et~al.}{2006}]{Bair_JASA}
\begin{barticle}[author]
\bauthor{\bsnm{Bair},~\bfnm{Eric}\binits{E.}},
  \bauthor{\bsnm{Hastie},~\bfnm{Trevor}\binits{T.}},
  \bauthor{\bsnm{Paul},~\bfnm{Debashis}\binits{D.}} \AND
  \bauthor{\bsnm{Tibshirani},~\bfnm{Robert}\binits{R.}}
(\byear{2006}).
\btitle{Prediction by Supervised Principal Components}.
\bjournal{Journal of the American Statistical Association}
\bvolume{101}
\bpages{119-137}.
\end{barticle}
\endbibitem

\bibitem[\protect\citeauthoryear{Barker and Rayens}{2003}]{barker2003partial}
\begin{barticle}[author]
\bauthor{\bsnm{Barker},~\bfnm{Matthew}\binits{M.}} \AND
  \bauthor{\bsnm{Rayens},~\bfnm{William}\binits{W.}}
(\byear{2003}).
\btitle{Partial least squares for discrimination}.
\bjournal{Journal of Chemometrics: A Journal of the Chemometrics Society}
\bvolume{17}
\bpages{166--173}.
\end{barticle}
\endbibitem

\bibitem[\protect\citeauthoryear{Barshan et~al.}{2011}]{barshan2011supervised}
\begin{barticle}[author]
\bauthor{\bsnm{Barshan},~\bfnm{Elnaz}\binits{E.}},
  \bauthor{\bsnm{Ghodsi},~\bfnm{Ali}\binits{A.}},
  \bauthor{\bsnm{Azimifar},~\bfnm{Zohreh}\binits{Z.}} \AND
  \bauthor{\bsnm{Jahromi},~\bfnm{Mansoor~Zolghadri}\binits{M.~Z.}}
(\byear{2011}).
\btitle{Supervised principal component analysis: Visualization, classification
  and regression on subspaces and submanifolds}.
\bjournal{Pattern Recognition}
\bvolume{44}
\bpages{1357--1371}.
\end{barticle}
\endbibitem

\bibitem[\protect\citeauthoryear{Belloni et~al.}{2012}]{belloni2012sparse}
\begin{barticle}[author]
\bauthor{\bsnm{Belloni},~\bfnm{Alexandre}\binits{A.}},
  \bauthor{\bsnm{Chen},~\bfnm{Daniel}\binits{D.}},
  \bauthor{\bsnm{Chernozhukov},~\bfnm{Victor}\binits{V.}} \AND
  \bauthor{\bsnm{Hansen},~\bfnm{Christian}\binits{C.}}
(\byear{2012}).
\btitle{Sparse models and methods for optimal instruments with an application
  to eminent domain}.
\bjournal{Econometrica}
\bvolume{80}
\bpages{2369--2429}.
\end{barticle}
\endbibitem

\bibitem[\protect\citeauthoryear{Biau, Bunea and Wegkamp}{2003}]{BBW2003}
\begin{barticle}[author]
\bauthor{\bsnm{Biau},~\bfnm{G\'erard}\binits{G.}},
  \bauthor{\bsnm{Bunea},~\bfnm{Florentina}\binits{F.}} \AND
  \bauthor{\bsnm{Wegkamp},~\bfnm{Marten~H.}\binits{M.~H.}}
(\byear{2003}).
\btitle{Functional classification in Hilbert spaces}.
\bjournal{IEEE Transactions on Information Theory}
\bvolume{11}
\bpages{1045 -- 1076}.
\end{barticle}
\endbibitem

\bibitem[\protect\citeauthoryear{Bing and Wegkamp}{2019}]{rank19}
\begin{barticle}[author]
\bauthor{\bsnm{Bing},~\bfnm{Xin}\binits{X.}} \AND
  \bauthor{\bsnm{Wegkamp},~\bfnm{Marten~H.}\binits{M.~H.}}
(\byear{2019}).
\btitle{Adaptive estimation of the rank of the coefficient matrix in
  high-dimensional multivariate response regression models}.
\bjournal{Ann. Statist.}
\bvolume{47}
\bpages{3157--3184}.
\bdoi{10.1214/18-AOS1774}
\end{barticle}
\endbibitem

\bibitem[\protect\citeauthoryear{Bing and Wegkamp}{2022}]{BW22}
\begin{barticle}[author]
\bauthor{\bsnm{Bing},~\bfnm{Xin}\binits{X.}} \AND
  \bauthor{\bsnm{Wegkamp},~\bfnm{Marten}\binits{M.}}
(\byear{2022}).
\btitle{Interpolating Discriminant Functions in High-Dimensional Gaussian
  Latent Mixtures}.
\bjournal{arXiv:2210.14347}.
\end{barticle}
\endbibitem

\bibitem[\protect\citeauthoryear{Bing et~al.}{2020}]{LOVE}
\begin{barticle}[author]
\bauthor{\bsnm{Bing},~\bfnm{Xin}\binits{X.}},
  \bauthor{\bsnm{Bunea},~\bfnm{Florentina}\binits{F.}},
  \bauthor{\bsnm{Ning},~\bfnm{Yang}\binits{Y.}} \AND
  \bauthor{\bsnm{Wegkamp},~\bfnm{Marten}\binits{M.}}
(\byear{2020}).
\btitle{Adaptive estimation in structured factor models with applications to
  overlapping clustering}.
\bjournal{The Annals of Statistics}
\bvolume{48}
\bpages{2055--2081}.
\end{barticle}
\endbibitem

\bibitem[\protect\citeauthoryear{Bing et~al.}{2021}]{bing2020prediction}
\begin{barticle}[author]
\bauthor{\bsnm{Bing},~\bfnm{Xin}\binits{X.}},
  \bauthor{\bsnm{Bunea},~\bfnm{Florentina}\binits{F.}},
  \bauthor{\bsnm{Strimas-Mackey},~\bfnm{Seth}\binits{S.}} \AND
  \bauthor{\bsnm{Wegkamp},~\bfnm{Marten}\binits{M.}}
(\byear{2021}).
\btitle{Prediction Under Latent Factor Regression: Adaptive PCR, Interpolating
  Predictors and Beyond}.
\bjournal{Journal of Machine Learning Research}
\bvolume{22}
\bpages{1-50}.
\end{barticle}
\endbibitem

\bibitem[\protect\citeauthoryear{Boulesteix}{2004}]{boulesteix2004pls}
\begin{barticle}[author]
\bauthor{\bsnm{Boulesteix},~\bfnm{Anne-Laure}\binits{A.-L.}}
(\byear{2004}).
\btitle{PLS dimension reduction for classification with microarray data}.
\bjournal{Statistical applications in genetics and molecular biology}
\bvolume{3}.
\end{barticle}
\endbibitem

\bibitem[\protect\citeauthoryear{Cai and Liu}{2011}]{CaiLiu2011}
\begin{barticle}[author]
\bauthor{\bsnm{Cai},~\bfnm{Tony}\binits{T.}} \AND
  \bauthor{\bsnm{Liu},~\bfnm{Weidong}\binits{W.}}
(\byear{2011}).
\btitle{A Direct Estimation Approach to Sparse Linear Discriminant Analysis}.
\bjournal{Journal of the American Statistical Association}
\bvolume{106}
\bpages{1566-1577}.
\bdoi{10.1198/jasa.2011.tm11199}
\end{barticle}
\endbibitem

\bibitem[\protect\citeauthoryear{Cai and Zhang}{2019a}]{caizhang2019}
\begin{barticle}[author]
\bauthor{\bsnm{Cai},~\bfnm{Tony}\binits{T.}} \AND
  \bauthor{\bsnm{Zhang},~\bfnm{Linjun}\binits{L.}}
(\byear{2019}a).
\btitle{High dimensional linear discriminant analysis: optimality, adaptive
  algorithm and missing data}.
\bjournal{Journal of the Royal Statistical Society: Series B (Statistical
  Methodology)}
\bvolume{81}
\bpages{675-705}.
\bdoi{10.1111/rssb.12326}
\end{barticle}
\endbibitem

\bibitem[\protect\citeauthoryear{Cai and Zhang}{2019b}]{cai2019convex}
\begin{bmisc}[author]
\bauthor{\bsnm{Cai},~\bfnm{T.~Tony}\binits{T.~T.}} \AND
  \bauthor{\bsnm{Zhang},~\bfnm{Linjun}\binits{L.}}
(\byear{2019}b).
\btitle{A Convex Optimization Approach to High-Dimensional Sparse Quadratic
  Discriminant Analysis}.
\end{bmisc}
\endbibitem

\bibitem[\protect\citeauthoryear{Chernozhukov et~al.}{2018}]{DDML}
\begin{barticle}[author]
\bauthor{\bsnm{Chernozhukov},~\bfnm{Victor}\binits{V.}},
  \bauthor{\bsnm{Chetverikov},~\bfnm{Denis}\binits{D.}},
  \bauthor{\bsnm{Demirer},~\bfnm{Mert}\binits{M.}},
  \bauthor{\bsnm{Duflo},~\bfnm{Esther}\binits{E.}},
  \bauthor{\bsnm{Hansen},~\bfnm{Christian}\binits{C.}},
  \bauthor{\bsnm{Newey},~\bfnm{Whitney}\binits{W.}} \AND
  \bauthor{\bsnm{Robins},~\bfnm{James}\binits{J.}}
(\byear{2018}).
\btitle{{Double/debiased machine learning for treatment and structural
  parameters}}.
\bjournal{The Econometrics Journal}
\bvolume{21}
\bpages{C1-C68}.
\bdoi{10.1111/ectj.12097}
\end{barticle}
\endbibitem

\bibitem[\protect\citeauthoryear{Chiaromonte and
  Martinelli}{2002}]{chiaromonte2002dimension}
\begin{barticle}[author]
\bauthor{\bsnm{Chiaromonte},~\bfnm{Francesca}\binits{F.}} \AND
  \bauthor{\bsnm{Martinelli},~\bfnm{Jessica}\binits{J.}}
(\byear{2002}).
\btitle{Dimension reduction strategies for analyzing global gene expression
  data with a response}.
\bjournal{Mathematical Biosciences}
\bvolume{176}
\bpages{123--144}.
\end{barticle}
\endbibitem

\bibitem[\protect\citeauthoryear{Clemmensen
  et~al.}{2011}]{clemmensen2011sparse}
\begin{barticle}[author]
\bauthor{\bsnm{Clemmensen},~\bfnm{Line}\binits{L.}},
  \bauthor{\bsnm{Hastie},~\bfnm{Trevor}\binits{T.}},
  \bauthor{\bsnm{Witten},~\bfnm{Daniela}\binits{D.}} \AND
  \bauthor{\bsnm{Ersb{\o}ll},~\bfnm{Bjarne}\binits{B.}}
(\byear{2011}).
\btitle{Sparse discriminant analysis}.
\bjournal{Technometrics}
\bvolume{53}
\bpages{406--413}.
\end{barticle}
\endbibitem

\bibitem[\protect\citeauthoryear{Dai, Lieu and Rocke}{2006}]{dai2006dimension}
\begin{barticle}[author]
\bauthor{\bsnm{Dai},~\bfnm{Jian~J}\binits{J.~J.}},
  \bauthor{\bsnm{Lieu},~\bfnm{Linh}\binits{L.}} \AND
  \bauthor{\bsnm{Rocke},~\bfnm{David}\binits{D.}}
(\byear{2006}).
\btitle{Dimension reduction for classification with gene expression microarray
  data}.
\bjournal{Statistical applications in genetics and molecular biology}
\bvolume{5}.
\end{barticle}
\endbibitem

\bibitem[\protect\citeauthoryear{Dettling}{2004}]{Dettling2004}
\begin{barticle}[author]
\bauthor{\bsnm{Dettling},~\bfnm{Marcel}\binits{M.}}
(\byear{2004}).
\btitle{{BagBoosting for tumor classification with gene expression data}}.
\bjournal{Bioinformatics}
\bvolume{20}
\bpages{3583-3593}.
\bdoi{10.1093/bioinformatics/bth447}
\end{barticle}
\endbibitem

\bibitem[\protect\citeauthoryear{Devroye, Gy\"orfi and Lugosi}{1996}]{DGL}
\begin{bbook}[author]
\bauthor{\bsnm{Devroye},~\bfnm{Luc}\binits{L.}},
  \bauthor{\bsnm{Gy\"orfi},~\bfnm{L\'aszl\'o}\binits{L.}} \AND
  \bauthor{\bsnm{Lugosi},~\bfnm{G\'abor}\binits{G.}}
(\byear{1996}).
\btitle{A Probabilistic Theory of Pattern Recognition}.
\bpublisher{Springer}.
\end{bbook}
\endbibitem

\bibitem[\protect\citeauthoryear{Fan and Fan}{2008}]{FanFan2008}
\begin{barticle}[author]
\bauthor{\bsnm{Fan},~\bfnm{Jianqing}\binits{J.}} \AND
  \bauthor{\bsnm{Fan},~\bfnm{Yingying}\binits{Y.}}
(\byear{2008}).
\btitle{{High-dimensional classification using features annealed independence
  rules}}.
\bjournal{The Annals of Statistics}
\bvolume{36}
\bpages{2605 -- 2637}.
\bdoi{10.1214/07-AOS504}
\end{barticle}
\endbibitem

\bibitem[\protect\citeauthoryear{Fan, Liao and Mincheva}{2013}]{fan2013large}
\begin{barticle}[author]
\bauthor{\bsnm{Fan},~\bfnm{Jianqing}\binits{J.}},
  \bauthor{\bsnm{Liao},~\bfnm{Yuan}\binits{Y.}} \AND
  \bauthor{\bsnm{Mincheva},~\bfnm{Martina}\binits{M.}}
(\byear{2013}).
\btitle{Large covariance estimation by thresholding principal orthogonal
  complements}.
\bjournal{Journal of the Royal Statistical Society: Series B (Statistical
  Methodology)}
\bvolume{75}
\bpages{603--680}.
\end{barticle}
\endbibitem

\bibitem[\protect\citeauthoryear{Fan, Xue and Yao}{2017}]{fan2017}
\begin{barticle}[author]
\bauthor{\bsnm{Fan},~\bfnm{Jianqing}\binits{J.}},
  \bauthor{\bsnm{Xue},~\bfnm{Lingzhou}\binits{L.}} \AND
  \bauthor{\bsnm{Yao},~\bfnm{Jiawei}\binits{J.}}
(\byear{2017}).
\btitle{Sufficient forecasting using factor models}.
\bjournal{Journal of Econometrics}
\bvolume{201}
\bpages{292 - 306}.
\end{barticle}
\endbibitem

\bibitem[\protect\citeauthoryear{Ghosh}{2001}]{ghosh2001singular}
\begin{bincollection}[author]
\bauthor{\bsnm{Ghosh},~\bfnm{Debashis}\binits{D.}}
(\byear{2001}).
\btitle{Singular value decomposition regression models for classification of
  tumors from microarray experiments}.
In \bbooktitle{Biocomputing 2002}
\bpages{18--29}.
\bpublisher{World Scientific}.
\end{bincollection}
\endbibitem

\bibitem[\protect\citeauthoryear{Hadef and Djebabra}{2019}]{hadef2019proposal}
\begin{barticle}[author]
\bauthor{\bsnm{Hadef},~\bfnm{Hafaidh}\binits{H.}} \AND
  \bauthor{\bsnm{Djebabra},~\bfnm{M{\'e}barek}\binits{M.}}
(\byear{2019}).
\btitle{Proposal method for the classification of industrial accident scenarios
  based on the improved principal components analysis (improved PCA)}.
\bjournal{Production Engineering}
\bvolume{13}
\bpages{53--60}.
\end{barticle}
\endbibitem

\bibitem[\protect\citeauthoryear{Hahn, Carvalho and
  Mukherjee}{2013}]{Partial_Factor_Modeling}
\begin{barticle}[author]
\bauthor{\bsnm{Hahn},~\bfnm{P.~Richard}\binits{P.~R.}},
  \bauthor{\bsnm{Carvalho},~\bfnm{Carlos~M.}\binits{C.~M.}} \AND
  \bauthor{\bsnm{Mukherjee},~\bfnm{Sayan}\binits{S.}}
(\byear{2013}).
\btitle{Partial Factor Modeling: Predictor-Dependent Shrinkage for Linear
  Regression}.
\bjournal{Journal of the American Statistical Association}
\bvolume{108}
\bpages{999-1008}.
\bdoi{10.1080/01621459.2013.779843}
\end{barticle}
\endbibitem

\bibitem[\protect\citeauthoryear{Hastie, Buja and
  Tibshirani}{1995}]{hastie1995penalized}
\begin{barticle}[author]
\bauthor{\bsnm{Hastie},~\bfnm{Trevor}\binits{T.}},
  \bauthor{\bsnm{Buja},~\bfnm{Andreas}\binits{A.}} \AND
  \bauthor{\bsnm{Tibshirani},~\bfnm{Robert}\binits{R.}}
(\byear{1995}).
\btitle{Penalized discriminant analysis}.
\bjournal{The Annals of Statistics}
\bvolume{23}
\bpages{73--102}.
\end{barticle}
\endbibitem

\bibitem[\protect\citeauthoryear{Hastie, Tibshirani and
  Friedman}{2009}]{hastie_09_esl}
\begin{bbook}[author]
\bauthor{\bsnm{Hastie},~\bfnm{Trevor}\binits{T.}},
  \bauthor{\bsnm{Tibshirani},~\bfnm{Robert}\binits{R.}} \AND
  \bauthor{\bsnm{Friedman},~\bfnm{Jerome}\binits{J.}}
(\byear{2009}).
\btitle{The elements of statistical learning: data mining, inference and
  prediction},
\bedition{2} ed.
\bpublisher{Springer}.
\end{bbook}
\endbibitem

\bibitem[\protect\citeauthoryear{Hotelling}{1957}]{Hotelling}
\begin{barticle}[author]
\bauthor{\bsnm{Hotelling},~\bfnm{Harold}\binits{H.}}
(\byear{1957}).
\btitle{The relations of the newer multivariate statistical methods to factor
  analysis}.
\bjournal{British Journal of Statistical Psychology}
\bvolume{10}
\bpages{69-79}.
\end{barticle}
\endbibitem

\bibitem[\protect\citeauthoryear{Hsu, Kakade and Zhang}{2014}]{Hsu2014}
\begin{barticle}[author]
\bauthor{\bsnm{Hsu},~\bfnm{Daniel}\binits{D.}},
  \bauthor{\bsnm{Kakade},~\bfnm{Sham~M.}\binits{S.~M.}} \AND
  \bauthor{\bsnm{Zhang},~\bfnm{Tong}\binits{T.}}
(\byear{2014}).
\btitle{Random Design Analysis of Ridge Regression}.
\bjournal{Found. Comput. Math.}
\bvolume{14}
\bpages{569--600}.
\bdoi{10.1007/s10208-014-9192-1}
\end{barticle}
\endbibitem

\bibitem[\protect\citeauthoryear{Izenman}{2008}]{Izenman-book}
\begin{bbook}[author]
\bauthor{\bsnm{Izenman},~\bfnm{Alan~Julian}\binits{A.~J.}}
(\byear{2008}).
\btitle{Modern Multivariate Statistical Techniques: Regression, Classification,
  and Manifold Learning}.
\bpublisher{Series: Springer Texts in Statistics}.
\end{bbook}
\endbibitem

\bibitem[\protect\citeauthoryear{Jin et~al.}{2021}]{jin2021classification}
\begin{barticle}[author]
\bauthor{\bsnm{Jin},~\bfnm{Dan}\binits{D.}},
  \bauthor{\bsnm{Henry},~\bfnm{Philippe}\binits{P.}},
  \bauthor{\bsnm{Shan},~\bfnm{Jacqueline}\binits{J.}} \AND
  \bauthor{\bsnm{Chen},~\bfnm{Jie}\binits{J.}}
(\byear{2021}).
\btitle{Classification of cannabis strains in the Canadian market with
  discriminant analysis of principal components using genome-wide single
  nucleotide polymorphisms}.
\bjournal{Plos one}
\bvolume{16}
\bpages{e0253387}.
\end{barticle}
\endbibitem

\bibitem[\protect\citeauthoryear{Li}{2016}]{li2016accurate}
\begin{barticle}[author]
\bauthor{\bsnm{Li},~\bfnm{Hailin}\binits{H.}}
(\byear{2016}).
\btitle{Accurate and efficient classification based on common principal
  components analysis for multivariate time series}.
\bjournal{Neurocomputing}
\bvolume{171}
\bpages{744--753}.
\end{barticle}
\endbibitem

\bibitem[\protect\citeauthoryear{Ma et~al.}{2020}]{ma2020unsupervised}
\begin{barticle}[author]
\bauthor{\bsnm{Ma},~\bfnm{Zhe}\binits{Z.}},
  \bauthor{\bsnm{Liu},~\bfnm{Zhe}\binits{Z.}},
  \bauthor{\bsnm{Zhao},~\bfnm{Yuanyuan}\binits{Y.}},
  \bauthor{\bsnm{Zhang},~\bfnm{Lin}\binits{L.}},
  \bauthor{\bsnm{Liu},~\bfnm{Diyou}\binits{D.}},
  \bauthor{\bsnm{Ren},~\bfnm{Tianwei}\binits{T.}},
  \bauthor{\bsnm{Zhang},~\bfnm{Xiaodong}\binits{X.}} \AND
  \bauthor{\bsnm{Li},~\bfnm{Shaoming}\binits{S.}}
(\byear{2020}).
\btitle{An unsupervised crop classification method based on principal
  components isometric binning}.
\bjournal{ISPRS International Journal of Geo-Information}
\bvolume{9}
\bpages{648}.
\end{barticle}
\endbibitem

\bibitem[\protect\citeauthoryear{Mai, Yang and Zou}{2019}]{mai2019multiclass}
\begin{barticle}[author]
\bauthor{\bsnm{Mai},~\bfnm{Qing}\binits{Q.}},
  \bauthor{\bsnm{Yang},~\bfnm{Yi}\binits{Y.}} \AND
  \bauthor{\bsnm{Zou},~\bfnm{Hui}\binits{H.}}
(\byear{2019}).
\btitle{Multiclass sparse discriminant analysis}.
\bjournal{Statistica Sinica}
\bvolume{29}
\bpages{97--111}.
\end{barticle}
\endbibitem

\bibitem[\protect\citeauthoryear{Mai, Zou and Yuan}{2012}]{mai2012}
\begin{barticle}[author]
\bauthor{\bsnm{Mai},~\bfnm{Qing}\binits{Q.}},
  \bauthor{\bsnm{Zou},~\bfnm{Hui}\binits{H.}} \AND
  \bauthor{\bsnm{Yuan},~\bfnm{Ming}\binits{M.}}
(\byear{2012}).
\btitle{{A direct approach to sparse discriminant analysis in ultra-high
  dimensions}}.
\bjournal{Biometrika}
\bvolume{99}
\bpages{29-42}.
\end{barticle}
\endbibitem

\bibitem[\protect\citeauthoryear{Mallary et~al.}{2022}]{mallary2022acoustic}
\begin{barticle}[author]
\bauthor{\bsnm{Mallary},~\bfnm{C}\binits{C.}},
  \bauthor{\bsnm{Berg},~\bfnm{CJ}\binits{C.}},
  \bauthor{\bsnm{Buck},~\bfnm{John~R}\binits{J.~R.}},
  \bauthor{\bsnm{Tandon},~\bfnm{Amit}\binits{A.}} \AND
  \bauthor{\bsnm{Andonian},~\bfnm{Alan}\binits{A.}}
(\byear{2022}).
\btitle{Acoustic rainfall detection with linear discriminant functions of
  principal components}.
\bjournal{The Journal of the Acoustical Society of America}
\bvolume{151}
\bpages{A149--A149}.
\end{barticle}
\endbibitem

\bibitem[\protect\citeauthoryear{Nguyen and Rocke}{2002}]{Nguyen2002}
\begin{barticle}[author]
\bauthor{\bsnm{Nguyen},~\bfnm{Danh~V.}\binits{D.~V.}} \AND
  \bauthor{\bsnm{Rocke},~\bfnm{David~M.}\binits{D.~M.}}
(\byear{2002}).
\btitle{{Tumor classification by partial least squares using microarray gene
  expression data }}.
\bjournal{Bioinformatics}
\bvolume{18}
\bpages{39-50}.
\bdoi{10.1093/bioinformatics/18.1.39}
\end{barticle}
\endbibitem

\bibitem[\protect\citeauthoryear{Shao et~al.}{2011}]{Shao2011}
\begin{barticle}[author]
\bauthor{\bsnm{Shao},~\bfnm{Jun}\binits{J.}},
  \bauthor{\bsnm{Wang},~\bfnm{Yazhen}\binits{Y.}},
  \bauthor{\bsnm{Deng},~\bfnm{Xinwei}\binits{X.}} \AND
  \bauthor{\bsnm{Wang},~\bfnm{Sijian}\binits{S.}}
(\byear{2011}).
\btitle{{Sparse linear discriminant analysis by thresholding for high
  dimensional data}}.
\bjournal{The Annals of Statistics}
\bvolume{39}
\bpages{1241 -- 1265}.
\bdoi{10.1214/10-AOS870}
\end{barticle}
\endbibitem

\bibitem[\protect\citeauthoryear{Singh et~al.}{2002}]{Singh2002}
\begin{barticle}[author]
\bauthor{\bsnm{Singh},~\bfnm{Dinesh}\binits{D.}},
  \bauthor{\bsnm{Febbo},~\bfnm{Phillip~G.}\binits{P.~G.}},
  \bauthor{\bsnm{Ross},~\bfnm{Kenneth}\binits{K.}},
  \bauthor{\bsnm{Jackson},~\bfnm{Donald~G.}\binits{D.~G.}},
  \bauthor{\bsnm{Manola},~\bfnm{Judith}\binits{J.}},
  \bauthor{\bsnm{Ladd},~\bfnm{Christine}\binits{C.}},
  \bauthor{\bsnm{Tamayo},~\bfnm{Pablo}\binits{P.}},
  \bauthor{\bsnm{Renshaw},~\bfnm{Andrew~A.}\binits{A.~A.}},
  \bauthor{\bsnm{D'Amico},~\bfnm{Anthony~V.}\binits{A.~V.}},
  \bauthor{\bsnm{Richie},~\bfnm{Jerome~P.}\binits{J.~P.}},
  \bauthor{\bsnm{Lander},~\bfnm{Eric~S.}\binits{E.~S.}},
  \bauthor{\bsnm{Loda},~\bfnm{Massimo}\binits{M.}},
  \bauthor{\bsnm{Kantoff},~\bfnm{Philip~W.}\binits{P.~W.}},
  \bauthor{\bsnm{Golub},~\bfnm{Todd~R.}\binits{T.~R.}} \AND
  \bauthor{\bsnm{Sellers},~\bfnm{William~R.}\binits{W.~R.}}
(\byear{2002}).
\btitle{Gene expression correlates of clinical prostate cancer behavior}.
\bjournal{Cancer Cell}
\bvolume{1}
\bpages{203-209}.
\bdoi{https://doi.org/10.1016/S1535-6108(02)00030-2}
\end{barticle}
\endbibitem

\bibitem[\protect\citeauthoryear{Stock and Watson}{2002a}]{SW2002_JASA}
\begin{barticle}[author]
\bauthor{\bsnm{Stock},~\bfnm{James~H.}\binits{J.~H.}} \AND
  \bauthor{\bsnm{Watson},~\bfnm{Mark~W.}\binits{M.~W.}}
(\byear{2002}a).
\btitle{Forecasting Using Principal Components from a Large Number of
  Predictors}.
\bjournal{Journal of the American Statistical Association}
\bvolume{97}
\bpages{1167--1179}.
\end{barticle}
\endbibitem

\bibitem[\protect\citeauthoryear{Stock and Watson}{2002b}]{SW2002_JB}
\begin{barticle}[author]
\bauthor{\bsnm{Stock},~\bfnm{James~H}\binits{J.~H.}} \AND
  \bauthor{\bsnm{Watson},~\bfnm{Mark~W}\binits{M.~W.}}
(\byear{2002}b).
\btitle{Macroeconomic Forecasting Using Diffusion Indexes}.
\bjournal{Journal of Business \& Economic Statistics}
\bvolume{20}
\bpages{147-162}.
\end{barticle}
\endbibitem

\bibitem[\protect\citeauthoryear{Tarigan and Van~de Geer}{2006}]{TG}
\begin{barticle}[author]
\bauthor{\bsnm{Tarigan},~\bfnm{Bernadetta}\binits{B.}} \AND
  \bauthor{\bparticle{Van~de} \bsnm{Geer},~\bfnm{Sara}\binits{S.}}
(\byear{2006}).
\btitle{Classifiers of support vector machine type with $\ell_1$ complexity
  regularization}.
\bjournal{Bernoulli}
\bvolume{12}
\bpages{1045 -- 1076}.
\end{barticle}
\endbibitem

\bibitem[\protect\citeauthoryear{Tibshirani et~al.}{2002}]{Tibshirani2002}
\begin{barticle}[author]
\bauthor{\bsnm{Tibshirani},~\bfnm{Robert}\binits{R.}},
  \bauthor{\bsnm{Hastie},~\bfnm{Trevor}\binits{T.}},
  \bauthor{\bsnm{Narasimhan},~\bfnm{Balasubramanian}\binits{B.}} \AND
  \bauthor{\bsnm{Chu},~\bfnm{Gilbert}\binits{G.}}
(\byear{2002}).
\btitle{Diagnosis of multiple cancer types by shrunken centroids of gene
  expression}.
\bjournal{Proceedings of the National Academy of Sciences}
\bvolume{99}
\bpages{6567--6572}.
\bdoi{10.1073/pnas.082099299}
\end{barticle}
\endbibitem

\bibitem[\protect\citeauthoryear{Tsybakov}{2004}]{tsybakov2004optimal}
\begin{barticle}[author]
\bauthor{\bsnm{Tsybakov},~\bfnm{Alexander~B}\binits{A.~B.}}
(\byear{2004}).
\btitle{Optimal aggregation of classifiers in statistical learning}.
\bjournal{The Annals of Statistics}
\bvolume{32}
\bpages{135--166}.
\end{barticle}
\endbibitem

\bibitem[\protect\citeauthoryear{{Tsybakov}}{2009}]{Intro_non_para}
\begin{binbook}[author]
\bauthor{\bsnm{{Tsybakov}},~\bfnm{A.~B.}\binits{A.~B.}}
(\byear{2009}).
\btitle{Introduction to nonparametric estimation}.
\bpublisher{Springer Series in Statistics. Springer, New York}.
\end{binbook}
\endbibitem

\bibitem[\protect\citeauthoryear{Vershynin}{2012}]{vershynin_2012}
\begin{binbook}[author]
\bauthor{\bsnm{Vershynin},~\bfnm{Roman}\binits{R.}}
(\byear{2012}).
\btitle{Introduction to the non-asymptotic analysis of random matrices}
In \bbooktitle{Compressed Sensing: Theory and Applications}
\bpages{210 -- 268}.
\bpublisher{Cambridge University Press}.
\end{binbook}
\endbibitem

\bibitem[\protect\citeauthoryear{Vu and Lei}{2013}]{vu2013minimax}
\begin{barticle}[author]
\bauthor{\bsnm{Vu},~\bfnm{Vincent~Q}\binits{V.~Q.}} \AND
  \bauthor{\bsnm{Lei},~\bfnm{Jing}\binits{J.}}
(\byear{2013}).
\btitle{Minimax sparse principal subspace estimation in high dimensions}.
\bjournal{The Annals of Statistics}
\bvolume{41}
\bpages{2905--2947}.
\end{barticle}
\endbibitem

\bibitem[\protect\citeauthoryear{Wegkamp and Yuan}{2011}]{WY}
\begin{barticle}[author]
\bauthor{\bsnm{Wegkamp},~\bfnm{Marten}\binits{M.}} \AND
  \bauthor{\bsnm{Yuan},~\bfnm{Ming}\binits{M.}}
(\byear{2011}).
\btitle{Support vector machines with a reject option}.
\bjournal{Bernoulli}
\bvolume{17}
\bpages{1368 -- 1385}.
\end{barticle}
\endbibitem

\bibitem[\protect\citeauthoryear{Witten and Tibshirani}{2011}]{Witten2011}
\begin{barticle}[author]
\bauthor{\bsnm{Witten},~\bfnm{Daniela~M.}\binits{D.~M.}} \AND
  \bauthor{\bsnm{Tibshirani},~\bfnm{Robert}\binits{R.}}
(\byear{2011}).
\btitle{Penalized classification using Fisher's linear discriminant}.
\bjournal{Journal of the Royal Statistical Society: Series B (Statistical
  Methodology)}
\bvolume{73}
\bpages{753-772}.
\bdoi{https://doi.org/10.1111/j.1467-9868.2011.00783.x}
\end{barticle}
\endbibitem

\bibitem[\protect\citeauthoryear{Yu, Wang and Samworth}{2014}]{yu2014}
\begin{barticle}[author]
\bauthor{\bsnm{Yu},~\bfnm{Y.}\binits{Y.}},
  \bauthor{\bsnm{Wang},~\bfnm{T.}\binits{T.}} \AND
  \bauthor{\bsnm{Samworth},~\bfnm{R.~J.}\binits{R.~J.}}
(\byear{2014}).
\btitle{{A useful variant of the Davis–Kahan theorem for statisticians}}.
\bjournal{Biometrika}
\bvolume{102}
\bpages{315-323}.
\bdoi{10.1093/biomet/asv008}
\end{barticle}
\endbibitem

\end{thebibliography}

\appendix

We first provide in Appendix \ref{app_sec_main_proofs}, section-by-section, the main proofs for the results in Sections  \ref{sec_excess_risk_lower_bound} --  \ref{sec_theory_application} and \ref{sec_multi_level}, except Theorem \ref{thm_lowerbound}.
The proof of our minimax lower bounds in Theorem \ref{thm_lowerbound} is stated separately in Appendix \ref{app_proof_thm_lowerbound}. 
Technical lemmas and auxiliary lemmas are collected in Appendices \ref{app_tech} and  \ref{app_aux}, respectively. 
Appendix E contains additional simulation results.

\section{Main proofs}\label{app_sec_main_proofs}

\subsection{Proofs of Section \ref{sec_excess_risk_lower_bound}}

\subsubsection{Proof of Lemma \ref{lem:RxvsRz}}\label{app_proof_RxvsRz}
We observe that 
\begin{align}\label{RxvsRz}
	R_x^* & := \inf_{g} \PP\{ g(AZ+W) \ne Y \} \nonumber\\
	&\ge \EE_W \inf_{g} \PP\{ g(AZ+W) \ne Y \mid  W\} \nonumber\\
	&\ge \EE_W  \inf_{h} \PP\{ h(Z) \ne Y \}  \\
	&= \inf_{h} \PP\{ h(Z) \ne Y \}\nonumber \\
	&:= R_z^*.\nonumber
\end{align}  
In the derivation (\ref{RxvsRz}) above, the infima are taken over all measurable functions $g:\RR^p\to\{0,1\}$ and $h:\RR^K\to\{0,1\}$, and note that the second inequality uses the independence between $W$ and $(Y,Z)$. \qed

\subsubsection{Proof of Lemma \ref{lem_risk}}\label{app_proof_benchmark} 
We define
\begin{equation}\label{def_Dt_x}
	\Dt_x^2 := (\a_1 - \a_0)^\T A^\T (A\szy A^\T + \sw)^{-1}A (\a_1 - \a_0).
\end{equation}
From standard LDA theory \citep[pp 241-244]{Izenman-book},
\[
R_x^* = 1 - \pi_1\Phi\left({\Dt_x\over 2}+ {\log {\pi_1 \over \pi_0} \over \Dt_x}\right) - \pi_0\Phi\left({\Dt_x\over 2}- {\log {\pi_1 \over \pi_0} \over \Dt_x}\right)
\]
which simplifies 
for $\pi_0 = \pi_1$ to
$
R_x^*= 1 - \Phi\left(\Dt_x/2\right)
$. 
Hence, we have 
\begin{align*}
	R_x^* - R_z^* &=  \Phi\left({\Dt\over 2} \right) -  \Phi\left({\Dt_x\over 2}\right).
\end{align*}
Since, by an application of the Woodbury identity,
\begin{align}\label{eq_diff_Dt_Dtx}\nonumber
	\Dt^2 - \Dt_x^2 &= (\a_1 -\a_0)^\T\left[
	\szy^{-1} - A^\T (A\szy A^\T + \sw)^{-1}A
	\right](\a_1 - \a_0)\\
	&= (\a_1 -\a_0)^\T \szy^{-1/2}\left(
	\bI_K + \szy^{1/2}A^\T \sw^{-1}A \szy^{1/2}
	\right)^{-1}\szy^{-1/2}(\a_1 - \a_0)
\end{align}	
we have
\begin{equation}\label{equiv_Dt_Dtx}
	\Dt \ge \Dt_x,\qquad \Dt^2-\Dt_x^2 \le {\Dt^2 \over 1 + \lambda_K(H)}
\end{equation}
with $H = \szy^{1/2}A^\T \sw^{-1}A\szy^{1/2}$.
Since $$\lambda_K(H) \ge { \lambda_K(A\szy A^\T) \over \lambda_1(\sw)} \overset{(\ref{def_xi_star})}{=}\xi^*, $$
and the function $x\mapsto x/(1+x)$ is increasing for $x>0$, display (\ref{equiv_Dt_Dtx}) further implies that 
\begin{align}\label{eq_diff_Dt_Dtx2}
	\Dt^2 \ge \Dt_x^2 \ge \Dt^2 {\lambda_K(H) \over 1+\lambda_K(H)} \ge \Dt^2 {\xi^* \over 1+\xi^*}.\end{align}
Finally,
using the mean value theorem, we find
\begin{align*} 
	R_x^* - R_z^*
	& \le {1\over 2}\left(\Dt - \Dt_x\right)\varphi\left({\Dt_x\over 2}\right)
	= {1\over 2}{\Dt^2 - \Dt_x^2 \over \Dt + \Dt_x}\varphi\left({\Dt_x\over 2}\right)\\
	&\le {1\over 2\sqrt{2\pi}}\cdot {\Dt \over 1 + \lambda_{K}(H)}\exp\left\{-{\Dt_x^2/ 8}\right\}\\
	&\le {1\over 2\sqrt{2\pi}}\cdot {\Dt \over 1 + \xi^*}\exp\left\{-{\xi^* \over 8(1+\xi^*)}\Dt^2\right\}.
\end{align*}
Our claim of the upper bound thus follows from 
$\xi^* \asymp \lambda / \sigma^2$ for any $\theta \in \Theta(\lambda, \sigma, \lambda)$. 

To prove the lower bound of $R_x^*-  R_z^*$,  note that, by display (\ref{eq_diff_Dt_Dtx}),
\[
\Dt^2 - \Dt_x^2 \ge {\|\a_1-\a_0\|_{\szy}^2 \over 1 + \lambda_1(H)} = {\Dt^2 \over 1 + \lambda_1(H)}.
\]
This implies 
\[
\Dt_x^2 \le {\lambda_1(H) \over 1 + \lambda_1(H)}\Dt^2.
\] 
Similarly, by the mean value theorem and $\Dt \ge \Dt_x$ from (\ref{equiv_Dt_Dtx}), 
\begin{align*}
	R_x^* - R_z^* & = \Phi\left({\Dt\over 2} \right) -  \Phi\left({\Dt_x\over 2}\right)\\
	& \ge {1\over 2}\left(\Dt - \Dt_x\right)\varphi\left({\Dt\over 2}\right)
	= {1\over 2}{\Dt^2 - \Dt_x^2 \over \Dt + \Dt_x}\varphi\left({\Dt\over 2}\right)\\
	&\ge {1\over 2\sqrt{2\pi}}\cdot {\Dt^2 \over \Dt+\Dt_x}{1 \over 1 + \lambda_{1}(H)}\exp\left\{-{\Dt^2/ 8}\right\}\\
	&\ge {1\over 4\sqrt{2\pi}}\cdot  {\Dt \over 1 + \lambda_1(H)}\exp\left\{-\Dt^2/8\right\}.
\end{align*}
The result  follows from this inequality and $\lambda_1(H)\asymp \lambda/\sigma^2$ for any $\theta \in \Theta(\lambda, \sigma, \Dt)$. \qed

\subsection{Proof of Proposition \ref{prop_ls_rule}}\label{app_proof_sec_method}

We prove Proposition \ref{prop_ls_rule} by proving the following more general result. Define, for  any scalar $a> 0$, 	
\begin{align}\label{eq_beta_a}
	\beta^a &= a ~  \sz^{-1}(\a_1 - \a_0),\\ 
	\beta_0^a &= -{1\over 2}(\a_0+\a_1)^\T \beta^a + \left[
	a - \pi_0\pi_1(\a_1-\a_0)^\T \beta^a
	\right]\log{\pi_1 \over \pi_0}.\nonumber
\end{align} 
\begin{lemma}\label{lem_rule}
	Let $\eta,\eta_0$ and $\beta^a,\beta_0^a$ be defined in (\ref{rule_Z}) and (\ref{eq_beta_a}), respectively. Under model (\ref{model_X}) and (\ref{model_YZ}) and Assumption {\rm (iv)},
	for any $a > 0$,  we have
	\[
	z^\T \eta+\eta_0 \ge 0 \quad \iff\quad z^\T \beta^a +\beta_0^a\ge0.
	\]
	Furthermore, the parameters $\beta:=\beta^a$  and $\beta_0:= \beta^a_0$ defined in (\ref{eq_beta_a}) with $a = \pi_0\pi_1$ satisfies 
	\[
	\beta = \sz^{-1}\Cov(Z, Y)
	\]
	and
	\[
	z^\T \eta+\eta_0 = {1\over \pi_0\pi_1[1 - (\a_1-\a_0)^\T \beta]} (z^\T \beta + \beta_0).
	\]
\end{lemma}
\begin{proof}
	To prove the first statement, write
	\begin{equation}\label{def_Gz_star}
		G_z^*(z) := z^\T \eta + \eta_0 = z^\T \eta - {1\over 2}(\a_0+\a_1)^\T \eta + \log {\pi_1 \over \pi_0}.
	\end{equation}
	It suffices to show that, for any $a > 0$,
	\begin{equation}\label{eq_eta_beta}
		\eta = {\beta^a \over a - \pi_0\pi_1 (\a_1-\a_0)^\T \beta^a}
	\end{equation}
	and 
	\begin{equation}\label{eq_pos}
		a - \pi_0\pi_1 (\a_1-\a_0)^\T \beta^a > 0.
	\end{equation}
	
	To show (\ref{eq_pos}), we observe that from Lemma \ref{fact}
	\begin{equation}\label{eq_sz_C}
		\sz =  \szy + \pi_0\pi_1 (\a_1 - \a_0)(\a_1 - \a_0)^\T.
	\end{equation}
	By the Woodbury  
	formula, 
	\begin{align*}
		\sz^{-1}(\a_1 -\a_0)
		&~ =  ~ \szy^{-1}(\a_1 - \a_0) - {\pi_0\pi_1 \|\a_1 -\a_0\|_{\szy}^2\over 1 + \pi_0\pi_1 \|\a_1 -\a_0\|_ {\szy}^2}\szy^{-1}(\a_1 - \a_0)\\
		&\overset{(\ref{def_Dt})}{=}{1\over 1 + \pi_0\pi_1 \Dt^2}\szy^{-1}(\a_1 - \a_0).
	\end{align*}
	This gives 
	\begin{equation}\label{eq_Deltas}
		\|\a_1 - \a_0\|_{\sz}^2 = {\Dt^2 \over 1 + \pi_0\pi_1 \Dt^2}
	\end{equation}
	which implies 
	\begin{equation}\label{eq_Deltas_prime}
		1  - \pi_0\pi_1 \|\a_1 - \a_0\|_{\sz}^2 =  {1\over 1 + \pi_0\pi_1 \Dt^2} > 0.
	\end{equation}
	Hence (\ref{eq_pos}) follows as 
	\[
	a - \pi_0\pi_1 (\a_1-\a_0)^\T \beta^a  = a\left(
	1 -  \pi_0\pi_1 \|\a_1-\a_0\|_{\sz}^2
	\right) = {a \over 1 + \pi_0\pi_1 \Dt^2}.
	\]
	We proceed to show (\ref{eq_eta_beta}). By using (\ref{eq_sz_C}) and the Woodbury 
	formula again, 
	\begin{align*}
		\eta &= \szy^{-1}(\a_1 -\a_0)\\
		&= \sz^{-1}(\a_1 - \a_0) + {\pi_0\pi_1 \|\a_1 -\a_0\|_{\sz}^2\over 1 - \pi_0\pi_1 \|\a_1 -\a_0\|_{\sz}^2}\sz^{-1}(\a_1 - \a_0)\\
		&=\left[1  + {\pi_0\pi_1 \|\a_1 -\a_0\|_{\sz}^2\over 1 - \pi_0\pi_1 \|\a_1 -\a_0\|_{\sz}^2}\right] {\beta^a \over a}\\
		&= {1 \over 1 - \pi_0\pi_1 \|\a_1 -\a_0\|_{\sz}^2}{\beta^a \over a}.
	\end{align*}
	This proves (\ref{eq_eta_beta}) and completes the proof of the first statement. 
	
	To prove the second statement, by the definition of $\beta$ and the choice of $a=\pi_0\pi_1$, we have
	\begin{align*}
		\beta = a~ \sz^{-1}(\a_1 - \a_0) = \sz^{-1}(\a_1 -\a_0) \pi_0\pi_1.
	\end{align*}
	On the other hand,  
	\begin{align*}
		[\Cov(Z)]^{-1}\Cov(Z, Y) 
		&= \sz^{-1} (\EE[ZY] - \EE[Z]\EE[Y])\\
		&=  \sz^{-1} \pi_1(\alpha_1  -  \pi_0 \alpha_0 - \pi_1\a_1 )\\
		&=  \sz^{-1} \pi_0\pi_1(\alpha_1 -\alpha_0),
	\end{align*}
	proving our claim.
	
	The last statement follows immediately from (\ref{eq_eta_beta}) with $a = \pi_0\pi_1$. 
\end{proof}

\subsection{Proofs of Section \ref{sec_theory_general}}

\subsubsection{Proof of Theorem \ref{thm_general_risk}}\label{app_proof_thm_general_risk}
Since $\D = \{\X, \Y\}$ is independent of $(X, Z, W, Y)$, we treat quantities that are only related with $\D$ fixed throughout the proof.
Recall the definitions of $\wh G_x$ and $G_z$ in (\ref{eq_G_hat}).
By definition, 
\begin{align*}
	R_x(\wh g_x) 
	&= \pi_0\PP\left\{
	\wh G_x(X) \ge 0 \mid Y = 0
	\right\} + \pi_1\PP\left\{
	\wh  G_x(X) < 0 \mid Y = 1
	\right\}
\end{align*}
and 
\[
R_z^* = \pi_0\PP\left\{
G_z(Z) \ge 0 \mid Y = 0
\right\} + \pi_1\PP\left\{
G_z(Z) < 0 \mid Y = 1
\right\}.
\]
Recall that $X = AZ+W$ and write $f_{Z|k}(z)$ for the p.d.f. of $N_K(\alpha_k, \szy)$ at the point $z\in \RR^K$ for $k\in \{0,1\}$. We have
\begin{align*}
	&R_x(\wh g_x) - R_z^*\\
	& = \pi_0\EE_W\EE_Z\left[
	\1\{\wh G_x(AZ+w) \ge 0\} - \1\{ G_z(Z) \ge 0\} \mid Y = 0,W=w
	\right]\\
	&\quad + \pi_1\EE_W\EE_Z\left[
	\1\{\wh G_x(AZ+w) < 0\} - \1\{ G_z(Z) < 0\} \mid  Y = 1,W=w
	\right]\\
	&=  \EE_W\int 
	\left(\1\{\wh G_x(Az+w) \ge 0\} - \1\{ G_z(z) \ge 0\} \right)\left(
	\pi_0 f_{Z|0}(z) - \pi_1 f_{Z|1}(z)\right) d z\\
	&= \underbrace{\EE_W\int_{\wh G_x\ge 0,  G_z< 0}\left(
		\pi_0 f_{Z|0}(z) - \pi_1 f_{Z|1}(z)\right) d z}_{(I)}+ \underbrace{\EE_W\int_{\wh G_x <  0,  G_z \ge 0}\left(\pi_1 f_{Z|1}(z) -
		\pi_0 f_{Z|0}(z) \right)d z}_{(II)}.
\end{align*}
The penultimate step uses the assumption that $W$ is independent of both $Z$ and $Y$. Notice that 
\begin{align*}
	\pi_0 f_{Z|0}(z) - \pi_1 f_{Z|1}(z) &=  \pi_0 f_{Z|0}(z)\left[
	1  - {\pi_1 f_{Z|1}(z) \over  \pi_0 f_{Z|0}(z)}\right] = \pi_0 f_{Z|0}(z)\left(1 - \exp\{G_z^*(z)\}\right)
\end{align*}
with 
$$
G_z^*(z) = \log {\pi_1 f_{Z|1}(z) \over \pi_0 f_{Z|0}(z)}  = z^\T \eta + \eta_0 =  {1+\pi_0\pi_1\Dt^2 \over a}G_z(z) := c_* G_z(z)
$$
from Lemma \ref{lem_rule} and (\ref{def_Gz_star}).
This implies the identity
\begin{align*}
	(I) & = \pi_0 \EE_W\EE_Z\left[\1\left\{\wh G_x(AZ+w)\ge 0,  G_z(Z)< 0\right\}\left(1 - \exp\{G_z^*(Z)\}\right) \mid Y=0, W = w\right].
\end{align*}
Define, for any $t\ge 0$,  the event 
\begin{equation}\label{E_t}
	\E_t := \left\{
	|\wh G_x(AZ+W) - G_z(Z) |\le t
	\right\}.
\end{equation}
We obtain 
\begin{align*}
	(I) & = \pi_0\EE_W\EE_Z\left[\1\left\{\wh G_x(AZ+w)\ge 0,  G_z(Z)< 0\right\}\left(1 - \exp\{G_z^*(Z)\}\right) \1\{\E_t\} \mid Y=0, W = w\right]\\
	&\quad +\pi_0\EE_W\EE_Z\left[\1\left\{\wh G_x(AZ+w)\ge 0,  G_z(Z)< 0\right\}\left(1 - \exp\{G_z^*(Z)\}\right) \1\{\E_t^c\} \mid Y=0, W = w\right]\\
	&\le \pi_0 c_*t ~ \EE_Z\left[\1\left\{-t \le G_z(Z)< 0\right\} \mid Y=0\right] +\pi_0\PP(\E_t^c \mid Y = 0).
\end{align*}
In the last step we use the basic inequality $1+x \le \exp(x)$ for all $x\in \RR$ 
and the inequalities  $-t\le G_z(Z) < 0$ and $-G_z^*(Z) \le c_*t$ on the event $\{\wh G_x\ge 0, G_z<0\} \cap \E_t$.

We can bound $(II)$ by analogous arguments using the identity
\[
\pi_1 f_{Z|1}(z) - \pi_0 f_{Z|0}(z) =  \pi_1 f_{Z|1}(z)\left(1 - \exp\{-G_z^*(z)\}\right),
\]
and find that
\begin{align*}
	(II)	& = \pi_1\EE_W\EE_Z\left[\1\left\{\wh G_x(AZ+w)< 0,  G_z(Z)\ge 0\right\}\left(1 - \exp\{-G_z^*(Z)\}\right) \1\{\E_t\} \mid Y=1, W = w\right]\\
	&\quad +\pi_1\EE_W\EE_Z\left[\1\left\{\wh G_x(AZ+w)< 0,  G_z(Z)\ge 0\right\}\left(1 - \exp\{-G_z^*(Z)\}\right) \1\{\E_t^c\} \mid Y=1, W = w\right]\\
	&\le \pi_1 c_*t ~ \EE_Z\left[\1\left\{-t \le G_z(Z)< 0\right\} | Y=0\right] +\pi_1\PP(\E_t^c \mid Y = 1)
\end{align*}
Combining the bounds for $(I)$ and $(II)$ and using $G_z^*(z) = c_*G_z(z)$, we conclude that
\begin{align*}
	R_x(\wh g_x) - R_z^* &\le  \PP \{\E_t^c\} +\pi_0 c_* t \PP\{ -c_* t< G_z^*(Z) <0 \mid Y=0\}\\
	&\quad + \pi_1  c_* t \PP\{ 0< G_z^*(Z) < c_*t\mid Y=1\}. 
\end{align*}
Using the fact that
\begin{align*}
	&G_z^*(Z) \mid Y = 1 \sim N\left({1\over 2}\Dt^2 + \log{\pi_1\over \pi_0}, \Dt^2\right),\\
	& G_z^*(Z) \mid Y = 0 \sim N\left(-{1\over 2}\Dt^2 + \log{\pi_1\over \pi_0}, \Dt^2\right),
\end{align*}
the proof easily follows. \qed

\subsubsection{Proof of Proposition \ref{prop_risk}}\label{app_proof_prop_risk}
For any $a\ge 1$ with some $C=C(a)$, recall that 
\[
\wh \omega_n(a) = C\left\{
\sqrt{a\log n}\left(
\wh r_1 + \|\sw\|_\op^{1/2} \wh r_2
\right) + \wh r_2 \wh r_3 + \sqrt{\log n\over n}  
\right\}
\]
where 
\[
\wh r_1 := \|\sz^{1/2}(A^\T \wh \theta - \beta)\|_2,\quad  \wh r_2 := \|\wh\theta\|_2,  \quad \wh r_3 := {1\over \sqrt n}\|\W(P_B-P_A)\|_{\op}.
\]
The proof of Proposition \ref{prop_risk} consists of two parts:

(i) We first show that, for any $a\ge 1$, there exists $C=C(a)$ such that,  with probability at
least $1-2n^{-a}$,
\begin{align}\label{target_1}
	\begin{split}
		&	|\wh G_x(X) - G_z(Z)| \le \\
		&        C\sqrt{a\log n}\left(\wh r_1 +  \|\sw\|_\op^{1/2}  ~ \wh r_2  \right) +                 \left|\wh \beta_0- \beta_0 + {1\over 2}(\a_1 + \a_0)^\T(A^\T \wh \theta - \beta)\right|.
	\end{split}
\end{align}
Notice that  randomness of the right-hand side depends on the training data $\D$ only. \\

(ii)
We then prove in Lemma \ref{lem_beta_0} that the inequality 
\begin{equation}\label{target_2}
	\left|
	\wh \beta_0- \beta_0 + {1\over 2}(\a_1 + \a_0)^\T(A^\T \wh \theta - \beta)
	\right|   \le C\left(\wh r_1 +  \|\sw\|_\op^{1/2} \wh r_2  + \wh r_2 \wh r_3 + \sqrt{\log n\over n}\right)
\end{equation}
holds with probability $1-\cO(n^{-1})$.   Combination of steps (i) and (ii) 	  yields the claim.\\

To prove (\ref{target_1}), starting with
\begin{align*}
	\wh G_x(X) - G_z(Z) &= \left(Z - {\a_1 + \a_0 \over 2}\right)^\T(A^\T \wh \theta - \beta) + W^\T \wh \theta\\
	&\quad  + \wh \beta_0- \beta_0 + {1\over 2}(\a_1 + \a_0)^\T(A^\T \wh \theta - \beta),
\end{align*}
we observe that $\wh\theta$ and $\wh\beta_0$ are independent of $W$ and $Z$. 
Since $W^\T \wh\theta$  given $\wh\theta$ is subGaussian with parameter  
\begin{align*}
	\gamma\sqrt{\wh\theta^\T \sw \wh\theta} &\le \gamma \|\sw\|_{\op}^{1/2} ~ \wh r_2,
\end{align*}
we find that, for any $\alpha>0$,
\begin{equation}\label{bd_tail_Wtheta}
	\PP\left\{
	|W^\T \wh \theta| \ge \gamma\sqrt{2\alpha \log n} ~  \|\sw\|_{\op}^{1/2} ~ \wh r_2
	\right\}	\le  2n^{-\alpha}.
\end{equation}
We prove our bound for $(Z - (\a_1+\a_0)/2)^\T(A^\T \wh \theta - \beta)$ by a  conditioning argument. Given $Y=0$ and $\wh\theta$, we use that $Z$ and $\wh \theta$ are independent and  derive
\begin{align*}
	\PP\left\{
	\left|\left(Z -{\a_1+\a_0\over 2}\right)^\T \left(A^\T \wh \theta - \beta\right)\right| \ge 
	M +  t\sqrt{V}  ~ \Big | ~ Y = 0, ~\wh\theta ~\right\}
	&\le  2e^{-t^2/2}
\end{align*}
from $Z \mid Y=0 \sim N_K(\a_0, \szy)$,
for all $t\ge 0$, where
\[
M = {1\over 2}|(\a_1-\a_0)^\T(A^\T \wh \theta - \beta)| ,\quad 
V = (A^\T \wh \theta - \beta)^\T \szy (A^\T \wh \theta - \beta).
\]
Here, by (\ref{eq_Deltas}), we have
\begin{align*}
	M & \le  {1\over 2}\|\a_1-\a_0\|_{\sz}   \|\sz^{1/2}(A^\T \wh \theta - \beta)\|_2  \lesssim   \|\sz^{1/2}(A^\T \wh \theta - \beta)\|_2 = \wh r_1
\end{align*}
while 
by the Cauchy-Schwarz inequality and (\ref{eq_sz_C}), we obtain
\begin{align*}
	V &\le \|\sz^{-1/2}\szy\sz^{-1/2}\|_{\op} \|\sz^{1/2}(A^\T \wh \theta - \beta)\|_2^2 \le  \|\sz^{1/2}(A^\T \wh \theta - \beta)\|_2^2 = \wh r_1^2.
\end{align*}
These bounds on $M$ and $V$ yield that, for any $\alpha > 0$,
\[
\PP\left\{\left|\left(Z-{\a_1+\a_0\over 2}\right)^\T(A^\T\wh\theta - \beta)\right| \gtrsim \left(\sqrt{\alpha \log n}  + 1\right) \wh r_1 ~ \Big | ~ Y = 0 \right\} \le 2n^{-\alpha}.
\]
By the same arguments, the above also holds by conditioning on $Y = 1$ and $\wh \theta$. After we take expectations, we  obtain the same bounds for the   unconditionial versions.
Together with (\ref{bd_tail_Wtheta}), the proof of (\ref{target_1}) is complete by taking $\alpha \ge 1$. This concludes the proof of Proposition \ref{prop_risk}. \qed 

\medskip 

\begin{lemma}\label{lem_beta_0}
	Under conditions of Proposition \ref{prop_risk}, with probability $1-\cO(n^{-1})$, 
	\[
	\left|
	\wh \beta_0- \beta_0 + {1\over 2}(\a_1 + \a_0)^\T(A^\T \wh \theta - \beta)
	\right|   \le C\left(\wh r_1 +  \|\sw\|_\op^{1/2} ~ \wh r_2  + \wh r_2 \wh r_3 + \sqrt{\log n\over n}~ \right)
	\]
	for some constant $C = C(\gamma)>0$.
\end{lemma}
\begin{proof}
	By definition, 
	\begin{align*}
		\left|\wh\beta_0  - \beta_0 +{1\over 2}(\a_1 + \a_0)^\T(A^\T \wh \theta - \beta)\right| & \le  ~{1\over 2}\underbrace{\left|(A\a_0+A\a_1 - \wh \mu_0 - \wh \mu_1)^\T  \wh\theta\right|}_{R_1}\\
		& \hspace{-3cm} +\underbrace{\left| \wh \pi_0\wh \pi_1\left[
			1- (\wh \mu_1-\wh \mu_0)^\T\wh\theta
			\right]\log{\wh \pi_1 \over \wh \pi_0}  -  \pi_0\pi_1\left[
			1 -(\a_1-\a_0)^\T \beta
			\right]\log{\pi_1 \over \pi_0}\right|}_{R_2}.
	\end{align*}
	We proceed to bound $R_1$ and $R_2$ separately.\\
	
	\noindent{\bf Bounding $R_1$.} 
	By recalling that, for any $k\in \{0,1\}$,
	\begin{equation}\label{def_W_bar}
		\begin{split}
			\wh \mu_k &= {1\over n_k}\sum_{i=1}^n X_i \1\{ Y_i = k\}\\
			&= A  \underbrace{{1\over n_k}\sum_{i=1}^n Z_i \1\{Y_i = k\}}_{\wh \a_k} + \underbrace{{1\over n_k}\sum_{i=1}^n W_i\1\{Y_i = k\}}_{\bar W_{(k)}},
		\end{split}
	\end{equation}
	we have
	\begin{align*}
		\left|\a_k^\T A^\T \wh \theta-\wh \mu_k^\T \wh\theta\right| &\le 
		\left|(\a_k - \wh \a_k)^\T A^\T \wh \theta \right| + \left| \bar W_{(k)}^\T \wh\theta\right| \\
		&\le 	\left|(\a_k - \wh \a_k)^\T \beta \right| + \left|(\a_k-\wh \a_k)^\T (\beta - A^\T\wh\theta)\right|+ \left|\bar W_{(k)}^\T \wh\theta\right|\\
		&\le  	\left|(\a_k - \wh \a_k)^\T \beta \right|  + 	 \|\sz^{-1/2}(\a_k-\wh \a_k)\|_2  \|\sz^{1/2}(\beta - A^\T \wh \theta)\|_2\\
		&\quad  + \|P_A\bar W_{(k)}\|_2  \|\wh\theta\|_2+ \|(P_B-P_A)\bar W_{(k)}\|_2 \|\wh\theta\|_2.
	\end{align*}
	The last step uses the identity
	\[
	\bar W_{(k)}^\T \wh\theta = \bar W_{(k)} P_B B(\Pi_n\X B)^+\Y = \bar W_{(k)}(P_A + P_B - P_A) \wh\theta
	\]
	and the Cauchy-Schwarz inequality. 
	By invoking Lemma \ref{lem_deviation} and using 
	\begin{equation}\label{quad_beta_sz}
		\|\sz^{1/2}\beta\|_2 = \pi_0\pi_1 \|\a_1-\a_0\|_{\sz}  \overset{(\ref{eq_Deltas})}{=} \pi_0\pi_1 \sqrt{\Dt^2 \over 1+\pi_0\pi_1\Dt^2} \lesssim 1,
	\end{equation}
	from (vi),
	we further have 
	\begin{align*}
		\left|(\a_k - \wh \a_k)^\T \beta \right|  + 	 \|\sz^{-1/2}(\a_k-\wh \a_k)\|_2  \|\sz^{1/2}(\beta - A^\T \wh \theta)\|_2   \lesssim \sqrt{\log n\over n} +  \sqrt{K\log n \over n_k} 	\wh r_1
	\end{align*}
	with probability $1- \cO(1/n)$.  Lemma \ref{lem_pi_hat} yields
	\begin{equation}\label{lb_n_min}
		\PP^{\D}\left\{
		{n_1 \wedge n_2 \over n}\ge c (\pi_0 \wedge \pi_1) \ge c\pi_0\pi_1
		\right\}\ge 1-2n^{-1}.
	\end{equation}
	After collecting the above terms and using Lemma \ref{fact} 
	and $K\log n \lesssim n$, we obtain
	\[
	\left|\a_k^\T A^\T \wh \theta -\wh \mu_k^\T \wh\theta\right|  \lesssim 	\wh r_1 \sqrt{K\log n\over n}+ \sqrt{\log n \over n} +\wh r_2 \left(\left\|P_A\bar W_{(k)}\right\|_2  + \left\|(P_B-P_A)\bar W_{(k)}\right\|_2\right) 
	\]
	with probability $1- \cO(1/n)$. Notice that 
	\begin{align*}
		\|(P_B-P_A)\bar W_{(1)}\|_2 &= {1\over n_1}\|(P_B-P_A)\W^\T \Y\|_2\\
		&\le {1\over \sqrt{n}}\|\W(P_B-P_A)\|_{\op} {\|\Y\|_2 \sqrt{n}\over n_1}\\
		&\lesssim ~ \wh r_3 &&\textrm{by (\ref{lb_n_min})}
	\end{align*}
	and, similarly, 
	\begin{align*}
		\|(P_B-P_A)\bar W_{(0)}\|_2 &
		\lesssim ~ \wh r_3.
	\end{align*}
	Then use Lemma \ref{lem_W} to obtain
	\[
	\wh r_2 \left(\left\|P_A\bar W_{(k)}\right\|_2  + \left\|(P_B-P_A)\bar W_{(k)}\right\|_2\right)  \lesssim 	\wh r_2 \sqrt{\|\sw\|_{\op}} \sqrt{K\log n \over n}+	\wh r_2\wh r_3
	\]
	which further implies 
	\begin{align*}
		R_1 &\lesssim 	\wh r_1 \sqrt{K\log n\over n}+ \sqrt{\log n \over n} +\wh r_2\left(\sqrt{\|\sw\|_{\op}} \sqrt{K\log n \over n} + 	\wh r_3\right),
	\end{align*}
	with probability $1- \cO(1/n)$. Therefore, with the same probability,  we have 
	\begin{align}\label{bd_R1}\nonumber
		&|(\alpha_0-\alpha_1)^\T \beta  - (\wh\mu_0 - \wh\mu_1)^\T \wh\theta|\\ \nonumber
		& \le |(\alpha_0 - \a_1)^\T (\beta -A^\T \wh\theta)|  + | (\alpha_0-\alpha_1)^\T A^\T \wh\theta - (\wh\mu_0 - \wh\mu_1)^\T \wh\theta|\\\nonumber
		&\le   \|\a_1-\a_0\|_{\sz} \|\sz^{1/2}(\beta-A^\T\wh\theta)\|_2 + \sum_{k\in\{0,1\}}   \left|\alpha_k^\T A^\T \wh\theta - \wh\mu_k^\T \wh\theta\right| \\
		&\lesssim 	\wh r_1+ \sqrt{\log n \over n} + 	\wh r_2\left(\sqrt{\|\sw\|_{\op}} \sqrt{K\log n \over n} +	\wh r_3\right).
	\end{align}
	In the last step, we also use $\|\a_1-\a_0\|_{\sz} \lesssim 1$ from Lemma \ref{fact} and $K\log n \lesssim n$ to collect terms.\\
	
	\noindent{\bf Bounding  $R_2$.}
	We bound from above the following two terms separately:
	\begin{align*}
		R_{21} &:= \left|\wh \pi_0\wh \pi_1(\wh \mu_1-\wh \mu_0)^\T\wh\theta - \pi_0\pi_1(\a_1-\a_0)^\T \beta+ \pi_0\pi_1 - \wh \pi_0\wh \pi_1\right|\cdot \left|\log{\wh \pi_1 \over \wh \pi_0}\right|,\\
		R_{22} &:= \left| \pi_0\pi_1
		- \pi_0\pi_1(\a_1-\a_0)^\T \beta
		\right|\cdot \left|\log{\wh \pi_1 \over \wh \pi_0} - \log{\pi_1 \over \pi_0}\right|.
	\end{align*}
	We start with 
	\begin{align*}
		R_{21}  &\le  \wh \pi_0\wh \pi_1\left|(\wh \mu_1-\wh \mu_0)^\T\wh\theta -(\a_1-\a_0)^\T \beta\right|\cdot \left|\log{\wh \pi_1 \over \wh \pi_0}\right|\\
		&\quad +\left|\wh \pi_0\wh \pi_1 - \pi_0\pi_1\right|   \pi_0\pi_1 \|\a_1 -\a_0\|_{\sz}^2\cdot  \left|\log{\wh \pi_1 \over \wh \pi_0}\right| + |\wh \pi_0\wh \pi_1 - \pi_0\pi_1|\cdot  \left|\log{\wh \pi_1 \over \wh \pi_0}\right|\\
		&\le \wh \pi_0\wh \pi_1\left|(\wh \mu_1-\wh \mu_0)^\T\wh\theta -(\a_1-\a_0)^\T \beta\right|\cdot \left|\log{\wh \pi_1 \over \wh \pi_0}\right|\\
		&\quad+   |\wh \pi_0 - \pi_0|\cdot  \left|\log{\wh \pi_1 \over \wh \pi_0}\right|  \pi_0\pi_1\|\a_1 -\a_0\|_{\sz}^2+   |\wh \pi_0 - \pi_0|\cdot  \left|\log{\wh \pi_1 \over \wh \pi_0}\right|
	\end{align*}
	by using 
	\begin{equation}\label{bd_a_hat_a}
		\left|\wh \pi_0\wh \pi_1 - \pi_0\pi_1\right| = 
		\left| (\wh \pi_0- \pi_0)\wh \pi_1 + (\wh \pi_1 - \pi_1)\pi_0\right| =
		\left|(\wh\pi_0 - \pi_0)(\wh \pi_1 - \pi_0)\right| \le |\wh\pi_0 - \pi_0 |
	\end{equation}
	in the last line. The concavity of $x\mapsto \log(x)$ implies 
	\[
	\left|\log{\wh \pi_1 \over \wh \pi_0}\right|   \le {|\wh \pi_1 - \wh \pi_0|\over \wh \pi_1 \wedge \wh \pi_0}
	\]
	and $\pi_0\pi_1\|\a_1-\a_0\|_{\sz}^2 \le 1 $ follows from (\ref{eq_Deltas}). We
	invoke the bound (\ref{bd_R1})  on $R_1$,  use Lemma \ref{lem_pi_hat}, inequality  (\ref{eq_Delta}) and condition  {(vi)} to obtain
	\begin{align*}
		\PP^{\D}\left\{ R_{21} 
		\lesssim   	\wh r_1 +    \sqrt{\log n \over n}  +	\wh r_2 \|\sw\|_{\op}^{1/2} \sqrt{K\log n \over n} + \wh r_2	\wh r_3\right\} \ge 1-cn^{-1}.
	\end{align*}
	
	To bound $R_{22}$, notice from  (\ref{eq_Deltas_prime}) that 
	\[
	\pi_0\pi_1 - \pi_0\pi_1(\a_1-\a_0)^\T \beta =  \pi_0\pi_1 \left[
	1 - \pi_0\pi_1 \|\a_1-\a_0\|_{\sz}^2
	\right] = {\pi_0\pi_1 \over 1 + \pi_0\pi_1 \Dt^2}.
	\]
	Use 
	\begin{align*}
		\left|\log{\wh \pi_1 \over \wh \pi_0} - \log{\pi_1 \over \pi_0}\right| &\le \left|
		{\wh \pi_1 \over \wh \pi_0} - {\pi_1 \over \pi_0}
		\right| \cdot \left({\pi_0 \over \pi_1} \vee  {\wh\pi_0 \over \wh\pi_1}\right)\\
		&\le \max \left\{
		{|\wh \pi_1 \pi_0 - \pi_1 \wh \pi_0|\over \wh \pi_0\pi_1}, ~ {|\wh \pi_1 \pi_0 - \pi_1 \wh \pi_0|\over \pi_0\wh \pi_1}
		\right\}
	\end{align*}
	and 
	\[
	|\wh \pi_1 \pi_0 - \pi_1 \wh \pi_0| \le |\wh \pi_1-\pi_1| \pi_0 + \pi_1 |\wh \pi_0-\wh \pi_0|
	\]
	together with Lemma \ref{lem_pi_hat} to conclude 
	\[
	R_{22} \lesssim {\pi_0\pi_1 \over 1 + \pi_0\pi_1 \Dt^2}\left(\sqrt{\pi_0\over \pi_1} + \sqrt{\pi_1 \over \pi_0}\right) \sqrt{\log n \over n} \lesssim  \sqrt{\log n \over n}
	\]
	with probability $1-\cO(1/n)$.  Combining the bounds of $R_1$, $R_{21}$ and $R_{22}$ yields the desired result.
\end{proof}

\subsubsection{Proof of Theorem \ref{thm_general_risk_explicit}}\label{app_proof_thm_general_risk_explicit} 
We take 
$\wh\omega_n(a)$ as given in (\ref{def_hat_omega_n}) of Proposition \ref{prop_risk}. After we
apply Theorem \ref{thm_general_risk} with $t=\omega_n$, we obtain, on the event $\{\wh\omega_n(a) \le \omega_n\}$,
\begin{align*} R_x(\wh g_x) - R_z^* 
	&= \PP\{ \wh g_x(X)\ne Y\mid \D\} - R_z^*\\
	&\le \PP\{ |\wh G_x(X)- G_z(Z)| > \omega_n\mid \D \} + c_*
	\omega_n P(\omega_n)\\
	&\le \PP\{ |\wh G_x(X)- G_z(Z)| > \wh \omega_n(a) \mid \D \}  + c_*
	\omega_n P(\omega_n)\\
	&\lesssim n^{-a} + c_*\omega_n P( \omega_n),
\end{align*}
with probability $1-\cO(1/n)$, by Proposition \ref{prop_risk}.
The second 
term $c_*\omega_n P( \omega_n)$ can be written as
\begin{equation}\label{def_T_original}
	T :=  \pi_0 c_* \omega_n\left[\Phi\left(R\right) - \Phi\left(R-c_*\omega_n/ \Dt\right)\right] + \pi_1 c_* \omega_n \left[\Phi\left(L+{c_* \omega_n / \Dt}\right) - \Phi\left(L\right)\right]
\end{equation}
with 
$$
c_* = {1 \over \pi_0\pi_1} + \Dt^2,\qquad L = -{1\over 2}\Dt - {\log{\pi_1\over \pi_0} \over \Dt},\qquad R = {1\over 2}\Dt - {\log{\pi_1\over \pi_0} \over \Dt}.
$$
By the mean-value theorem, we obtain the bound
\begin{align*}
	T \le \frac{ c_*^2 \omega_n^2 } {\Delta} \exp(-m^2/2) \qquad
	\text{ with }  
	m\in
	\left[L, L+ \frac{c_*\omega_n}{\Dt}\right]
	\cup \left[ R- \frac{c_*\omega_n}{\Dt}, R \right]  .
\end{align*}
We consider three scenarios:

(1) $\Delta \asymp 1$. In this case, $c_* \asymp 1$ and $m\asymp1$, so that
$
T \lesssim \omega_n^2.
$

(2) $\Delta \to \i$. In this case, $c_*\asymp\Dt^2$, $c_* \omega_n / \Dt \asymp \omega_n \Delta = o(\Dt)$, whence $m^2 = c_\pi  \Dt^2 + o(\Dt^2)$  with $c_\pi = 1/8$ if $\pi_0 = \pi_1$, and
\begin{align*}
	T &\lesssim \omega_n^2   \Dt^3 \exp\left[-c_\pi  \Dt^2 + o(\Dt^2)\right] = \omega_n^2   \exp\left[-c_\pi  \Dt^2 + o(\Dt^2)\right] .
\end{align*}

(3a)  $\Dt \to 0$ and $\pi_1$ and $\pi_0$ are distinct.  
In this case $c_*\asymp1$, $L =  -\log(\pi_1/\pi_0) / \Dt + o(1)$,
$R= -\log(\pi_1/\pi_0) / \Dt + o(1)$,
$c_*\omega_n /\Dt \asymp \omega_n / \Dt = o(1/\Dt)$, whence $m = - \log(\pi_1/\pi_0) / \Dt + o(1/\Dt)$ 
and
\begin{align*}
	T &\lesssim {\omega_n^2  \over \Dt} \exp\left[-{\log(\pi_1/\pi_0) \over \Dt^2 }+ o\left(1\over \Dt^2\right)\right]=  \omega_n^2   \exp\left[-{\log(\pi_1/\pi_0) \over \Dt^2 }+ o\left(1\over \Dt^2\right)\right].
\end{align*}

(3b)  $\Dt \to 0$ and $\pi_0=\pi_1$.
In this case, $c_*\asymp1$, $L=-\Dt/2=-R$. 
Thus 
$
T\lesssim  {\omega_n^2}/{\Dt}.$
The second bound $T\lesssim \omega_n$ follows directly from   (\ref{def_T_original}).

In view of the above three cases,   the proof is complete. 
\qed

\subsection{Proofs of Section \ref{sec_theory_application}}

We define $\wt \Z = \Z \sz^{-1/2}$ (the so-called whitened $\Z$).
Most of the proofs work on the following events
\begin{align}\label{def_event_z}
	\E_z &:= \left\{
	{n\over 2}  \le \lambda_K(\wt \Z^\T \Pi_n \wt \Z) \le  \lambda_1(\wt \Z^\T \Pi_n  \wt \Z) \le 2 n
	\right\}\\
	\E_w &:= \E_w^1 \cap \E_w^2\\
	\E_w^1 &:= \left\{
	\| \W^\T \W \|_\op \le 12 \gamma^2 n\errw
	\right\} \\
	\E_w^2 &:=  \left\{
	\|\W\|_F^2 \le 6\gamma^2 n~\tr(\sw)
	\right\}
\end{align}
Here
\begin{align}
	\errw &:= \| \sw\|_\op \left( 1+ \frac{r_e(\sw)}{n} \right).
\end{align}
Part (vi) of Lemma \ref{lem_deviation} states that  $\PP(\E_z)\ge 1-\cO(1/n)$, while Lemma \ref{lem_W} and Lemma \ref{lem_op_norm} state that
$\PP\{\E_w\} \ge 1 - 2\exp(-n)$. 

For notational simplicity, we write 
\begin{equation*} 
	\lambda_k := \lambda_k(A\szy A^\T),\qquad \text{for all } k = 1,\ldots, K.
\end{equation*}
For future reference, by \eqref{eq_sz_C}, we also have 
\begin{equation}\label{ineq_lambdas}
	\lambda_K(A\sz  A^\T) \ge \lambda_K, \qquad  \lambda_1(A\sz  A^\T) \le  \pi_0\pi_1 \Dt^2 \lambda_1.
\end{equation}
Finally, we write the singular value decomposition of $\Pi_n \X$ as 
\[
\Pi_n \X =  \V_K \D_K \U_K^\T + (\Pi_n \X)_{(-K)}
\]
with $\D_K = \diag(\sigma_1, \ldots, \sigma_K)$.

\subsubsection{Proof of Theorem \ref{thm_PCR_K_hat}}\label{app_proof_thm_PCR_K_hat}
We show $\wh K = K$ with probability $1-\cO(1/n)$.
Let 
\begin{align*}
	\mu_n = c_0(n+p).
\end{align*} Under the conditions of Theorem \ref{thm_PCR_K_hat}, Proposition 8 in \cite{bing2020prediction} shows that 
\begin{align*}
	\PP\{ \wh K \le K \} \ge \PP\{ \E_w \} \ge 1-2\exp(-n)
\end{align*}
We will prove the theorem by showing that
\[ \PP\{\wh K \ge K\}\ge \PP\{ \E_z\cap \E_w\}=1-\cO(1/n)\]
From Corollary 10 of \cite{rank19},  we need to verify
\[
\sigma_K^2(\Pi_n\Z A^\T) \ge \mu_n {\|\Pi_n\W\|_F^2 \over np} \left[
{\sqrt 2 \over 2} + \sqrt{np \over np - \mu_n K}
\right]^2.
\]
For the left-hand-side, invoking $\E_z$ in (\ref{def_event_z}) gives 
\[
\sigma_K^2(\Pi_n \Z A^\T)  \ge  ~  {n \over 2}  \lambda_K(A\sz A^\T) \overset{\eqref{ineq_lambdas}}{\ge} ~  {n \over 2} \lambda_K.
\]
The last inequality follows from (\ref{eq_sz_C}). 
Regarding the right-hand-side, by invoking the inequalities in $\E_w^2$ and using $$
K\le \bar K \le {\nu\over 1+\nu}{np \over \mu_n}
$$ 
from (\ref{def_K_hat}),
it can be bounded from above by
\[
\mu_n {\|\W\|_F^2 \over np} \left[
{\sqrt 2 \over 2} +\sqrt{1+\nu}
\right]^2 \le  C \tr(\sw){n+p \over p}
\]
for some $C = C(c_0, \nu)$.
The proof is then completed by observing that 
$n\lambda_K \ge 2C \tr(\sw)(n+p)/p$ as
\[
{\tr(\sw) \over \lambda_K}{n+p \over np} \le {\tr(\sw)\over n\lambda_K} + {\lambda_1(\sw) \over \lambda_K} = {\errw \over \lambda_K}  = {1\over \xi} \le {1\over 2C}.
\]
\qed

\subsubsection{Proof of Theorem \ref{thm_PCR}}\label{app_proof_PCR} 
According to Theorem \ref{thm_general_risk_explicit}, we need to bound the quantities
$\wh r_1$, $\wh r_2$ and $\wh r_3$. A combination of  the bounds (\ref{bound:r3}), (\ref{bound:r2}) and (\ref{bound:r1})
below yields that, with probability $1-\cO(n^{-1})$,  
\begin{align*}
	\wh r_1 &\lesssim \sqrt{K\log n\over n}  +  {\min\{1, \Dt\}\over \xi^*}  
	+   \sqrt{\kappa\over \xi^2}\\
	\wh r_2 &\lesssim   {1\over \sqrt{\lambda_K}} \left(
	\min\{1, \Dt\} + \sqrt{K\log n\over n} + \sqrt{\kappa\over \xi^2}
	\right)\\  
	\wh r_3 &\lesssim   \sqrt{\kappa { \errw  \over\xi }}
\end{align*}
Hence, for any $a\ge 1$, 
\begin{align*}
	\wh \omega_n(a) &= C\left\{
	\sqrt{a\log n}\left(
	\wh  r_1 + \|\sw\|_\op^{1/2} \wh r_2
	\right) + \wh r_2 \wh r_3 + \sqrt{\log n\over n}  
	\right\}\\
	&\lesssim  
	\sqrt{a\log n}\left(
	\sqrt{K\log n\over n}  +    \sqrt{\kappa \over \xi^2} + \min(1,\Dt) \sqrt{1\over \xi^*}
	\right)\\
	&\qquad +  \sqrt{\kappa\over \xi^2} \left(
	\min\{1, \Dt\} + \sqrt{K\log n\over n} +\sqrt{\kappa\over \xi^2}
	\right)    \\
	&\lesssim  \sqrt{a\log n}\left(
	\sqrt{K\log n\over n} + \min(1,\Dt) \sqrt{1\over \xi^*} + \sqrt{\kappa \over \xi^2}
	\right).
\end{align*} 
The theorem follows now from Theorem \ref{thm_general_risk_explicit}.  
\qed\\

\begin{lemma}\label{lem:r3} 
	Assume $\xi \gtrsim 1$. On the event $\E_z \cap \E_w^1$, we have 
	\begin{align}\label{bound:r3}
		\wh r_3 \lesssim  \sqrt{\errw} \left( 1 \wedge 
		\sqrt{\kappa \over\xi }  ~ \right).
	\end{align}     
\end{lemma}
\begin{proof}
	We have, on the event $\E_w^1$,
	\begin{align*}
		\wh r_3 &=n^{-1/2} \| \W (P_A-P_{\U_K}) \|_\op \\
		&\le n^{-1/2} \| \W \|_\op \| P_A-P_{\U_K} \|_\op\\
		&\le 2\sqrt{3}~ \delta_W^{1/2} \| P_A-P_{\U_K} \|_\op.
	\end{align*}
	The first bound follows trivially by $\| P_A-P_{\U_K} \|_\op \le 1$. To prove the other bound, on the event $\E_z$, the left-singular vectors $\U_A\in \cO_{p\times K}$ of the matrix $A$ equal the first $K$ left-singular vectors of the matrix $A\Z^\T \Pi_n \Z A^T$. By a variant of Davis-Kahan theorem \citep[Theorem 2]{yu2014}, we have, for some orthogonal matrix $Q\in \cO_{K\times K}$, 
	\begin{align*}
		\left\| \U_K  - \U_A Q\right\|_{\op} & \le 2^{3/2}{\|\X^\T \Pi_n \X - A\Z^\T \Pi_n \Z A^\T\|_{\op} \over \lambda_K(A\Z^\T \Pi_n \Z A^\T)} \\
		&\le 2^{3/2}{\|\W^\T \Pi_n \W\|_{\op} + 2 \|A\Z^\T \Pi_n\|_\op \| \Pi_n \W\|_{\op} \over \lambda_K(A\Z^\T\Pi_n  \Z A^\T)}.
	\end{align*} 
	On the event $\E_w^1$,
	\begin{align*}
		\| \W^\T \Pi_n \W\|_\op \le \| \W\|^2_\op \le 12\gamma^2 n \errw
	\end{align*}
	while, on the event $\E_z$, both
	\begin{align*}
		\lambda_K(A\Z^\T \Pi_n \Z A^\T) &\ge \lambda_K  (\wt\Z^\T \Pi_n\wt\Z) \lambda_K(A \sz A^\T) 
		\ge  {n\over 2}\lambda_K( A\sz A^\T)  
	\end{align*}
	and
	\begin{align*}
		\lambda_1(A\Z^\T \Pi_n \Z A^\T) &\le  2n\lambda_1( A\sz A^\T)
	\end{align*} hold.
	Hence,  
	\begin{align}\label{ineq:PUPA}\nonumber
		\left\| \U_K  - \U_A Q\right\|_{\op}  &  \lesssim    {\errw \over \lambda_K(A\sz A^\T)} + \sqrt{\errw\over \lambda_K(A\sz A^\T)}\sqrt{\lambda_1(A\sz A^\T) \over \lambda_K(A\sz A^\T)}\\\nonumber
		&\le {1\over \xi} +  \sqrt{\kappa\over\xi} &&\text{by \eqref{ineq_lambdas}}\\
		&\lesssim \sqrt{\kappa \over \xi} && \text{by }\xi \gtrsim 1.
	\end{align}
	After observing that  
	\begin{align*}
		\| P_A-P_{\U_K} \|_\op &\le  \|\U_AQ (\U_A Q - \U_K)^\T\|_\op + \|(\U_AQ   - \U_K)Q^\T \U_A^\T \|_\op\\
		&\le 2\|\U_K - \U_A Q\|_\op,
	\end{align*} 
	the proof is complete.   \end{proof}

\begin{lemma} \label{lem:r2} 
	Assume $K\log n \lesssim n$ and $\xi \gtrsim 1$. With probability at least $1-\cO(1/n)$ as $n\to\infty$, we have
	\begin{align}\label{bound:r2}
		\wh r_2 \lesssim   {1\over \sqrt{\lambda_K}} \left(
		\min\{1, \Dt\} + \sqrt{K\log n\over n} + {\wh r_3 \over \sqrt{\lambda_K}}
		\right) \lesssim {1\over \sqrt{\lambda_K}}.
	\end{align} 
\end{lemma}
\begin{proof}
	First, recall that $\X=\Z A^\T +\W$ and $\Pi_n \X \U_K = \V_K\D_K$. We write
	\begin{align*}
		\wh r_2 &= \| \U_K (\Pi_n \X \U_K)^+ \Y\|_2\\ 
		&= \| \U_K \D_K^{-2} \U_K^\T  \X^\T \Pi_n \Y\|_2\\
		&= \| \U_K \D_K^{-2} \U_K^\T (\Z A^\T + \W)^\T\Pi_n \Y\|_2\\
		&\le \| \U_K \D_K^{-2} \U_K^\T \W^\T\Pi_n \Y\|_2+ \| \U_K \D_K^{-2} \U_K^\T A \Z ^\T\Pi_n \Y\|_2
	\end{align*}
	We will bound the two terms on the right-hand side separately.\\
	Bound for $\rI:=\| \U_K \D_K^{-2} \U_K^\T \W^\T\Pi_n \Y\|_2$.   
	We first recall  that $\D_K = \diag(\sigma_1,\ldots, \sigma_K)$   so that 
	\begin{align*}
		\rI &\le {1\over \sigma_K^2} \|P_{\U_K} \W^\T\Pi_n \Y\|_2\\
		&\le {1\over \sigma_K^2} \left(\|P_{A} \W^\T\Pi_n \Y\|_2 + \|(P_{\U_K} -P_{A}) \W^\T\Pi_n \Y\|_2\right)\\
		&\le {1\over \sigma_K^2} \left(\|P_{A} \W^\T\Pi_n \Y\|_2 + \|\W(P_{\U_K} -P_{A})\|_\op \|\Y\|_2\right).
	\end{align*}
	Since $\|\Y\|_2=\sqrt{n_1}\le \sqrt{n}$, invoking Lemmas \ref{lem_sigma_K} and \ref{lem_W} yields 
	\begin{equation}\label{bd_I_r2}
		\rI \lesssim {1\over \lambda_K}\left(\sqrt{\|\sw\|_\op} \sqrt{K\log n\over n}+ \wh r_3\right) 
	\end{equation}
	with probability $1-\cO(n^{-K})$.\\
	
	\noindent     Bound for  $\rII:=\| \U_K \D_K^{-2} \U_K^\T A \Z ^\T\Pi_n \Y\|_2$. This is the most challenging part, in that we successfully avoid an unwanted multiplicative  factor of the condition number $\kappa$ of the matrix $A\sz A^\T$ to appear in our bound.   We have
	\begin{align*}
		\rII & \le n\| \U_K \D_K^{-2} \U_K^\T A\sz^{1/2}\|_\op ~  {1\over n}\|\wt\Z ^\T\Pi_n \Y\|_2\\
		&\le n\|\D_K^{-2} \U_K^\T A\sz^{1/2}\|_\op ~   2 \|(\wt \Z^\T \Pi_n \wt Z)^{-1}\wt\Z ^\T\Pi_n \Y\|_2 & \text{on $\E_z$}\\
		&\le 2 n\|\D_K^{-2} \U_K^\T A\sz^{1/2}\|_\op \left( \|(\Pi_n \wt \Z)^+ \Y - \sz^{1/2}\beta\|_2 + \|\sz^{1/2}\beta\|_2\right).
	\end{align*}
	On the one hand, we easily verify that 
	\begin{align}\label{bd_sz_beta}\nonumber
		\|  \sz^{1/2} \beta \|_2^2 &=  \| \pi_0\pi_1 \sz^{-1/2}(\a_1-\a_0)\|_2^2\\\nonumber
		&= \pi_0\pi_1 \frac{ \pi_0\pi_1\Delta^2}{ 1+ \pi_0\pi_1 \Dt^2} &&\text{from (\ref{eq_Deltas}})\\
		& \le \pi_0\pi_1 \min\{1, \pi_0\pi_1\Dt^2\},
	\end{align}
	and $\|(\Pi_n \wt \Z)^+ \Y - \sz^{1/2}\beta\|_2$ is controlled by Lemma \ref{lemma:zybeta} stated below. 
	On the other hand, again on the event $\E_z$, 
	\begin{align*}
		n^2 \|\D_K^{-2} \U_K^\T A\sz^{1/2}\|_\op^2 & =   n^2  \|\D_K^{-2} \U_K^\T A\sz A^\T \U_K \D_K^{-2} \|_\op  \\
		&\le {n\over 2}\|\D_K^{-2} \U_K^\T A \Z^\T \Pi_n \Z A^\T \U_K \D_K^{-2}  \|_\op.
	\end{align*}
	By the identity $\X = \Z A^\T + \W$ and the triangle inequality, we find 
	\begin{align*}
		& n^2 \| \D_K^{-2} \U_K^\T A\sz^{1/2}\|_\op^2 \\
		&\le {n\over 2}\|  \D_K^{-2} \U_K^\T \X^\T \Pi_n \X \U_K \D_K^{-2}  \|_\op + {n\over 2}\| \D_K^{-2} \U_K^\T \W^\T \Pi_n \W \U_K \D_K^{-2}  \|_\op \\ 
		&\qquad + {n}\| \D_K^{-2} \U_K^\T A \Z^\T \Pi_n \W \U_K \D_K^{-2} \|_\op \\
		&\le {n\over 2\sigma_K^2} + {n\over 2\sigma_K^4}\|\Pi_n \W  P_{\U_K}\|_\op^2 +  \left(n\|\D_K^{-2} \U_K^\T A \sz^{1/2}\|_\op \right) {1 \over \sigma_K^2} \|\wt \Z^\T \Pi_n \W P_{\U_K}\|_\op.
	\end{align*}
	Using the basic inequalities $x^2 \le a + bx \le a + b^2/2 + x^2/2$, for all $x$ and any $a,b>0$, we conclude 
	\begin{align*}
		n^2 \| \D_K^{-2} \U_K^\T A\sz^{1/2}\|_\op^2 &\le {n\over  \sigma_K^2} + {n\over  \sigma_K^4}\|\Pi_n \W  P_{\U_K}\|_\op^2 +    {1 \over \sigma_K^4} \|\wt \Z^\T \Pi_n \W P_{\U_K}\|_\op^2.
	\end{align*}
	Lemma \ref{lem_W} ensures that, with probability $1-e^{-n}$,
	\begin{align*}
		{1\over \sqrt n}\|\Pi_n \W  P_{\U_K}\|_\op &\le {1\over \sqrt n}\|\W  P_A\|_\op + {1\over \sqrt n}\|\W (P_{\U_K} - P_A)\|_\op\\
		&\le 12\gamma^2 \sqrt{\|\sw\|_{\op}} + \wh r_3
	\end{align*} 
	and, with probability $1-\cO(n^{-1})$, 
	\begin{align}\label{bd_ZWP_UK}\nonumber
		{1\over n}\|\wt \Z^\T \Pi_n \W P_{\U_K}\|_\op &\le   {1\over n}\|\wt \Z^\T \Pi_n \W  P_A\|_\op + {1\over \sqrt n}\|\Pi_n\wt \Z\|_\op {1\over \sqrt n}\|\W(P_{\U_K} - P_A)\|_\op \\
		&\lesssim  \sqrt{\|\sw\|_{\op}}\sqrt{K\log n\over n} + \wh r_3.
	\end{align}
	Next, we  use the inequalities  $\sigma_K^2 \ge n\lambda_K / 4$ and $\wh r_3^2 \le \errw $ 
	stated in Lemma \ref{lem_sigma_K} and Lemma \ref{bound:r3}, respectively,
	together with $K\log n \lesssim n$ and $\xi^*\ge \xi \ge C$ to  conclude that 
	\begin{align}\label{bd_key}\nonumber
		n^2 \| \D_K^{-2} \U_K^\T A\sz^{1/2}\|_\op^2 &\lesssim  ~ {1\over \lambda_K} + {\|\sw\|_\op + \wh r_3^2\over \lambda_K^2}  + {1 \over \lambda_K^2} \left(
		\|\sw\|_{\op}  {K\log n\over n} + \wh r_3^2
		\right)\\
		&\lesssim {1\over \lambda_K} + {1 \over \lambda_K \xi} ~ \lesssim ~ {1\over \lambda_K}
	\end{align} 
	with probability $1-\cO(n^{-1})$.
	Finally, we 
	combine the bounds \eqref{bd_sz_beta} and \eqref{bd_key} and invoke Lemma \ref{lemma:zybeta} to obtain the bound
	\begin{align} \label{bd_II}
		\rII \lesssim {1\over \sqrt{\lambda_K}} \left(
		\min\{1, \Dt\} + \sqrt{K\log n\over n}
		\right).
	\end{align} that holds 
	with probability $1-\cO(n^{-1})$. \eqref{bd_II}  in conjunction with \eqref{bd_I_r2} completes our proof. 
\end{proof}

\begin{lemma}\label{lem:r1} 
	Assume $\xi \ge C\kappa^2$ for some sufficiently large constant $C>0$. On the event $\E_z\cap \E_w^1$, with probability $1-\cO(n^{-1})$ as $n\to \infty$, we have
	\begin{align}\label{bound:r1}
		\wh r_1 \lesssim \sqrt{K\log n\over n}  +  {\min\{1, \Dt\}\over \xi^*}  
		+   {\wh r_3\over \sqrt{\lambda_K}} 
	\end{align}
	
\end{lemma}
\begin{proof}
	
	We first observe that
	\begin{align*}
		A^\T \wh \theta &=
		(\Pi_n\Z)^+ \Pi_n\Z A^\T \wh \theta
		&&\text{since } (\Pi_n\Z)^+\Pi_n \Z =\bI_K\\
		&= (\Pi_n\Z)^+ \Pi_n\X \U_K(\Pi_n\X \U_K)^+\Y  - (\Pi_n\Z)^+ \Pi_n\W \wh \theta
		&&\text{since } \X=\Z A^\T +\W\\
		&=  (\Pi_n\Z)^+\Y - (\Pi_n\Z)^+P_{\Pi_n\X\U_K}^{\perp}\Y  - (\Pi_n\Z)^+\Pi_n \W \wh \theta.
	\end{align*}
	Next, since $\wt \Z = \Z \sz^{-1/2}$, it is easily seen that
	$\sz^{1/2}(\Pi_n\Z)^+ = (\Pi_n\wt \Z)^+$ and
	hence,
	\begin{align}\label{rhat1}
		\begin{split}
			\wh r_1 &= \left\| \sz^{1/2} (A^\T \wh \theta -\beta) \right\|_2 \\
			&\le 
			\left\|  (\Pi_n\wt\Z)^+\Y - \sz^{1/2} \beta
			\right\|_2 + \left\|   (\Pi_n\wt\Z)^+\Pi_n\W \wh \theta \right\|_2 + \left\|
			(\Pi_n\wt\Z)^+P_{\Pi_n\X\U_K}^{\perp}\Y \right\|_2.
		\end{split} 
	\end{align}
	We will bound the three terms on the right separately.\\

	\noindent (i) We refer to Lemma \ref{lemma:zybeta} for the first term, $\|  (\Pi_n\wt\Z)^+\Y - \sz^{1/2} \beta
	\|_2$.\\
	
	\noindent
	{(ii) Bound for the second term $\|(\Pi_n\wt \Z)^+ \Pi_n\W \wh \theta\|$}. 
	We have, with probability $1-\cO(n^{-1})$,
	\begin{align*}
		\|(\Pi_n\wt \Z)^+ \Pi_n\W \wh\theta\|_2 
		&  \le {2\over n}\|\wt \Z^\T \Pi_n\W \wh\theta\|_2&& \text{on the event } \E_z\\
		&
		\le  {2\over n} \left\|\wt \Z ^\T \Pi_n \W P_{\U_K} \right\|_\op \|\wh \theta\|_2
		&& \text{since } \wh \theta= P_{\U_K} \wh \theta\\
		&\lesssim  \wh r_2~ \sqrt{ \|\sw\|_\op ~ K\log n\over n} + \wh r_2\wh r_3 &&\text{by \eqref{bd_ZWP_UK}}.
	\end{align*}  
	
	\noindent (iii) Third term: Bound for $\|(\Pi_n\wt\Z)^+P_{\Pi_n\X\U_K}^{\perp}\Y \|_2$. 
	This is the most challenging part.
	We first write
	\begin{align*}
		\left\|  (	\Pi_n\wt\Z)^+P_{\Pi_n\X\U_K}^{\perp}\Y \right\|_2 &\le 
		\frac{2}{n}  \left\|  \wt \Z^\T \Pi_n P_{\Pi_n\X\U_K}^{\perp}\Y  \right\|_2 
		&&\text{ on the event } \E_z  
	\end{align*} 
	and, using the identity $  \wt\Z^\T= \sz^{-1/2} \Z^\T = \sz^{-1/2} A^+ (\X^\T-\W^\T)$, we obtain
	\begin{align*}
		\frac{1}{n}  \left\|   \wt \Z^\T \Pi_n P_{\Pi_n\X\U_K}^{\perp}\Y  \right\|_2
		& ={1\over n}  \left\|       \sz^{-1/2}  A^+ (\X^\T-\W^\T) \Pi_n P_{\Pi_n\X\U_K}^{\perp}\Y  \right\|_2\\
		& = {1\over n}  \left\|    \sz^{-1/2}  A^+ (P_{\U_K} ^\perp \X^\T-\W^\T) \Pi_n P_{\Pi_n\X\U_K}^{\perp}\Y  \right\|_2
		\\
		&\le {1\over n}  \left\|      \sz^{-1/2}  A^+  (P_A- P_{\U_K} )  A(\Pi_n\Z)^\T P_{\Pi_n\X\U_K}^{\perp}\Y\right\|_2  \\ &\quad +
		{1\over n}  \left\|   \sz^{-1/2}  A^+  P_{\U_K}  (\Pi_n\W)^\T P_{\Pi_n\X\U_K}^{\perp}\Y \right\|_2
	\end{align*}
	The last line uses $  A^+ P_{\U_K}^\perp = A^+  P_{\U_K}^\perp -A^+ P_{A}^\perp$.
	Notice the subtle occurrence of the terms $P_A-P_{\U_K}$ and $P_{\U_K}$ which are crucial.
	The idea of the proof is to first show that the first term on the right is less than the left-hand side, and then to give a bound for the second term on the right. Indeed, we have
	\begin{align*}
		& {1\over n} \left\|   \sz^{-1/2}  A^+  (P_A- P_{\U_K} )  A(\Pi_n\Z)^\T P_{\Pi_n\X\U_K}^{\perp}\Y \right\|_2 \\
		&\le   \|  \sz^{-1/2}  A^+\|_\op \|  P_{\U_K}  - P_A   \|_\op \| A\sz^{1/2} \|_\op ~ {1\over n} \left\|  (\Pi_n\wt \Z)^\T P_{\Pi_n\X\U_K}^{\perp}\Y \right\|_2 
	\end{align*}
	and the factor $ \|  \sz^{-1/2}  A^+\|_\op \|  P_{\U_K}  - P_A   \|_\op \| A\sz^{1/2} \|_\op$ can be made less than 1/2 for $\xi\ge C\cdot \kappa^2$ by taking $C$ large enough on the event $\E_w$. This follows directly from the inequalities
	(\ref{ineq:PUPA}) and
	\begin{align*} 
		\|\sz^{-1/2}A^+\|_{\op}^2 &=
		\| \sz^{-1/2} (A^\T A)^{-1} A^\T \|_\op^2 
		=\|(\sz^{1/2}A^\T A\sz^{1/2})^{-1}\|_{\op} = {1\over \lambda_K(A\sz A^\T)}.
	\end{align*}
	Hence, on the event $\E_z\cap \E_w$, using the assumption $\xi\ge C \cdot\kappa^2$ and \eqref{ineq_lambdas}, we proved that
	\begin{align*}
		\frac{1}{n}  \left\|   \wt \Z^\T \Pi_n P_{\Pi_n\X\U_K}^{\perp}\Y  \right\|_2 &\le 
		{2\over n}  \left\|   \sz^{-1/2}  A^+  P_{\U_K}  (\Pi_n\W)^\T P_{\Pi_n\X\U_K}^{\perp}\Y \right\|_2\\
		&\le {2\over n\sqrt{\lambda_K}} \left( \left\|    P_{\U_K}  \W ^\T \Pi_n \Y \right\|_2 +\left\|P_{\U_K}  \W ^\T \Pi_n P_{\Pi_n\X\U_K} \Y \right\|_2\right).
	\end{align*}
	It remains to bound the two terms in the right-hand side. Recall that the first term has already been studied in \eqref{bd_I_r2}. For the second term, we find
	\begin{align*}
		\left\|P_{\U_K}  \W ^\T \Pi_n P_{\Pi_n\X\U_K} \Y \right\|_2 
		&= \left\|    P_{\U_K}   \W^\T \Pi_n  \X\U_K (\Pi_n\X\U_K)^{+} \Y \right\|_2\\
		& = \left\|    P_{\U_K}   \W^\T \Pi_n (\Z A^\T +\W) \wh\theta~ \right\|_2\\
		&\le \left\|    P_{\U_K}   \W^\T \Pi_n \wt \Z\right\|_\op \left\|\sz^{1/2} A^\T\wh \theta ~\right\|_2 + \left\|    P_{\U_K}   \W^\T \Pi_n  \W \wh\theta~ \right\|_2.
	\end{align*} 
	Notice that, by the definition of $\wh r_1$ and \eqref{bd_sz_beta},
	\[
	\left\|\sz^{1/2} A^\T\wh \theta ~\right\|_2  \le \left\|\sz^{1/2} (A^\T\wh \theta -\beta) ~\right\|_2  + \|\sz^{1/2}\beta\|_2 \le \wh r_1 + \min\{1,\Dt\}.
	\]
	Invoking \eqref{bd_ZWP_UK} thus yields 
	\[
	\left\|    P_{\U_K}   \W^\T \Pi_n \wt \Z\right\|_\op \left\|\sz^{1/2} A^\T\wh \theta ~\right\|_2 \lesssim n\left( \sqrt{\|\sw\|_{\op}}\sqrt{K\log n\over n} + \wh r_3\right) \left(\wh r_1 + \min\{1,\Dt\}\right)
	\]
	with probability $1-\cO(n^{-1})$. Next, we use Lemma \ref{lem_W} to find
	\begin{align*}
		\left\|    P_{\U_K}   \W^\T \Pi_n  \W \wh\theta~ \right\|_2 &= \|    P_{\U_K}   \W^\T \Pi_n  \W P_{\U_K}~ \|_\op \|\wh\theta\|_2\\
		&\le 2\wh r_2\left( \|\W P_A\|_\op^2 + \|\W(P_A-P_{\U_K})\|_\op^2\right)\\
		&\le 2n \wh r_2 \left(\|\sw\|_\op + \wh r_3^2\right)
	\end{align*}
	with probability $1-e^{-n}$. Combining the last two displays gives 
	\begin{align*}
		&{1\over n\sqrt\lambda_K} \left\|P_{\U_K}  \W ^\T \Pi_n P_{\Pi_n\X\U_K} \Y \right\|_2\\ 
		&\lesssim \left( \sqrt{1\over \xi^*}\sqrt{K\log n\over n} + {\wh r_3 \over \sqrt{\lambda_K}}\right) \left(\wh r_1 + \min\{1,\Dt\}\right) + {\wh r_2 \left(\|\sw\|_\op + \wh r_3^2\right)\over \sqrt{\lambda_K}}.
	\end{align*}
	Observe that the coefficient of $\wh r_1$ is sufficiently small as $\wh r_3 / \sqrt{\lambda_K} \le \sqrt{\errw/\lambda_K} \le \sqrt{1/\xi}$. Together with \eqref{bd_I_r2} and the bounds for the first two terms in \eqref{rhat1}, we obtain the following bound
	\begin{align*}
		\wh r_1 &\lesssim \sqrt{K\log n\over n} + \wh r_2~ \sqrt{ \|\sw\|_\op ~ K\log n\over n} + \wh r_2\wh r_3 +  {1\over \sqrt{\lambda_K}}\left(\sqrt{\|\sw\|_\op} \sqrt{K\log n\over n}+ \wh r_3\right) \\
		&\quad + \left( \sqrt{1\over \xi^*}\sqrt{K\log n\over n} + {\wh r_3 \over \sqrt{\lambda_K}}\right)   \min\{1,\Dt\} + {\wh r_2 \left(\|\sw\|_\op + \wh r_3^2\right)\over \sqrt{\lambda_K}}\\
		&\lesssim \sqrt{K\log n\over n}  +   \wh r_2 \sqrt{\|\sw\|_\op\over \xi^*}+   {\wh r_3\over \sqrt{\lambda_K}},
	\end{align*}
	with probability $1-\cO(n^{-1})$. In the second step we have used $\wh r_2 \le \sqrt{2/\lambda_K}$ and $\xi^*\ge \xi \ge C$ to reduce terms. 
	Finally, we complete the proof by invoking Lemma \ref{bound:r2} and further collecting terms. 
\end{proof}

\begin{remark}\label{rem:xiconstant}
	We provide an alternative proof to bound  $\|
	(\Pi_n\wt\Z)^+P_{\Pi_n\X\U_K}^{\perp}\Y \|_2$ in the third term of \eqref{rhat1} under the assumption that $\xi \ge C$ for some large enough $C$.  We will then provide a similar, sometimes slightly slower rate, albeit under a weaker assumption on the signal to noise $\xi$.
	
	As before, we 
	observe that, on the event $\E_z$,  
	\begin{align}\label{ineq:2ndterm2}
		\begin{split}
			&	\left\| (	\Pi_n\wt\Z)^+P_{\Pi_n\X\U_K}^{\perp}\Y \right\|_2\\
			& \lesssim   \frac1n \left\| 	  \sz^{-1/2} A^+ P_{\U_K}^\perp \X^\T \Pi_n P_{\Pi_n\X\U_K}^{\perp}\Y \right\|_2+  \frac1n \left\| 	   \sz^{-1/2} A^+\W^\T \Pi_n P_{\Pi_n\X\U_K}^{\perp}\Y \right\|_2.
		\end{split} 
	\end{align}
	For the second term on the right of (\ref{ineq:2ndterm2}), notice that 
	\begin{align*}
		\frac1n \left\| 	  \sz^{-1/2} A^+\W^\T \Pi_n P_{\Pi_n\X\U_K}^{\perp}\Y \right\|_2 & \le  {1\over n\sqrt{\lambda_K}}\left( \left\|P_A\W^\T \Pi_n \Y \right\|_2 + \left\|P_A\W^\T \Pi_n P_{\Pi_n\X\U_K} \Y \right\|_2\right).
	\end{align*} 
	Following the exact same arguments of bounding $\|P_{\U_K}  \W ^\T \Pi_n \Y \|_2$ and $ \|P_{\U_K}  \W ^\T \Pi_n P_{\Pi_n\X\U_K} \Y \|_2$ except by replacing $P_{\U_K}$ with $P_A$,
	we have, with probability $1-\cO(n^{-1})$,
	\begin{align*}
		&    \frac1n \left\| 	  \sz^{-1/2} A^+\W^\T \Pi_n P_{\Pi_n\X\U_K}^{\perp}\Y \right\|_2  \lesssim  \sqrt{1\over \xi^*}\sqrt{K\log n\over n}  +    {\wh r_2  \|\sw\|_\op \over \sqrt{\lambda_K}}.
	\end{align*}
	For the first term on the right of (\ref{ineq:2ndterm2}), as argued before,
	\begin{align*}
		\frac{1}{n} \left\|  \sz^{-1/2} A^+ P_{\U_K}^\perp \X  ^\T \Pi_n P_{\Pi_n \X \U_K}^\perp \Y \right\| &\le {1\over \sqrt{\lambda_K}} \|P_{\U_K}-P_A\|_\op {1\over \sqrt n} \| \Pi_n\X P_{\U_K}^\perp\|_\op\\ &\lesssim  {\sqrt{\kappa} \over \xi}.
	\end{align*} with probability $1-\cO(1/n)$. Here we also used 
	\[
	{1\over \sqrt n} \| \Pi_n\X P_{\U_K}^\perp\|_\op \le {1\over\sqrt{n}}\| \W\|_\op \lesssim \sqrt{\errw}
	\]
	by Weyl's inequality.  After we combine the bounds for the first two terms in \eqref{rhat1} with the bounds \eqref{bound:r3} and \eqref{bound:r2}, and the inequalities  $\wh r_3 \lesssim \sqrt{\errw}$ and $\wh r_2\lesssim 1/\sqrt{\lambda_K}$, we conclude that, with probability $1-\cO(n^{-1})$,
	\begin{align*}
		\wh r_1 &\lesssim \sqrt{K\log n\over n} + {\wh r_3\over \sqrt{\lambda_K}} 
		\min\{1, \Dt\}    +    \sqrt{\kappa\over \xi^2}\\
		&\lesssim \sqrt{K\log n\over n}    +    \sqrt{\kappa\over \xi^2}.
	\end{align*}  
	This bound only requires $\xi \ge C$, but is sub-optimal  compared to \eqref{bound:r1} when $\wh r_3$ is of smaller order than $\sqrt{\kappa\errw / \xi}$, for instance, when we have independent data to construct the estimate $\wt\U_K$.
	Combining the bound above with the bounds \eqref{bound:r3} and \eqref{bound:r2} leads to the same $\omega_n(a)$ in \eqref{def_omega_n_PCR}.
\end{remark}

\subsubsection{Technical lemmas used in the proof of Theorem \ref{thm_PCR}}

The following lemma provides lower bounds of the $K^{\rm th}$ singular value $\sigma_K$ of the matrix $\Pi_n\X$. 
\begin{lemma}\label{lem_sigma_K}
	Assume $\xi \ge 48\gamma^2$. On the event $\E_z\cap \E_w^1$, we have 
	\[
	\sigma_K^2 \ge  {n \over 4} \lambda_K(A\sz A^\T) \ge {n\over 4}\lambda_K.
	\]
\end{lemma}
\begin{proof}
	Recall 
	$$\Pi_n\X = \Pi_n\Z A^\T +\Pi_n\W = \Pi_n\wt \Z \sz^{1/2}A^\T + \Pi_n\W.$$ 
	By Weyl's inequality,
	\begin{align*}
		\sigma_K &\ge \sigma_K(\Pi_n \wt\Z \sz^{1/2}A^\T)- \sigma_1(\Pi_n\W) \\
		&\ge  \sigma_K(\sz^{1/2}A^\T)\sigma_K(\Pi_n\wt \Z) - \sigma_1(\Pi_n\W)  \\
		&=   \lambda_K^{1/2} (A \sz A^\T)\lambda_K^{1/2}(\wt\Z^\T \Pi_n\wt \Z) - \lambda_1^{1/2}(\W^\T\Pi_n\W)  \\   
		&\ge  \sqrt{n  \lambda_K(A \sz A^\T) / 2} - \sqrt{12\gamma^2 n\errw} &&\text{by $\E_z\cap \E_w^1$}.
	\end{align*}
	From \eqref{ineq_lambdas}, the result follows for $ \xi = \lambda_K / \errw  \ge 48\gamma^2$.
\end{proof}

\begin{lemma}\label{lemma:zybeta} 
	Under the conditions of Theorem \ref{thm_PCR}, the inequality
	\begin{align}\label{ineq:zybeta}
		\left\|  (\Pi_n\wt\Z)^+\Y - \sz^{1/2} \beta
		\right\|_2
		&\lesssim ~     \sqrt{K\log n \over n}  
	\end{align}
	holds with probability $1- \cO(1/n)$, as $n\to\i$. 
\end{lemma}
\begin{proof}     
	We can argue that on the event $\E_z$ in (\ref{def_event_z}),
	\begin{align*}
		\left\|  (\Pi_n\wt\Z)^+\Y - \sz^{1/2} \beta
		\right\|_2& =  \left\|  \left(\frac1n \wt \Z^\T \Pi_n \wt \Z \right)^{+} \frac1n \wt \Z^\T \Pi_n \Y  - \sz^{1/2} \beta
		\right\|_2\\
		&\lesssim \left\|   \frac1n \wt \Z^\T \Pi_n \Y  - \sz^{1/2} \beta  
		\right\|_2+ \left\| \left (  \frac1n \wt \Z^\T \Pi_n \wt \Z \right )^+ - \bI_K\right\|_\op \left\|   \sz^{1/2} \beta
		\right\|_2
	\end{align*}
	We use identity (\ref{eq_Deltas}) and Lemma \ref{lem_deviation} to obtain that  
	\begin{align*}
		\left\|  \left(  \frac1n \wt \Z^\T \Pi_n \wt \Z \right)^+  - \bI_K\right\|_\op \left\|   \sz^{1/2} \beta
		\right\|_2
		&\le   2  \left\|   \frac1n \wt \Z^\T \Pi_n \wt \Z   - \bI_K\right\|_\op \left\|   \sz^{1/2} \beta
		\right\|_2 &&\text{on the event } \E_z\\  
		&\lesssim \min(1 , \Dt)\sqrt{\frac{K\log n}{n} }
	\end{align*}
	holds with probability $1-\cO(1/n)$.
	Now, we argue by simple algebra, using the notation $\wt Z_i= \sz^{-1/2} Z_i$ and $\ba=\EE[Z]$,
	\begin{align*}
		{1\over n}\wt \Z^\T \Pi_n \Y
		&={1\over n} \sum_{i=1}^n( \wt Z_i -{1\over n}\sum_{j=1}^n \wt Z_j )\1\{ Y_i = 1\}\\
		&=  {1\over n} \sum_{i=1}^n\left( \wt Z_i -\sz^{-1/2}\ba \right)\1\{ Y_i = 1\} - {1\over n}\sum_{i=1}^n\left(\wt Z_i -\sz^{-1/2}\ba \right)  {n_1\over n}.
	\end{align*}   
	and, using the notation
	\begin{equation}\label{def_alpha_hat}
		\wh \a_k := {1\over n_k}\sum_{i=1}^n \1\{ Y_i = k\} Z_i, \qquad  k\in \{0,1\},
	\end{equation}
	we find
	\begin{align*}
		\sz^{-1/2}(\wh\a_1 - \wh\a_0 )&= 
		{1\over n_1}\sum_{i=1}^n \1\{ Y_i = 1\}\wt Z_i -	{1\over n_0}\sum_{i=1}^n \1\{ Y_i = 0\}\wt Z_i\\ 
		& = {n\over n_0n_1} 
		\sum_{i=1}^n \1\{ Y_i = 1\} \wt Z_i   -  {1 \over n_0}  \sum_{i=1}^n\wt Z_i\\
		&=  {n\over n_0n_1} 
		\sum_{i=1}^n \1\{ Y_i = 1\}(\wt Z_i -\sz^{-1/2}\ba)  -  {1 \over n_0}  \sum_{i=1}^n(\wt Z_i - \sz^{-1/2}\ba).
	\end{align*}
	Combining both identities, 
	we obtain 
	\[
	{1\over n}\wt \Z^\T \Pi_n \Y = {n_0n_1 \over n^2}\sz^{-1/2}\left(\wh \a_1 - \wh \a_0\right) + {2n_1  \over n^2}\sum_{i=1}^n (\wt Z_i- \sz^{-1/2}\ba).
	\]
	Hence,
	\begin{align*}
		\left\| n^{-1}\wt\Z^\T \Pi_n \Y - \sz^{1/2}\beta\right\| &\le
		2   \left\| {n_1\over n^2}\sum_{i=1}^n (\wt Z_i-\sz^{-1/2}\bar \a) \right\|\\ &+
		\left\| \frac{n_0n_1}{n^2}\sz^{-1/2}(\wh\a_1-\wh\a_0) -\pi_0\pi_1 \sz^{-1/2} (\a_1-\a_0) \right\|
	\end{align*}
	Finally, we invoke Lemmas \ref{lem_pi_hat} and \ref{lem_deviation}, and displays (\ref{bd_a_hat_a}),
	(\ref{quad_beta_sz}) and (\ref{lb_n_min})
	and we arrive at the desired bound (\ref{ineq:zybeta})
	with probability $1- \cO(1/n)$.
\end{proof}

\subsubsection{Proof of Theorem \ref{thm_PCR_indep}}\label{app_proof_thm_PCR_indep}
We mainly follow the arguments in the proof of Theorem \ref{thm_PCR} above to bound $\wh r_1$, $\wh r_2$ and $\wh r_3$ for $B = \wt \U_K$. For simplicity, we assume $n' = n$.\\

\noindent {\bf Bound for $\wh r_3$:}
To bound 
$$
\wh r_3:= {1\over \sqrt n}\|\W (P_A-P_{\wt\U_K})\|_\op,
$$
by inspecting the proof of Lemma \ref{lem:r3}, we have
\begin{equation}\label{ineq:PUPA_indep}
	\PP\left\{
	\|P_A-P_{\wt\U_K}\|_\op \lesssim \sqrt{\kappa\over \xi}
	\right\} = 1-\cO(n^{-1}).
\end{equation}
Since $\wt \U_K$ is independent of $\X$, and as a result  independent of $\W$, an application of Lemma \ref{lem_op_norm}
yields 
\[
\PP^{\D}\left\{
{1\over n}\left\|\W(P_{\wt\U_K}- P_A)\right\|_\op^2 \lesssim \|H\|_\op + {\tr(H) \over n}
\right\}\ge 1 - \exp(-n),
\]
where the matrix
\begin{align*}
	H = \sw^{1/2}(P_{\wt\U_K}- P_A)^2\sw^{1/2}
\end{align*} 
satisfies 
\begin{align*}
	\| H\|_\op &= \| \sw^{1/2}(P_{\wt\U_K}- P_A)^2\sw^{1/2} \|_\op\\
	&\le  \|\sw\|_\op \|P_{\wt\U_K} - P_A\|_\op^2 
\end{align*}
and 
\begin{align*}
	{\tr(H) \over n} &\le 2 {K\over n}  \|H\|_\op  ~\le 2\|H\|_\op.
\end{align*}
It  follows by using \eqref{ineq:PUPA_indep} that, with probability $1-\cO(1/n)$,
\begin{equation}\label{bd_r3_td}
	\wh r_3  \lesssim  \sqrt{\kappa \|\sw\|_\op \over \xi}.
\end{equation}
We point out  that this bound  differs from (\ref{bound:r3}) in that $\errw$ is replaced by the smaller quantity $\|\sw\|_\op$.\\ 

\noindent
{\bf Bound for $\wh r_2$:} We follow the arguments of proving Lemma \ref{lem:r2}. To this end, we first bound from below 
\begin{align}\label{def_sigma_K_td}
	\wt \sigma_K &:=\sigma_{K}(\Pi_{n} \X \wt \U_K) 
	\nonumber\\
	&\ge  \sigma_K(\Pi_{n}\Z A^\T \wt \U_K) - \sigma_1(\Pi_{n}\W \wt \U_K) &&\text{by Weyl's inequality}\nonumber\\
	&\ge  \sigma_K(\Pi_{n}\wt \Z)\sigma_K(\sz^{1/2}A^\T \wt \U_K) - \sigma_1(\W)\\
	&\ge \sqrt{n\over 2}\sigma_K(\sz^{1/2}A^\T \wt \U_K)   - \sqrt{12\gamma^2 n \errw} &&\text{on  }\E_z\cap \E_w^1.\nonumber
\end{align}
Since, with probability $1-\cO(n^{-1})$, 
\begin{align*}
	\sigma_K(\sz^{1/2}A^\T \wt \U_K) &=  \sigma_K(\sz^{1/2}A^\T) - \sigma_1\left(\sz^{1/2}A^\T(P_A - P_{\wt \U_K})\right) \\
	&\ge \sqrt{\lambda_K(A\sz A^\T)} -\sqrt{\lambda_1(A\sz A^\T)}~ \|P_A - P_{\wt \U_K}\|_\op\\
	&\ge \sqrt{\lambda_K(A\sz A^\T)} -\sqrt{\lambda_1(A\sz A^\T)}\sqrt{\kappa \over \xi}\\
	&\gtrsim \sqrt{\lambda_K(A\sz A^\T)} &&\text{by }\xi \gtrsim \kappa^2\\
	& \ge \sqrt{\lambda_K} &&\text{on }(\ref{ineq_lambdas}),
\end{align*}
we conclude 
\begin{equation}\label{bd_sigma_K_td}
	\PP^{\D}\left\{
	\wt \sigma_K^2 \gtrsim n \lambda_K 
	\right\}  = 1- \cO(n^{-1}).
\end{equation}
We start by writing 
\begin{align*}
	\wh r_2 &= \| \wt\U_K (\Pi_n \X \wt\U_K)^+ \Y\|_2\\ 
	&= \|(\wt\U_K^\T \X^\T \Pi_n \X \wt\U_K)^{-1} \wt \U_K^\T   \X^\T \Pi_n \Y\|_2\\
	&\le  \|(\wt\U_K^\T \X^\T \Pi_n \X \wt\U_K)^{-1}  \wt  \U_K^\T \W^\T\Pi_n \Y\|_2+ \|(\wt\U_K^\T \X^\T \Pi_n \X \wt\U_K)^{-1}  \wt \U_K^\T A \Z ^\T\Pi_n \Y\|_2.
\end{align*}
The first term is bounded from above by   
\begin{align*}
	&\|(\wt\U_K^\T \X^\T \Pi_n \X \wt\U_K)^{-1}\|_\op  \|\wt  \U_K^\T \W^\T\Pi_n \Y\|_2\\
	&\le {1\over \wt \sigma_K^2}\|P_{\wt  \U_K} \W^\T\Pi_n \Y\|_2 &&\text{by \eqref{def_sigma_K_td}}\\
	&\le {1\over \wt \sigma_K^2}\left(\|P_A \W^\T\Pi_n \Y\|_2 + \|(P_{\wt \U_K}-P_A) \W^\T\Pi_n \Y\|_2\right).
\end{align*}
The same proof for the last result of Lemma \ref{lem_W} with $P_A$ replaced by $(P_{\wt \U_K}-  P_A)$ yields that, with probability $1-\cO(n^{-1})$, 
\begin{align*}
	{1\over n} \|(P_{\wt \U_K}-  P_A)\W^\T \Pi_n \Y\|_2 &\lesssim   (P_{\wt \U_K}-  P_A) \sqrt{\|\sw\|_\op  {K\log n \over n}}\\
	&\lesssim \sqrt{{\kappa \|\sw\|_\op \over \xi} {K\log n \over n}} &&\text{by \eqref{ineq:PUPA_indep}}\\
	&\lesssim \sqrt{{ \|\sw\|_\op } {K\log n \over n}} &&\text{by }\xi \ge \kappa.
\end{align*} 
By invoking \eqref{bd_sigma_K_td} and Lemma \ref{lem_W}, we have  
\begin{align*}
	\|(\wt\U_K^\T \X^\T \Pi_n \X \wt\U_K)^{-1}  \wt  \U_K^\T \W^\T\Pi_n \Y\|_2 
	&\lesssim \sqrt{1\over \lambda_K\xi^*} \sqrt{K\log n\over n}
\end{align*}
with probability $1-\cO(n^{-1})$.

Regarding the term second term $\rII := \|(\wt\U_K^\T \X^\T \Pi_n \X \wt\U_K)^{-1}  \wt \U_K^\T A \Z ^\T\Pi_n \Y\|_2$, using similar arguments, we have 
\begin{align*}
	\rII \le 2n\|(\wt\U_K^\T \X^\T \Pi_n \X \wt\U_K)^{-1} \wt\U_K^\T A\sz^{1/2}\|_\op \left(
	\sqrt{K\log n\over n} + \min\{1, \Dt\}
	\right)
\end{align*}
with probability $1-\cO(n^{-1})$. Moreover,
\begin{align*}
	& n^2\|(\wt\U_K^\T \X^\T \Pi_n \X \wt\U_K)^{-1} \wt\U_K A\sz^{1/2}\|_\op^2\\
	&\le {n\over 2}\|(\wt\U_K^\T \X^\T \Pi_n \X \wt\U_K)^{-1} \wt\U_K A \Z^\T \Pi_n \Z A^\T \wt \U_K (\wt\U_K^\T \X^\T \Pi_n \X \wt\U_K)^{-1}\|_\op\\
	&\le {n\over 2\wt \sigma_K^2} + {n\over 2\wt\sigma_K^4}\|\Pi_n \W P_{\wt\U_K}\|_\op^2  + {n\over \wt \sigma_K^2}\|(\wt\U_K^\T \X^\T \Pi_n \X \wt\U_K)^{-1} \wt\U_K A\sz^{1/2}\|_\op  \|\wt \Z^\T \Pi_n  \W P_{\wt \U_K}\|_\op.
\end{align*}
Since $\wt\U_K$ is independent of $\W$ and $\Z$, invoking \eqref{bd_sigma_K_td} and Lemma \ref{lem_W} with $P_{\wt \U_K}$ in place of $P_A$  gives
\[
n \|(\wt\U_K^\T \X^\T \Pi_n \X \wt\U_K)^{-1} \wt\U_K A\sz^{1/2}\|_\op  \lesssim
{1\over \sqrt{\lambda_K}} 
\]
with probability $1-\cO(n^{-1})$, implying that 
\[
\rII \lesssim {1\over \sqrt{\lambda_K}} \left(
\min\{1, \Dt\} + \sqrt{K\log n\over n}
\right).
\]
Thus, with probability $1-\cO(n^{-1})$, we conclude 
\begin{equation}\label{bd_r2_td}
	\wh r_2 \lesssim  {1\over \sqrt{\lambda_K}} \left(
	\min\{1, \Dt\} + \sqrt{K\log n\over n}
	\right).
\end{equation}
We emphasize that  the rate in (\ref{bd_r2_td}) above compared to the earlier bound \eqref{bound:r2} is  faster. \\

\noindent{\bf Bound for $\wh r_1$:} The bound of $\wh r_1$ for $B = \wt \U_K$ can be derived by exactly the same arguments of proving Lemma \ref{lem:r1} with $\wt \U_K$ in lieu of $\U_K$. The only difference is that the bound of the term $\|P_{\U_K}  \W ^\T \Pi_n \Y\|_2$ in this case can be improved to 
\[
\PP\left\{
{1\over n}\|P_{\wt\U_K}  \W ^\T \Pi_n \Y\|_2 \lesssim \sqrt{\|\sw\|_\op {K\log n\over n}}
\right\} = 1-\cO(n^{-1})
\]
by Lemma \ref{lem_W} with $P_A$ replaced by $P_{\wt \U_K}$. Consequently, we find that with probability $1-\cO(n^{-1})$,
\begin{align}\label{bd_r1_td}
	\begin{split}    
		\wh r_1 &\lesssim \sqrt{K\log n\over n} + \wh r_2~ \sqrt{ \|\sw\|_\op ~ K\log n\over n} + \wh r_2\wh r_3 +    \sqrt{1\over \xi^*} \sqrt{K\log n\over n}  \\
		&\quad + \left( \sqrt{1\over \xi^*}\sqrt{K\log n\over n} + {\wh r_3 \over \sqrt{\lambda_K}}\right)   \min\{1,\Dt\} + {\wh r_2 \left(\|\sw\|_\op + \wh r_3^2\right)\over \sqrt{\lambda_K}}\\
		&\lesssim \sqrt{K\log n\over n} +   \sqrt{\kappa   \over \xi^*\xi}  \min\{1,\Dt\}
	\end{split}
\end{align}
We  used \eqref{bd_r3_td}, \eqref{bd_r2_td} and $\xi\ge \kappa^2$ to collect terms  and simplify  the expression in the final bound.\\

Finally, putting \eqref{bd_r3_td}, \eqref{bd_r2_td} and \eqref{bd_r1_td} together concludes that for any $a\ge 1$, with probability $1-\cO(n^{-1})$, 
\begin{align*}
	\wh \omega_n(a) &= C\left\{
	\sqrt{a\log n}\left(
	\wh  r_1 + \|\sw\|_\op^{1/2} \wh r_2
	\right) + \wh r_2 \wh r_3 + \sqrt{\log n\over n}  
	\right\}\\
	&\lesssim   \sqrt{a\log n}\left(
	\sqrt{K\log n\over n} +   \sqrt{1   \over \xi^*}  \min\{1,\Dt\}   
	\right),
\end{align*}
completing the proof. \qed

\subsubsection{Proof of Corollary \ref{cor_PCR}}\label{app_proof_cor_PCR}

Since $\sigma^2(1+p/n) \le c'\lambda$ implies $\xi \ge C$ for some constant $C(c')>0$, the proof follows from Theorem \ref{thm_PCR_indep} by choosing $a = \Dt^2/\log n + 1$ for  $\omega_n(a)$ in (\ref{def_omega_n_PCR_indep}) and by noting that 
\[
\omega_n\left(a\right) \asymp  \left(\sqrt{K\log n\over n} +  \min\{1,\Dt\}\sqrt{1\over \xi^*}\right)\sqrt{\log n + \Dt^2}.
\] 
Note that when  $\Dt\to\infty$, the term $\sqrt{\Delta^2}$ in $\omega_n(a)$ gets absorbed by $\exp(-\Dt^2/8)$, reflected in the term $\exp(-(1/8 + o(1))\Dt^2)$.
\qed

\subsection{Proofs of Section \ref{sec_multi_level}}\label{app_sec_multi_level}

For notational convenience, define 
\begin{align}\label{def_Gz_star_ell_k}
	G_z^{(\ell | k)}(z) &:= \left(z - {\a_\ell + \a_k \over 2}\right)^\T \szy^{-1}(\a_\ell - \a_k) + \log{\pi_\ell \over \pi_k},\quad \forall ~ \ell, k \in \cL.
\end{align}
In particular,  for any $\ell \in \cL$, we have 
\begin{align*}
	G_z^{(\ell | 0)}(z)   & ~ =~  \left(z - {\a_\ell + \a_0 \over 2}\right)^\T \szy^{-1}(\a_\ell - \a_0) + \log{\pi_\ell \over \pi_0}\\
	& \overset{(\ref{def_etas_ell})}{=}z^\T \eta^{(\ell)} + \eta^{(\ell)}_0\\
	& \overset{(\ref{def_betas_ell})}{=} {1\over \bar \pi_0 \bar \pi_\ell[1 - (\a_\ell - \a_0)^\T \beta^{(\ell)}]} \left(z^\T \beta^{(\ell)} +\beta_0^{(\ell)}\right).		
\end{align*}
Further recall that  
\begin{equation*}
	\wh G_x^{(\ell | 0)}(x) ~ := ~ {1\over \wt\pi_0\wt\pi_\ell[1 - (\wh\mu_\ell - \wh \mu_0)^\T \wh \theta^{(\ell)}]}\left(x^\T \wh\theta^{(\ell)} + \wh\beta_0^{(\ell)}\right),\quad \forall ~ \ell \in \cL.
\end{equation*}
For any $t\ge 0$,  define the event 
\begin{equation}\label{def_event_multi}
	\E_t = \bigcap_{\ell \in \cL}\left\{
	\left|
	\wh G_x^{(\ell | 0)}(X) - G_z^{(\ell | 0)}(Z) 
	\right| \le t  ~ \mid \D
	\right\}.
\end{equation}
Finally, we write for simplicity 
\begin{equation}\label{def_Dt_ell}
	\Dt_{(\ell | k)} = \|\a_\ell - \a_k\|_{\szy},\qquad \forall ~ k,\ell \in \cL.
\end{equation}

\subsubsection{Proof of Theorem \ref{thm_risk_multi}}

By definition, 
we start with 
\begin{align*}
	&R_x(\wh g_x^*) - R_z^*\\ &=  \sum_{k \in \cL} \pi_k \Bigl\{
	\EE\left[
	\1\{\wh g_x^*(X) \ne k\} \mid Y = k
	\right] - 
	\EE\left[
	\1\{g_z^*(Z) \ne k\} \mid Y = k
	\right]
	\Bigr\}\\
	&=  \sum_{k \in \cL} \pi_k 
	\EE\left[
	\1\{\wh g_x^*(X) \ne k, g_z^*(Z) = k\} \mid Y = k
	\right]-\sum_{k \in \cL} \pi_k 
	\EE\left[
	\1\{\wh g_x^*(X) = k, g_z^*(Z) \ne k\} \mid Y = k
	\right]\\
	&= \sum_{\substack{k,\ell \in \cL\\ k\ne \ell}} \pi_k 
	\EE\left[
	\1\{\wh g_x^*(X) = \ell, g_z^*(Z) = k\} \mid Y = k
	\right]- \sum_{\substack{k,\ell \in \cL\\ k\ne \ell}}  \pi_k 
	\EE\left[
	\1\{\wh g_x^*(X) = k, g_z^*(Z) = \ell\} \mid Y = k
	\right]\\
	&=  \sum_{\substack{k,\ell \in \cL\\ k\ne \ell}} 
	\Bigl\{\pi_k \EE\left[
	\1\{\wh g_x^*(X) = \ell, g_z^*(Z) = k\} \mid Y = k
	\right]-
	\pi_\ell \EE\left[
	\1\{\wh g_x^*(X) = \ell, g_z^*(Z) = k\} \mid Y = \ell\right]
	\Bigr\}.
\end{align*}
Recall that $f_{Z|k}(z)$ is the p.d.f. of $Z=z \mid Y = k$ for each $k\in \cL$.  Repeating arguments in the proof of Theorem \ref{thm_general_risk_explicit} gives
\begin{align*}
	R_x(\wh g_x^*) - R_z^*
	&= \sum_{\substack{k,\ell \in \cL\\ k\ne \ell}} 
	\EE_W\int_{\wh g_x^* = \ell,  g_z^* = k}\left(
	\pi_k f_{Z|k}(z) - \pi_\ell f_{Z|\ell }(z)\right) d z\\
	&=\sum_{\substack{k,\ell \in \cL\\ k\ne \ell}} 
	\EE_W\int_{\wh g_x^* = \ell,  g_z^* = k} \pi_k f_{Z|k}(z) \left(1 - \exp\left\{G_z^{(\ell | k)}(z)\right\}\right) dz
\end{align*}
with $G_z^{(\ell | k)}(z)$ defined in (\ref{def_Gz_star_ell_k}).  Since 
\begin{equation}\label{eq_G_inter}
	G_z^{(\ell | k)}(z)
	= G_z^{(\ell | 0)}(z)  - G_z^{(k | 0)}(z),
\end{equation}
the event $\{\wh g_x^*(X) = \ell,  g_z^*(Z) = k\} \cap \E_t$ implies
\[
0 > G_z^{(\ell | k)}(z)  \overset{\E_t}{\ge} \wh G_x^{(\ell | 0)}(X) - \wh G_x^{(k | 0)}(X)- 2t \ge  -2t,\quad \forall ~ t>0.
\]
By repeating the arguments of analyzing term $(I)$ in the proof of Theorem \ref{thm_general_risk_explicit}, we obtain that, for any $t>0$,
\begin{align}\label{bd_diff_R}\nonumber
	&R_x(\wh g_x^*) - R_z^*\\ \nonumber
	&\le \sum_{\substack{k,\ell \in \cL\\ k\ne \ell}} 
	\left\{    2t \pi_k \EE_Z\left[ \1 \{-2t \le G_z^{(\ell | k)}(Z) \le 0 \mid Y = k\} \right] + \pi_k  \PP(\E_t^c \mid Y = k)
	\right\}\\
	&\le (L-1)  \sum_{k\in \cL} 2  \pi_k  t   \max_{\ell \in \cL \setminus \{k\}}  \left[
	\Phi\left(
	R^{(\ell | k)} 
	\right)- \Phi\left(
	R^{(\ell | k)}  -{2t\over \Dt_{(\ell | k)}}
	\right)
	\right] +  (L-1)  \PP(\E_t^c )\\\nonumber
	&\le (L-1) \sum_{k\in \cL} 4 \pi_k   t^2  \max_{\ell \in \cL \setminus \{k\}} {1\over \Dt_{(\ell | k)}} \exp\left(	-{m_{(\ell | k)}^2 \over 2}
	\right) + (L-1) \PP(\E_t^c)
\end{align}
where  
\[
R^{(\ell | k)}  =  { \Dt_{(\ell | k)} \over 2} - {\log{\pi_\ell \over \pi_k} \over  \Dt_{(\ell | k)}},\qquad m_{(\ell | k)} \in \left[R^{(\ell | k)} - {2t\over \Dt_{(\ell | k)}}, ~ R^{(\ell | k)}  \right].
\]
The penultimate step uses the fact that 
$$
G_z^{(\ell | k)}(Z) \mid Y = k  ~  \sim ~ N \left(
-\Dt_{(\ell | k)} R^{(\ell | k)}, ~ \Dt_{(\ell | k)}^2
\right)
$$
while the last step applies the mean-value theorem. 
By choosing  
$$
t^* = (1 + \Dt^4) \omega_n
$$
and 
invoking condition (\ref{cond_Dt_multi}) and $(1 +  \Dt^2)\omega_n = o(1)$, we find that:
\begin{enumerate}
	\item[(a)] If $\Dt \asymp 1$, then 
	\[
	R_x(\wh g_x^*) - R_z^*  ~ \lesssim ~  L \omega_n^2+  L\PP(\E_{t^*}^c).
	\]
	\item[(b)] If $\Dt \to \i$, then $\Dt^2\omega_n = o(1)$ ensures that $m_{(\ell|k)} \asymp \Dt$ hence 
	\[
	R_x(\wh g_x^*) - R_z^*  ~  \lesssim ~ L \omega_n^2 e^{-c \Dt^2 + o(\Dt^2)}+  L \PP(\E_{t^*}^c).
	\]
	\item[(c)] If $\Dt \to 0$, then $t^* \asymp \omega_n$ and 
	\[
	R_x(\wh g_x^*) - R_z^*  ~  \lesssim ~ L {\omega_n^2  \over \Dt} +  L \PP(\E_{t^*}^c).
	\]
	For $\Dt \to 0$, by (\ref{bd_diff_R}), we also have 
	\[
	R_x(\wh g_x^*) - R_z^*  ~  \lesssim ~ L \min\left\{{\omega_n^2  \over \Dt}, \omega_n\right\} +  L \PP(\E_{t^*}^c).
	\]
\end{enumerate}
In view of cases (a) -- (c), since the event $\{\wh\omega_n \le \omega_n\}$ implies 
\[
\PP(\E_{t^*}^c) \le  \PP\left\{
\max_{\ell\in \cL}\left|
\wh G_x^{(\ell | 0)}(X) - G_z^{(\ell | 0)}(Z) 
\right|  \ge  (1+\Dt^4)  \wh \omega_n ~\mid \D
\right\},
\]
it remains to prove that, with probability $1-\cO(n^{-1})$,  the right-hand side of the above display is no greater than $n^{-1}e^{-\Dt^2}$. This is  proved  by combining Lemmas \ref{lem_G_star_dev} and \ref{lem_td_omega}. \qed 

\subsubsection{Lemmas used in the proof of Theorem \ref{thm_risk_multi}}

The following lemma establishes the probability tail of the event $\E_t$ defined in (\ref{def_event_multi}) for $t = \wt \omega_n$, a random sequence defined below whose randomness only depends on $\D$. Recall $\wh r_1$ and $\wh r_2$ from (\ref{def_r_hat_multi}).  Set 
\begin{align}\label{def_td_omega_n}
	\wt\omega_n  &=  \max_{\ell\in \cL} C\left\{ {\wh r_1 + \|\sw\|_\op^{1/2} \wh r_2  \over |\wt\pi_0\wt\pi_\ell[1 - (\wh\mu_\ell - \wh \mu_0)^\T \wh \theta^{(\ell)}]|}\left(\sqrt{\log n} + \Dt\right) \right.\\\nonumber
	&\left. \qquad + \left|{\wh\beta_0^{(\ell)} - \beta_0^{(\ell)} + {1 \over 2}\left(\a_1 + \a_0\right)^\T\left(A^\T \wh \theta ^{(\ell)}- \beta^{(\ell)}\right)\over \wt\pi_0\wt\pi_\ell[1 - (\wh\mu_\ell - \wh \mu_0)^\T \wh \theta^{(\ell)}]}\right|\right.\\\nonumber
	&\left. \qquad + 
	\left|{\wt\pi_0\wt\pi_\ell[1 - (\wh\mu_\ell - \wh \mu_0)^\T \wh \theta^{(\ell)}] - \bar\pi_0\bar\pi_\ell[1 - (\a_\ell - \a_0)^\T \beta^{(\ell)}] \over |\wt\pi_0\wt\pi_\ell[1 - (\wh\mu_\ell - \wh \mu_0)^\T \wh \theta^{(\ell)}]}\right| \Dt \left(\sqrt{\log n} + \Dt\right)
	\right\}.
\end{align}

\begin{lemma}\label{lem_G_star_dev}
	Under conditions of Theorem \ref{thm_risk_multi}, we have,
	\[
	\PP\left\{\max_{\ell\in \cL}\left|
	\wh G_x^{(\ell | 0)}(X) - G_z^{(\ell | 0)}(Z) 
	\right|  \ge   \wt\omega_n   ~\mid \D \right\} \le  n^{-1}e^{-\Dt^2}.
	\]
\end{lemma}
\begin{proof}
	Pick any $\ell \in \cL$.
	By definition, 
	\[
	\left| \wh G_x^{(\ell | 0)}(X) - G_z^{(\ell | 0)}(Z)\right| \le \rI + \rII + \rIII
	\]
	where 
	\begin{align*}
		\rI &~ =  ~  \left| {X^\T \wh\theta^{(\ell)}  - Z^\T \beta^{(\ell)}  - {1\over 2}(\a_1 + \a_0)^\T(A^\T \wh \theta^{(\ell)} - \beta^{(\ell)})  \over \wt\pi_0\wt\pi_\ell[1 - (\wh\mu_\ell - \wh \mu_0)^\T \wh \theta^{(\ell)}]}\right|,\\
		\rII & ~ = ~ \left| {\wh\beta_0^{(\ell)} - \beta_0^{(\ell)} +  {1\over 2}(\a_1 + \a_0)^\T(A^\T \wh \theta^{(\ell)} - \beta^{(\ell)}) \over \wt\pi_0\wt\pi_\ell[1 - (\wh\mu_\ell - \wh \mu_0)^\T \wh \theta^{(\ell)}]}\right|,\\
		\rIII &~ = ~ \left|{1 \over \wt\pi_0\wt\pi_\ell[1 - (\wh\mu_\ell - \wh \mu_0)^\T \wh \theta^{(\ell)}]} - {1\over \bar\pi_0\bar \pi_\ell[1 - (\a_\ell - \a_0)^\T \beta^{(\ell)}]}\right| 
		\left|Z^\T \beta^{(\ell)} +\beta_0^{(\ell)}\right|\\
		&\overset{(\ref{eq_eta_beta_multi})}{=}  \left|{\wt\pi_0\wt\pi_\ell[1 - (\wh\mu_\ell - \wh \mu_0)^\T \wh \theta^{(\ell)}] - \bar\pi_0\bar\pi_\ell[1 - (\a_\ell - \a_0)^\T \beta^{(\ell)}]\over \wt\pi_0\wt\pi_\ell[1 - (\wh\mu_\ell - \wh \mu_0)^\T \wh \theta^{(\ell)}]}\right| 
		\left|Z^\T \eta^{(\ell)} +\eta_0^{(\ell)}\right|.
	\end{align*}
	First, notice that the numerator of $\rI$ is bounded from above by 
	\[
	\left|W^\T \wh\theta^{(\ell)}\right|+ \left|\left(Z - {1\over 2}(\a_\ell + \a_0 )\right)^\T(A^\T \wh \theta^{(\ell)} - \beta^{(\ell)})\right|,
	\]
	which, by the arguments of proving Proposition \ref{prop_risk} and by conditioning on $Y = k$ for any $k\in \cL$, with probability $1-\cO(n^{-a})$ for any $a>0$, is no greater than 
	\begin{align*}
		&C\left(\sqrt{a \log n} + \left\|
		\a_k - {1\over 2}(\a_\ell + \a_0)
		\right\|_{\sz^{(\ell)}} \right) \left\|\bigl[\sz^{(\ell)}\bigr]^{1/2}\bigl(A^\T \wh \theta^{(\ell)} - \beta^{(\ell)}\bigr)\right\|_2\\
		&\quad  +   C\sqrt{a \log n} \|\wh\theta^{(\ell)}\|_2\|\sw\|_\op^{1/2}\\
		&\lesssim \left(\sqrt{a \log n} + \max_{k\in \cL} \Dt_{(k|0)} +1\right) \left\|\bigl[\sz^{(\ell)}\bigr]^{1/2}\bigl(A^\T \wh \theta^{(\ell)} - \beta^{(\ell)}\bigr)\right\|_2+   \sqrt{a \log n} \|\wh\theta^{(\ell)}\|_2\|\sw\|_\op^{1/2}\\
		&\lesssim \left(\sqrt{a \log n} + \Dt+1 \right) \left(\left\|\bigl[\sz^{(\ell)}\bigr]^{1/2}\bigl(A^\T \wh \theta^{(\ell)} - \beta^{(\ell)}\bigr)\right\|_2+   \|\wh\theta^{(\ell)}\|_2\|\sw\|_\op^{1/2}\right).
	\end{align*}
	In the second step, we also used 
	\[
	\|\a_k- \a_0\|_{\sz^{(\ell)}}^2  \le  \|\a_k - \a_0\|_{\szy}^2 \left\|\szy^{1/2} [\sz^{(\ell)}]^{-1}\szy^{1/2}\right \|_\op \le \Dt_{(k|0)}^2, \quad \forall ~ k\in \cL.
	\]
	Again, by the arguments of proving Proposition \ref{prop_risk}, with probability $1-\cO(n^{-a})$ for any $a>0$, 
	\begin{align*}
		Z^\T \eta^{(\ell)} +\eta_0^{(\ell)}  &~ \lesssim  \|\a_\ell - \a_0\|_{\szy}\sqrt{a\log n}  + \left|
		\left(\a_k - {\a_\ell + \a_0 \over 2}\right) \szy^{-1}(\a_\ell-\a_0)
		\right|\\
		&~ \lesssim \Dt_{(\ell|0)} \left(
		\sqrt{a \log n} +  \Dt_{(\ell|0)} +  \Dt_{(k|0)}
		\right)\\
		&~ \lesssim \Dt \left(\sqrt{a\log n} + \Dt\right).
	\end{align*}
	Taking $a = C + \Dt^2/\log n$ for some positive constant $C$ in these two bounds yields the claim.
\end{proof}

\smallskip

We proceed to bound from above $\wt \omega_n$ defined in (\ref{def_td_omega_n}) by $\wh \omega_n$ in (\ref{def_wh_omega_n}).  Recall that 
\begin{equation*}
	\wh \omega_n  = C\sqrt{\log n}\left(\wh r_1 + \|\sw\|_\op^{1/2}\wh r_2 + \wh r_2 \wh r_3 +  \sqrt{L \over n} \right).
\end{equation*}

\begin{lemma}\label{lem_td_omega}
	Under conditions of Theorem \ref{thm_risk_multi}, we have 
	\[
	\PP^{\D}\left\{
	\wt \omega_n \lesssim (1 + \Dt^4)\wh \omega_n 
	\right\} = 1-\cO(n^{-1}).
	\]
\end{lemma}
\begin{proof}
	We first bound from above the numerators of the last two terms in $\wt\omega_n$ defined 	in (\ref{def_td_omega_n}).  By Lemma \ref{lem_pi_hat} and  $\pi_k \asymp 1 / L$ for all $k\in \cL$, we have 
	\[
	\PP\left\{
	\max_{\ell\in \cL} |\wh \pi_\ell - \pi_\ell| \lesssim \sqrt{ \log n \over nL}
	\right\} = 1-\cO(Ln^{-C}).
	\]
	for some constant $C>1$. With the same probability, using $L\log n \lesssim n$ further yields that, for any $\ell \in \cL$, 
	$$
	\wh \pi_\ell \asymp   {1 \over L}, \qquad 
	n_\ell  \asymp {n\over L}
	$$
	as well as 
	\begin{align*}
		\left|
		\wt \pi_\ell - \bar \pi_\ell 
		\right| = \left|
		{\wh \pi_\ell  - \pi_\ell \over \wh \pi_\ell+\wh \pi_0}\right| +  	\left| {\pi_\ell( \wh \pi_\ell - \pi_\ell +\wh \pi_0- \pi_0) \over ( \wh \pi_\ell+\wh \pi_0)(\pi_\ell + \pi_0)}
		\right| \lesssim \sqrt{L\log n \over n},\quad \wt \pi_\ell \asymp 1.
	\end{align*}
	Pick any $\ell \in \cL$. 
	By following the same arguments of proving Lemma \ref{lem_beta_0} and using the condition $KL \log n \lesssim n$,  we have, with probability $1-\cO(n^{-C})$,
	\begin{align}\label{bd_last_two}\nonumber
		&\max\left\{
		\left|\wh\beta_0^{(\ell)} - \beta_0^{(\ell)} + {1\over 2}(\a_1 + \a_0)^\T(A^\T \wh \theta^{(\ell)} - \beta^{(\ell)})\right|,\right.\\\nonumber
		&\quad \qquad ~ \left. \left|\wt\pi_0\wt\pi_\ell[1 - (\wh\mu_\ell - \wh \mu_0)^\T \wh \theta^{(\ell)}] -\bar \pi_0\bar \pi_\ell[1 - (\a_\ell - \a_0)^\T \beta^{(\ell)}]\right|
		\right\}\\
		&\quad  ~ \lesssim ~ \wh r_1 +   \|\sw\|_{\op}^{1/2} \wh r_2 +\wh r_2 \wh r_3+  \sqrt{L\log n \over n} ~\le  ~ \omega_n.
	\end{align}
	By taking the union bounds over $\ell\in \cL$, the above bound also holds for all $\ell \in \cL$ with probability $1-\cO(Ln^{-C})$. 
	
	It remains to bound from below $\left|\wt\pi_0\wt\pi_\ell[1 - (\wh\mu_\ell - \wh \mu_0)^\T \wh \theta^{(\ell)}] \right|$. To this end, repeating arguments of proving Lemma \ref{lem_rule} gives 
	\[
	\Cov(Z, \1\{Y =\ell\} \mid Y \in \{0,\ell\}) = \bar \pi_0\bar \pi_\ell(\a_\ell - \a_0),
	\]
	and, by recalling that $\sz^{(\ell)} = \Cov(Z \mid Y \in \{0,\ell\})$,
	\[
	\|\a_\ell - \a_0\|_{\sz^{(\ell)}}^2  =  {\|\a_\ell -\a_0\|_{\szy}^2 \over 1 + \bar \pi_0\bar \pi_\ell \|\a_\ell -\a_0\|_{\szy}^2}\overset{(\ref{def_Dt_ell})}{=} {\Dt_{(\ell|0)}^2 \over 1 + \bar \pi_0\bar \pi_\ell \Dt_{(\ell|0)}^2}.
	\]
	It then follows that
	\begin{align*}
		\bar\pi_0\bar\pi_\ell[1 - (\a_\ell - \a_0)^\T \beta^{(\ell)}] &=  \bar\pi_0 \bar\pi_\ell\left[ 1 -\bar\pi_0\bar \pi_\ell \|\a_\ell - \a_0\|_{\sz^{(\ell)}}^2\right] = {\bar\pi_0\bar \pi_\ell \over 1+ \bar\pi_0 \bar \pi_\ell \Dt_{(\ell|0)}^2}.
	\end{align*}
	Thus,  by (\ref{bd_last_two}), condition (\ref{cond_Dt_multi}) and condition  
	$
	(1 + \Dt^2) \omega_n= o(1),
	$  
	we find that, 
	with probability $1-\cO(L n^{-C})$, 
	\begin{align*}
		\left|\wt\pi_0\wt\pi_\ell[1 - (\wh\mu_\ell - \wh \mu_0)^\T \wh \theta^{(\ell)}] \right| &\gtrsim {\bar\pi_0 \bar\pi_\ell \over 1+\bar \pi_0 \bar\pi_\ell \Dt_{(\ell|0)}^2} - \omega_n\gtrsim {\bar\pi_0 \bar\pi_\ell \over 1+ \bar\pi_0 \bar\pi_\ell \Dt_{(\ell|0)}^2}.
	\end{align*}
	Combining the last display with (\ref{bd_last_two}) gives that, with probability $1-\cO(n^{-1})$,
	\begin{align*}
		\wt\omega_n &\lesssim    \max_{\ell\in \cL} (1 + \Dt^2) \left\{ \left(\sqrt{\log n} +\Dt\right) \left(\wh r_1 + \wh r_2 \|\sw\|_\op^{1/2}\right) \right.\\\nonumber
		&\left. \quad + \left(\wh r_1+   \|\sw\|_{\op}^{1/2} \wh r_2 + \wh r_2 \wh r_3 +  \sqrt{L\log n \over n}\right)\left(1 + \Dt \sqrt{\log n} + \Dt^2\right)
		\right\}\\
		&\lesssim (1 + \Dt^4) \omega_n,
	\end{align*}
	completing the proof. 
\end{proof}

\subsubsection{Proof of Corollary \ref{cor_risk_multi}}\label{app_proof_cor_risk_multi}
In view of Theorem \ref{thm_risk_multi}, we only need to bound from above $\wh r_1$, $\wh r_2$ and $\wh r_3$ for each choice of $B$. Inspecting the proofs of Lemmas \ref{lem:r3}, \ref{lem:r2} and \ref{lem:r1} reveals that the same conclusions 
therein hold with $K$ replaced by $KL$. Consequently, repeating the steps in the proofs of Theorems \ref{thm_PCR} \& \ref{thm_PCR_indep} yields the desired result.  \qed

\section{Proof of the minimax lower bounds of the excess risk}\label{app_proof_thm_lowerbound}

\begin{proof}[Proof of Theorem \ref{thm_lowerbound}]
	
	Recall that $\pi_0 = \pi_1 = 1/2$. It suffices to consider
	$
	\a_1 = -\a_0  = \a.
	$
	Further recall that  $K/(n\vee p) \le c_1$, $\sigma^2/\lambda\le c_2$ and $\sigma^2 p / (\lambda n)\le c_3$ for sufficiently small positive constants $c_1, c_2$ and $c_3$.

	To prove Theorem \ref{thm_lowerbound}, it suffices to consider the Gaussian case. Specifically, for any $\theta = (A,\szy, \sw, \a, -\a, 1/2, 1/2)$, consider 
	\begin{equation}\label{distr_X}
		X \mid Y = 1 \sim N_p(\mut, \Sigt)\quad \text{and} \quad X \mid Y = 0 \sim N_p(-\mut, \Sigt)
	\end{equation}
	with
	\[
	\mut = A\a,\qquad  \Sigt = A\szy A^\T + \sw.
	\]
	In this case,  the Bayes rule of using $X$ is
	\begin{align}\label{def_gx_star}
		g^*_\theta(x) = \1\left\{
		G^*_\theta(x) \ge 0
		\right\} 
		= \1\left\{2 x^\T\Sigt^{-1}\mut\ge 0\right\}.
	\end{align}
	For any classifier $\wh g: \RR^p \to \{0, 1\}$, one has 
	\[
	R_x(\wh g) - R_z^*  = R_x(\wh g) - R_x(g_\theta^*) + R_x(g_\theta^*) - R_z^*.
	\]
	Lemma \ref{lem_risk} together with $\sigma^2/ \lambda \le c_2$ ensures that, for any $\theta \in \Theta(\lambda, \sigma, \Dt) $,
	\begin{equation}\label{lb_0}
		R_x(g_\theta^*) - R_z^* \gtrsim {\sigma^2 \over \lambda} \Dt ~ \exp\left(-{\Dt^2\over 8}\right).
	\end{equation}
	Note that $g^*_\theta$ has the smallest risk over all measurable functions $\wh g: \RR^p \to \{0, 1\}$.  We proceed to bound from below $R_x(\wh g) - R_x(g_\theta^*)$ by splitting into two scenarios depending on the magnitude of $\Dt$.\\

	{\bf Case 1: $\Dt \gtrsim 1$.}  We may assume $\Dt \ge 2$ for simplicity.
	It   suffices to show
	\begin{equation}\label{lb_goal}
		\inf_{\wh g}\sup_{\theta \in \Theta(\lambda, \sigma, \Dt)} \PP^{\D}_{\theta}\left\{
		R_x(\wh g) - R_x(g_\theta^*) \ge {\eta \over \Dt } \exp\left(-{\Dt^2 \over 8} +\delta\right)
		\right\} \ge c_0,
	\end{equation}
	where 
	\begin{equation}\label{def_delta_n}
		\delta =  {\sigma^2\over \sigma^2 + \lambda}{\Dt^2 \over 8}
	\end{equation}
	and
	\begin{equation}\label{def_eta_n}
		\eta = C\left[{K\over n} + {\sigma^4 (p-K)\over \lambda^2 n} 			\right] .
	\end{equation}
	We take the leading constant $C>0$ in $\eta$ small enough such that
	\begin{itemize}
		\item[(a)] $ C <  3( c_1+ c_2 c_3)$, where $c_1,c_2,c_3$ are defined in Theorem \ref{thm_lowerbound}.  
		\item[(b)] $ C < \min(C_1, C_2)/6$, where $C_1$ and $C_2$ are defined in (\ref{lb_1}) and (\ref{lb_2}).
	\end{itemize}
	These two requirements will become apparent soon.
	
	To prove (\ref{lb_goal}), we first introduce another loss function
	\begin{equation}\label{def_L}
		L_{\theta}(\wh g) = \PP_{\theta}\{
		\wh g(X) \ne g^*_\theta(X)
		\}.
	\end{equation}
	We proceed to bound $R_x(\wh g) - R_x(g_\theta^*)$  from below by using $L_{\theta}(\wh g)$.  By following the same arguments in the proof of Theorem \ref{thm_general_risk} with $G_z(Z)$ replaced by $G^*_\theta(X)$, one can deduce that 
	\begin{align*}
		R_x(\wh g) - R_x(g_\theta^*) & = \PP_{\theta}\{\wh g(X) \ne Y\} - \PP_{\theta}\{g_\theta^*(X) \ne Y\}  := \rI + \rII
	\end{align*}
	where 
	\begin{align*}
		&\rI = \pi_0\EE_\theta \left[
		\1\{\wh g(X) = 1,  G_\theta^*(X) < 0 \}\left(1 - \exp(G_\theta^*(X)\right) \mid Y = 0
		\right],\\
		&\rII  = \pi_1\EE_\theta \left[
		\1\{\wh g(X) = 0,  G_\theta^*(X) \ge 0 \}\left(1 - \exp(-G_\theta^*(X)\right) \mid Y = 1
		\right].
	\end{align*}
	For any $t > 0$, 
	\begin{align*}
		\rI  & \ge  \pi_0\EE_\theta \left[
		\1\{\wh g(X) = 1,  G_\theta^*(X) \le -t\}\left(1 - \exp(G_\theta^*(X)\right) \mid Y = 0
		\right]\\
		& \ge \pi_0\left(1-e^{-t} \right)\EE_\theta\left[
		\1\{\wh g(X) = 1,  G_\theta^*(X) \le -t \}  \mid Y = 0
		\right]\\
		& \ge  \pi_0\left(1-e^{-t} \right)\Bigl\{\EE_\theta\left[
		\1\{\wh g(X) = 1,  G_\theta^*(X) < 0 \}  \mid Y = 0
		\right] - \PP_{\theta}\left(
		-t \le G_\theta^*(X) < 0 \mid Y = 0
		\right)\Bigr\}\\
		& =   \pi_0\left(1-e^{-t} \right)\Bigl\{\EE_\theta\left[
		\1\{\wh g(X) = 1,  g_\theta^*(X) = 0 \}  \mid Y = 0
		\right] - \PP_{\theta}\left(
		-t \le G_\theta^*(X) < 0 \mid Y = 0
		\right)\Bigr\}.
	\end{align*}
	Similarly, 
	\[
	\rII \ge  \pi_1\left(1-e^{-t} \right)\Bigl\{\EE_\theta\left[
	\1\{\wh g(X) = 0,  g_\theta^*(X) = 1 \}  \mid Y = 1
	\right] - \PP_{\theta}\left(
	0 \le G_\theta^*(X) \le  t \mid Y = 1
	\right)\Bigr\}.
	\]
	Combine these two lower bounds, the identity $\pi_0=\pi_1 = 1/2$ 
	and the inequality $1-\exp(-t) \ge t/2$ for $0<t<1$
	to obtain,
	\begin{align*}
		R_x(\wh g) - R_x(g_\theta^*)  &\ge 
		{t\over 2}
		\left\{L_\theta (\wh g)  -  {1\over 2} ~ \PP_{\theta}\left(
		0 \le G_\theta^*(X) \le  t \mid Y = 1
		\right)\right.\\
		&\hspace{3.1cm} \left. ~ - {1\over 2}~ \PP_{\theta}\left(
		-t \le G_\theta^*(X) < 0 \mid Y = 0
		\right)\right\},
	\end{align*}
	for any $0<t<1$.
	From  (\ref{def_Dt_x}), we see that $\Dt_x^2 = 4\mut^\T \Sigt^{-1}\mut$, and we easily find
	\[
	\left(	G_\theta^*(X) \mid Y=0 \right) = \left( 2 X^\T\Sigt^{-1}\mut  \mid Y = 0\right) \sim N\left(-{1\over 2}\Dt_x^2, \Dt_x^2\right),
	\]
	and, similarly, 
	\[
	G_\theta^*(X)  \mid Y = 1 \sim N\left({1\over 2}\Dt_x^2, \Dt_x^2\right).
	\]
	An application of the mean value theorem yields
	\begin{equation}\label{lb_key}
		R_x(\wh g) - R_x(g_\theta^*) \ge   
		{t\over 2} \left(L_\theta (\wh g) - {t\over 2\Dt_x}
		\varphi(R_t) -  {t\over 2\Dt_x} \varphi(L_t)
		\right)
	\end{equation}
	for 
	\[
	R_t \in \left[
	{1\over 2}\Dt_x - {t\over \Dt_x}, ~ {1\over 2}\Dt_x
	\right], \qquad L_t \in \left[
	-{1\over 2}\Dt_x, ~ -{1\over 2}\Dt_x +  {t\over \Dt_x}
	\right], \qquad 0<t<1.
	\]
	Then, for $0<t<\min(1,\Dt_x^2)$, we easily find from (\ref{lb_key}) that
	\begin{align*}
		\frac{t}{2\Dt_x} \left\{ 
		\varphi(R_t) + \varphi(L_t) \right\} &\le 
		\frac{t}{\Dt_x}\sqrt{e\over 2\pi}
		\exp\left(  -\frac{\Dt_x^2}{8} 
		\right).
	\end{align*}
	Hence, for any $0<t\le \min(1, \Dt_x^2/2)$, we proved that
	\begin{align}\label{display_risk_t}
		&\inf_{\wh g}\sup_{\theta \in \Theta} \PP^{\D}_{\theta}\left\{
		R_x(\wh g) - R_x(g_\theta^*) \ge {\eta \over \Dt } \exp\left(-{\Dt^2 \over 8} +\delta\right)
		\right\}\\\nonumber
		& \ge \inf_{\wh g}\sup_{\theta \in \Theta} \PP^{\D}_{\theta}\left\{
		{t\over 2}
		\left(L_\theta (\wh g) - {t\over 2\Dt_x}
		\varphi(R_t) -  {t\over 2\Dt_x} \varphi(L_t)
		\right) \ge {\eta \over \Dt } \exp\left(-{\Dt^2 \over 8} +\delta\right)
		\right\}\\\nonumber
		&\ge 
		\inf_{\wh g}\sup_{\theta \in \Theta} \PP^{\D}_{\theta}\left\{
		L_\theta (\wh g) \ge {2\eta \over \Dt t } \exp\left(-{\Dt^2 \over 8} +\delta \right) + \frac{t}{\Dt_x}\sqrt{e\over 2\pi}
		\exp\left(  -\frac{\Dt_x^2}{8} 
		\right)
		\right\}
	\end{align}
	Next, choose
	\begin{align*}
		t^* =  \left(\pi \over e\right)^{1/4} 2\sqrt{\eta} \overset{(i)}{\le } 1
	\end{align*}
	with $\eta$ defined in (\ref{def_eta_n}).
	Inequality $(i)$  holds by using $K  /n \le c_1$, $\sigma^2/\lambda \le c_2$, $\sigma^2 p/(\lambda n)\le c_3$ and requirement (a) of the constant $C$ in the definition (\ref{def_eta_n}) of $\eta$.
	In the proof of the lower bounds (\ref{lb_1}) and (\ref{lb_2}) below, we  consider subsets of $\Theta(\lambda, \sigma, \Dt)$ such that, for any $\theta\in \Theta(\lambda, \sigma, \Dt)$,
	\begin{equation}\label{eq_Dt_Dtx}
		\Dt_x^2 = {\lambda \over \sigma^2 + \lambda} \Dt^2.
	\end{equation}
	This implies 
	\begin{equation}\label{eq_Dt_Dtx2}
		{\Dt^2 \over 2}\le \Dt_x^2 \le \Dt^2,
	\end{equation}
	provided that $\sigma^2 / \lambda \le c_2\le 1$,
	and, using (\ref{def_delta_n}),
	\begin{equation}\label{eq_Dt_Dtx3}
		-    {\Dt^2 \over 8}+\delta^2 = - { \Dt_x^2\over8}.
	\end{equation}
	Note that (\ref{eq_Dt_Dtx}) further implies 
	$t^* \le 1\le \Dt^2/4 \le \Dt_x^2/2$. Then, by plugging $t^*$ into (\ref{display_risk_t}) and using  (\ref{eq_Dt_Dtx2}) and (\ref{eq_Dt_Dtx3}), we find 
	\begin{align}\label{final}
		&\inf_{\wh g}\sup_{\theta \in \Theta} \PP^{\D}_{\theta}\left\{
		R_x(\wh g) - R_x(g_\theta^*) \ge {\eta \over \Dt } \exp\left(-{\Dt^2 \over 8} +\delta\right)
		\right\}\\
		&\ge 
		\inf_{\wh g}\sup_{\theta \in \Theta} \PP^{\D}_{\theta}\left\{
		L_\theta (\wh g) \ge \left(e\over \pi\right)^{1/4}{\sqrt\eta\over \Dt} \exp\left(-{\Dt^2 \over 8} +\delta\right) +\left(e\over \pi\right)^{1/4}{\sqrt{\eta} \over \Dt_x\sqrt 2}
		\exp\left(  -\frac{\Dt_x^2}{8}  
		\right)
		\right\}\nonumber\\
		&= \inf_{\wh g}\sup_{\theta \in \Theta} \PP^{\D}_{\theta}\left\{
		L_\theta (\wh g) \ge 2\left(e\over \pi\right)^{1/4} {\sqrt{\eta} \over \Dt}
		\exp\left(   -\frac{\Dt_x^2}{8}  
		\right)
		\right\}.\nonumber
	\end{align}
	In the next two sections we prove the inequalities 
	\begin{align}\label{lb_1}
		&\inf_{\wh g}\sup_{\theta \in \Theta} \PP^{\D}_{\theta}\left\{
		L_\theta(\wh g) \ge C_1~ 
		\sqrt{K\over n}~  {1\over \Dt} \exp\left(   -\frac{\Dt_x^2}{8}  
		\right)
		\right\}\ge (1+c_0)/2,\\\label{lb_2}
		&\inf_{\wh g}\sup_{\theta \in \Theta} \PP^{\D}_{\theta}\left\{
		L_\theta(\wh g) \ge C_2\sqrt{\sigma^4 (p-K) \over \lambda^2 n} 
		\exp\left(   -\frac{\Dt_x^2}{8}  
		\right)
		\right\}\ge (1+c_0)/2,
	\end{align}
	for some positive constants $C_1$ and $C_2$.
	By using requirement (b) for the leading constant $C$ 
	in the definition (\ref{def_eta_n}) of $\eta$, we can conclude  from the final lower bound (\ref{final})    the proof of Theorem \ref{thm_lowerbound} for $\Dt\gtrsim 1$.\\
	
	{\bf Case 2: $\Dt = o(1)$.}  We further consider two cases and recall that 
	\[
	\omega_n^* = \sqrt{{K\over n}+ {\sigma^2\over \lambda} \Dt^2 + {\sigma^2\over \lambda} {\sigma^2 p \over \lambda n}\Dt^2}.
	\]
	When $\omega_n^* = o (\Dt)$,
	we now prove the lower bound $(\omega_n^*)^2/\Dt$. By choosing 
	\begin{equation}\label{def_t1}
		t_1 = c_t
		\sqrt{{K\over n} +  {\sigma^4 (p-K)\over \lambda^2 n}\Dt^2}  ~ \le ~  1
	\end{equation}
	in (\ref{lb_key}) for some constant $c_t>0$ and by using $\varphi(R_{t_1})\le 1$, $ \varphi(L_{t_1})\le 1$ and $\Dt_x \le \Dt$, we find 
	\[
	R_x(\wh g) - R_x(g_\theta^*) \ge 
	{c_t\over 2} L_\theta (\wh g)  
	\sqrt{{K\over n} +  {\sigma^4 (p-K)\over \lambda^2 n}\Dt^2}- {c_t^2\over 2\Dt}\left[{K\over n} +  {\sigma^4 (p-K)\over \lambda^2 n}\Dt^2\right].
	\]
	From (\ref{lb_1}) and (\ref{lb_2}), it follows that 
	\begin{align*}
		\inf_{\wh g}\sup_{\theta \in \Theta} \PP^{\D}_{\theta}&\left\{
		R_x(\wh g) - R_x(g_\theta^*) \ge 
		{c_t C_3\over 2}\left[{K\over n}{1\over \Dt} +  {\sigma^4 (p-K)\over \lambda^2 n}\Dt\right]\exp\left(
		-{\Dt_x^2 \over 8}
		\right)
		\right.\\
		&\hspace{3.5cm} \left.- {c_t^2\over 2}\left[{K\over n}{1\over \Dt} +  {\sigma^4 (p-K)\over \lambda^2 n}\Dt\right]
		\right\}\ge c_0
	\end{align*}
	for some constant $C_3>0$ depending on $C_1$ and $C_2$. Therefore, by using $\Dt_x \ge \Dt / 2$ and $\Dt = o(1)$ and taking $c_t$ sufficiently small,  we conclude 
	\begin{align*}
		\inf_{\wh g}\sup_{\theta \in \Theta} \PP^{\D}_{\theta}&\left\{
		R_x(\wh g) - R_x(g_\theta^*) \ge 
		{c_1C_3}\left[{K\over n}{1\over \Dt} +  {\sigma^4 (p-K)\over \lambda^2 n}\Dt\right]
		\right\}\ge c_0.
	\end{align*}
	The above display together with (\ref{lb_0}) proves the lower bound $(\omega_n^*)^2/\Dt$.
	
	When $\omega_n^*/\Dt \gtrsim 1$,
	we proceed to prove the lower bound $\omega_n^*$.   Notice that $\omega_n^*\gtrsim \Dt$ implies $\sqrt{K/n}\gtrsim \Dt$, which, in view of  (\ref{lb_1}) and by $-\Dt_x  \le -\Dt/2 = o(1)$, further implies 
	\[
	\inf_{\wh g}\sup_{\theta \in \Theta} \PP^{\D}_{\theta}\left\{
	L_\theta(\wh g) \ge C_L
	\right\}\ge c_0
	\]
	for some $C_L\in (0,1]$. By choosing $t_1$ as (\ref{def_t1}) in (\ref{lb_key}), we have $t_1 \asymp \sqrt{K/n}$, $t_1/\Dt \gtrsim 1$ and 
	\[
	\max\{\varphi(R_{t_1}), \varphi(L_{t_1})\}~  \lesssim ~ 
	\exp\left(
	-{c_t t_1^2 \over \Dt^2}
	\right),
	\]
	hence
	\begin{align*}
		\inf_{\wh g}\sup_{\theta \in \Theta} \PP^{\D}_{\theta}&\left\{
		R_x(\wh g) - R_x(g_\theta^*) \ge 
		{C_L t_1\over 2} - {t_1^2\over 2\Dt}\exp\left(
		-{c_t t_1^2 \over \Dt^2}
		\right)
		\right\}\ge c_0.
	\end{align*}
	By choosing $c_t$ to be sufficiently large and $t_1 /\Dt \gtrsim 1$, we have 
	\[
	{t_1 \over \Dt}\exp\left(
	-{c_t t_1^2 \over \Dt^2}
	\right)  \le  {C_L \over 2},
	\]
	such that 
	\begin{align*}
		\inf_{\wh g}\sup_{\theta \in \Theta} \PP^{\D}_{\theta}&\left\{
		R_x(\wh g) - R_x(g_\theta^*) \ge 
		{C_L t_1\over 4}
		\right\}\ge c_0.
	\end{align*}
	The claim then follows from
	\[
	{t_1\over 4} + {\sigma^2 \over \lambda}\Dt^2 \asymp \sqrt{{K\over n} + {\sigma^2 p \over \lambda^2 n}\Dt^2}  +  {\sigma^2 \over \lambda}\Dt^2\asymp \sqrt{K\over n}  \asymp  \omega_n^*
	\]
	by using $\Dt \lesssim 1$, $\sqrt{K/n}\gtrsim \Dt$, $\sigma^2 \lesssim \lambda$ and $p\sigma^2 \lesssim n\lambda$. 
\end{proof}

\subsection{Proof of (\ref{lb_2})}
\begin{proof}
	We aim to invoke the following lemma to obtain the desired lower bound. The lemma below follows immediately from the proof of Proposition 1 in \cite{Martin_2013} together with Theorem 2.5 in \cite{Intro_non_para}.
	\begin{lemma}\label{lem_tsybakov}
		Let $M\ge 2$ and $\theta_0, \ldots, \theta_M \in \Theta$. For some constant $c_0 \in (0, 1/8]$, $\gamma >0$ and any classifier $\wh g$, if $\KL(\PP^{\D}_{\theta_i}, \PP^{\D}_{\theta_0}) \le c_0 \log M$ for all $1\le i\le M$, and $L_{\theta_i} (\wh g)< \gamma$ implies $L_{\theta_j} (\wh g)\ge \gamma$ for all $0\le i\ne j\le M$, then 
		\[
		\inf_{\wh g}\sup_{i\in \{1,\ldots, M\}} \PP^{\D}_{\theta_i}\{L_{\theta_i}(\wh g) \ge \gamma\} \ge
		{\sqrt M\over \sqrt M + 1}\left[
		1 - 2c_0 - \sqrt{2c_0 \over \log M}\right].
		\]
	\end{lemma}
	
	To this end, we start by describing our construction of hypotheses 
	of $\theta \in \Theta(\lambda, \sigma, \Dt)$ defined in (\ref{def_Theta}). 
	Without loss of generality, we assume $\sigma = 1$ and $\szy = \bI_K$. We consider a subspace of $\Theta(\lambda, \sigma, \Dt)$ where  $\lambda_1(A\szy A^\T) = \lambda_K(A\szy A^\T) = \lambda$. By further writing $A\szy A^\T = AA^\T = \lambda BB^\T$ with $B \in \cO_{p\times K}$, we consider
	\begin{equation}\label{def_thetas}
		\theta^{(j)} = \left(\sqrt{\lambda}~ B^{(j)}, \bI_K, \bI_p, \a, -\a, {1\over 2}, {1\over 2}\right), \quad \text{for } j = 1,\ldots, M,
	\end{equation}
	where 
	\begin{equation}\label{def_alpha_B_j}
		\a = \begin{bmatrix}
			\Dt/2 \vspace{1mm}\\ \0_{K-1}
		\end{bmatrix}, \qquad B^{(j)} = \begin{bmatrix}
			\sqrt{1-\eps^2} & 0 \vspace{1mm}\\
			\0_{K-1} & \bI_{K-1} \vspace{1mm}\\
			\eps J^{(j)} &  \0_{p-K}
		\end{bmatrix} := \begin{bmatrix}
			B_1^{(j)} & B_{-1}
		\end{bmatrix},
	\end{equation}
	with 
	\begin{equation}\label{def_epsilon}
		\eps^2 = c_0c_1 {(p-K) \over \lambda ~ n}{1\over {2\lambda\over 1+\lambda} + \Dt^2}
	\end{equation}
	for some constants $c_0\in (0,1/8]$ and $c_1>0$. 
	Here $J^{(1)},\ldots, J^{(M)}\in \cO_{(p-K)\times 1}$ are chosen according to the hypercube construction in Lemma \ref{lem_J} with $m = p-K$. 
	It is easy to see that $\theta^{(j)} \in \Theta(\lambda,\sigma=1, \Dt)$ for all $1\le j\le M$. 		Lemma \ref{lem_theta} below  collects several useful properties of $\theta^{(j)}$.

	Next, to apply Lemma \ref{lem_tsybakov}, it suffices to verify 
	\begin{enumerate}
		\item[(1)] $\KL(\PP^{\D}_{\theta^{(1)}}, \PP^{\D}_{\theta^{(i)}})\le c_0\log (M-1)$ for all $1\le i\le M$;
		\item [(2)] $L_{\theta^{(i)}}(\wh g) + L_{\theta^{(j)}}(\wh g) \ge 2\gamma$, for all $1\le i\ne j\le M$ and any measurable $\wh g$, with 
		\[
		\gamma \asymp e^{- {\Dt_x^2/ 8}}\sqrt{\eps^2 \over \lambda},\qquad \Dt_x^2 = {\lambda \over 1+\lambda}\Dt^2.
		\]
	\end{enumerate}
	The first claim is proved by invoking Lemmas \ref{lem_J} and \ref{lem_KL} together with the choice of $\eps$ in (\ref{def_epsilon}) while the second claim is proved in Lemma \ref{lem_lb_L}.  The result then follows by noting that 
	$$
	\eps^2 \asymp {p-K\over n\lambda (1+\Dt^2)} \asymp {p-K \over n\lambda \Dt^2}.
	$$
\end{proof}

\subsubsection{Lemmas used in the proof of (\ref{lb_2})}

The following lemma is  adapted from \cite[Lemma A.5]{vu2013minimax}. 
\begin{lemma}[Hypercube construction]\label{lem_J}
	Let $m \ge 1$ be an integer. There exist $J^{(1)},\ldots, J^{(M)} \in \cO_{m\times 1}$ with the following properties:
	\begin{enumerate}
		\item  $\|J^{(i)} -J^{(j)}\|_2^2 \ge 1/4$ for all $i\ne j$, and
		\item $\log M \ge \max\{c m, \log m\}$, where $c > 1/30$ is an absolute constant.
	\end{enumerate}
\end{lemma}
\begin{proof}
	The case for $m\ge e$ is proved in \cite[Lemma A.5]{vu2013minimax} by taking $m = s$. For $m = 2$, one can choose  
	$J^{(i)} = (-1)^i \be_1$, for $i = 1,2$, and $J^{(i)} = (-1)^i \be_2$, for $i = 3,4$, such that $M = 4$ and $\|J^{(i)} - J^{(j)}\|_2^2 = 4$. Here $\{\be_1,\be_2\}$ represents the set of canonical vectors in $\RR^2$. For $m=1$, one can simply take $J^{(i)} = (-1)^{i}$ for $i = 1,2$. 
\end{proof}

\medspace

The following lemma collects some useful identities, under the choices of $\theta^{(j)}$ in (\ref{def_thetas}) -- (\ref{def_alpha_B_j}).

\begin{lemma}\label{lem_theta}
	Fix any $i \in \{1,\ldots,M\}$. Let $B^{(i)}$ and $\a$ defined in (\ref{def_alpha_B_j}). Further let 
	$$
	\Sigma^{(i)} = \lambda  B^{(i)}(B^{(i)})^\T + \bI_p,\qquad \mu^{(i)} = \sqrt{\lambda} B^{(i)} \a.
	$$
	\begin{enumerate}
		\item[(i)]  $|\Sigma^{(i)}| = (\lambda+1)^K$ and 
		\begin{align}\label{eq_Sigma_inv_1}
			( \Sigma^{(i)})^{-1}  &= {1\over \lambda+1} B^{(i)}(B^{(i)})^\T + \bI_p - B^{(i)}(B^{(i)})^\T\\\label{eq_Sigma_inv_2}
			&= \bI_p - {\lambda \over \lambda+1}B^{(i)}(B^{(i)})^\T.
		\end{align} 
		
		\item[(ii)] 
		$$
		(\Sigma^{(i)})^{-1}\mu^{(i)} = {\sqrt{\lambda}\over 1+\lambda}B^{(i)}\alpha = {\sqrt{\lambda}\over 1+\lambda} {\Dt\over 2}B_1^{(i)}.
		$$
		
		\item[(iii)]
		\begin{equation*}
			(\mu^{(i)})^\T(\Sigma^{(i)})^{-1}\mu^{(i)} = {\lambda  \over 1+\lambda}\alpha^\T (B^{(i)})^\T B^{(i)} \alpha  ={\lambda \over 1+\lambda} {\Dt^2\over 4}
		\end{equation*}
	\end{enumerate}
\end{lemma}
\begin{proof}
	Notice that $B^{(i)} \in \cO_{p\times K}$. Then part (i) is easy to verify. Parts (ii) and (iii) follow immediately from (\ref{def_alpha_B_j}) and (\ref{eq_Sigma_inv_1}).
\end{proof}

\medspace

Let $\PP^{\D}_{\theta^{(i)}}$, for $2\le i\le M$, denote the distribution of $(\X,\Y)$ parametrized by $\theta^{(i)}$. The following lemma provides upper bounds of the KL-divergence between $\PP_{\theta^{(1)}}$ and $\PP_{\theta^{(i)}}$.

\begin{lemma}[KL-divergence]\label{lem_KL}
	For any $\theta^{(i)}$, let 
	\[
	(X \mid Y = 1) \sim N_p(\mu^{(i)}, \Sigma^{(i)}),\qquad (X\mid Y = 0) \sim N_p(-\mu^{(i)}, \Sigma^{(i)})
	\]
	with $\mu^{(i)} = \sqrt{\lambda}B^{(i)}\a$, $\Sigma^{(i)} = \lambda B^{(i)}(B^{(i)})^\T + \bI_p$ and $B^{(i)} \in \cO_{p\times K}$. Then 
	\[
	\KL(\PP^{\D}_{\theta^{(1)}}, \PP^{\D}_{\theta^{(i)}}) \le n \left({2\lambda \over 1+\lambda} + {\Dt^2\over 2}\right) \lambda ~ \eps^2
	\]
\end{lemma}
\begin{proof}
	Since $(\X, \Y)$ contains $n$ i.i.d. copies of $(X, Y)$, it suffices to prove 
	\[
	\KL(\PP_{\theta^{(1)}}, \PP_{\theta^{(i)}}) =  \KL\left(N_p(\mu^{(1)}, \Sigma^{(1)}),N_p(\mu^{(i)}, \Sigma^{(i)}) \right) \le  \left({2\lambda \over 1+\lambda} + {\Dt^2\over 2}\right) \lambda ~ \eps^2.
	\]
	By the formula of KL-divergence between two multivariate normal distributions, 
	\begin{align*}
		\KL(\PP_{\theta^{(1)}}, \PP_{\theta^{(i)}}) \le   & ~  {1\over 2}\left\{
		\tr\left[(\Sigma^{(i)})^{-1}\left(\Sigma^{(1)} -\Sigma^{(i)}\right) \right] + \log{|\Sigma^{(i)}| \over |\Sigma^{(1)}|}
		\right\}\\
		& + {1\over 2}\left(\mu^{(i)} - \mu^{(1)}\right)^\T ( \Sigma^{(i)})^{-1} \left(\mu^{(i)} - \mu^{(1)}\right)\\
		:= & ~  I_1 + I_2.
	\end{align*}
	From \cite[Lemmas A.2 \& A.3]{vu2013minimax}, 
	\[
	I_1 = {\lambda^2 \over 1+\lambda}\cdot {1\over 2}\left\|B^{(i)}(B^{(i)})^\T-B^{(1)}(B^{(1)})^\T\right\|_F^2 \le {\lambda^2\over 1+\lambda}{\eps^2\over 2}\left\|J^{(i)} - J^{(1)}\right\|_2^2.
	\]
	For $I_2$, by using part (i) of Lemma \ref{lem_theta}
	together with 
	\[
	\mu^{(i)} - \mu^{(1)} = \sqrt{\lambda} (B^{(i)} - B^{(1)})\alpha = {\Dt\sqrt{\lambda}\over 2} \eps(J^{(i)} - J^{(1)}),
	\]
	from (\ref{def_alpha_B_j}), we find
	\begin{align*}
		I_2 &=  {\lambda \Dt^2 \over 8}\eps^2 (J^{(i)} - J^{(1)})^\T \left(
		\bI_p - {\lambda \over \lambda+1}B^{(i)}(B^{(i)})^\T\right) (J^{(i)} - J^{(1)})\\
		&\le  {\lambda \Dt^2 \over 8}\eps^2 \left\|J^{(i)} - J^{(1)}\right\|_2^2.
	\end{align*}
	Combining the bounds of $I_1$ and $I_2$ and using $\|J^{(i)} - J^{(1)}\|_2^2 \le 4$ complete the proof. 
\end{proof}

\medspace

Recall that $L_{\theta}(\cdot)$ is defined in (\ref{def_L}). 
The following lemma establishes lower bounds of $L_{\theta^{(i)}}(\wh g) + L_{\theta^{(j)}}(\wh g)$ for any measurable $\wh g$.

\begin{lemma}\label{lem_lb_L}
	Let $\theta^{(i)}$ for $1\le i\le M$ be constructed as (\ref{def_thetas}) -- (\ref{def_alpha_B_j}). Under conditions of Theorem \ref{thm_lowerbound}, for any measurable $\wh g$, one has  
	\[
	L_{\theta^{(i)}}(\wh g) + L_{\theta^{(j)}}(\wh g) ~ \gtrsim ~ e^{-{\Dt_x^2/ 8}}\sqrt{\eps^2 \over \lambda}
	\]
	with $\Dt_x^2 = \lambda \Dt^2/(1+\lambda)$.
\end{lemma}
\begin{proof}
	Pick any $i\ne j\in \{1,\ldots, M\}$ and any $\wh g$. For simplicity, we write $\theta = \theta^{(i)}$ and $\theta' = \theta^{(j)}$ with corresponding $B = B^{(i)}$ and $B' = B^{(j)}$.  We also write $L_\theta = L_\theta(\wh g)$ and $L_{\theta'} = L_{\theta'}(\wh g)$. The proof consists of three steps:
	\begin{enumerate}
		\item[(a)] Bound $L_{\theta} + L_{\theta'}$ from below by a $p$-dimensional integral;
		\item[(b)] Reduce the $p$-dimensional integral to a $2$-dimensional integral; 
		\item[(c)] Bound from below the $2$-dimensional integral.
	\end{enumerate}
	
	\paragraph{Step (a)}
	
	By definition in (\ref{def_L}), 
	\begin{align*}
		L_{\theta} + L_{\theta'} &=  \int_{\wh g \ne g^*_\theta}d \PP_{\theta}(x) + \int_{\wh g \ne g^*_{\theta'}}  d\PP_{\theta'}(x)\\
		&\ge \int_{\{\wh g \ne g^*_\theta\} \cup \{\wh g \ne g^*_{\theta'}\}}  \min\left\{d\PP_{\theta}(x), ~  d\PP_{\theta'}(x)\right\}\\
		&\ge  \int_{g^*_{\theta} \ne g^*_{\theta'}} \min\left\{d\PP_{\theta}(x), ~  d\PP_{\theta'}(x)\right\}.
	\end{align*}
	In the last step we  used 
	\begin{align*}
		\{g^*_{\theta} \ne g^*_{\theta'}\} &= \{\wh g=  g^*_{\theta}, \wh g \ne  g^*_{\theta'}\} \cup \{
		\wh g \ne  g^*_{\theta}, \wh g =  g^*_{\theta'}
		\}\\
		& \subseteq \{\wh g \ne g^*_\theta\} \cup \{\wh g \ne g^*_{\theta'}\}.
	\end{align*}
	Since 
	\[
	\PP_{\theta}  = {1\over 2}N_p(\mu_\theta, \Sigma_\theta)+ {1\over 2}N_p(-\mu_\theta, \Sigma_\theta)
	\]
	and 
	$
	g_\theta^*(x) = \1\{
	x^\T \Sigt^{-1} \mut \ge 0
	\}
	$
	from (\ref{def_gx_star}),  we obtain 
	\begin{align}\label{lb_L_theta_theta_prime}\nonumber
		& 	L_{\theta} + L_{\theta'} \\\nonumber
		& \ge {1\over 2}\int_{\substack{x^\T \Sigt^{-1}\mut \ge 0\\
				x^\T \Sigtp^{-1}\mutp <0}}{1\over (2\pi)^{p/2}}
		\min\left\{
		|\Sigt|^{-1/2} \left[\exp\left(-{1\over 2}\|x-\mut\|^2_{\Sigt}\right)+\exp\left(-{1\over 2}\|x+\mut\|^2_{\Sigt}\right)\right],\right.\\\nonumber
		&\hspace{5cm}\left. |\Sigtp|^{-1/2} \left[\exp\left(-{1\over 2}\|x-\mutp\|^2_{\Sigtp}\right)+\exp\left(-{1\over 2}\|x+\mutp\|^2_{\Sigtp}\right)\right]
		\right\}{\rm d} x\\\nonumber
		&~ + {1\over 2}\int_{\substack{x^\T \Sigt^{-1}\mut < 0\\
				x^\T \Sigtp^{-1}\mutp \ge 0}}{1\over (2\pi)^{p/2}}
		\min\left\{
		|\Sigt|^{-1/2} \left[\exp\left(-{1\over 2}\|x-\mut\|^2_{\Sigt}\right)+\exp\left(-{1\over 2}\|x+\mut\|^2_{\Sigt}\right)\right],\right.\\\nonumber
		&\hspace{5cm}\left. |\Sigtp|^{-1/2} \left[\exp\left(-{1\over 2}\|x-\mutp\|^2_{\Sigtp}\right)+\exp\left(-{1\over 2}\|x+\mutp\|^2_{\Sigtp}\right)\right]
		\right\}{\rm d} x\\\nonumber
		&= \int_{\substack{x^\T \Sigt^{-1}\mut \ge 0\\
				x^\T \Sigtp^{-1}\mutp <0}}{	|\Sigt|^{-1/2}\over (2\pi)^{p/2}}
		\min\left\{
		\exp\left(-{1\over 2}\|x-\mut\|^2_{\Sigt}\right)+\exp\left(-{1\over 2}\|x+\mut\|^2_{\Sigt}\right),\right.\\\nonumber
		&\hspace{5cm}\left.  \exp\left(-{1\over 2}\|x-\mutp\|^2_{\Sigtp}\right)+\exp\left(-{1\over 2}\|x+\mutp\|^2_{\Sigtp}\right)
		\right\}{\rm d} x\\\nonumber
		&\ge \int_{\substack{x^\T \Sigt^{-1}\mut \ge 0\\
				x^\T \Sigtp^{-1}\mutp <0}}{	|\Sigt|^{-1/2}\over (2\pi)^{p/2}}
		\min\left\{
		\exp\left(-{1\over 2}\|x-\mut\|^2_{\Sigt}\right), ~ \exp\left(-{1\over 2}\|x+\mutp\|^2_{\Sigtp}\right)
		\right\}{\rm d} x\\
		&\ge  e^{- {\Dt_x^2\over 8}} \int_{\substack{x^\T \Sigt^{-1}\mut \ge 0\\
				x^\T \Sigtp^{-1}\mutp <0}}{	|\Sigt|^{-1/2}\over (2\pi)^{p/2}}
		\min\left\{
		\exp\left(-{1\over 2}x^\T \Sigt^{-1} x\right), ~ \exp\left(-{1\over 2}x^\T \Sigtp^{-1} x\right)
		\right\}{\rm d} x .
	\end{align}
	The equality uses the fact that $X$ has the same distribution as $-X$ and the identity
	\begin{equation}\label{eq_deter_Sigma}
		|\Sigt| = |\Sigtp| = (\lambda+1)^K
	\end{equation}
	from part (i) of Lemma \ref{lem_theta}. The last step uses 
	the fact that 
	\[
	{\Dt_x^2 \over 4} \overset{(\ref{def_Dt_x})}{=} \mut^\T \Sigt^{-1}\mut  ={\lambda \over 1+\lambda} {\Dt^2\over 4} = \mutp^\T \Sigtp^{-1}\mutp
	\]
	from part (iii) of Lemma \ref{lem_theta}.\\ 
	
	\paragraph{Step (b)}
	In the following, we provide a lower bound for  
	\[
	T := \int_{\substack{x^\T \Sigt^{-1}\mut \ge 0\\
			x^\T \Sigtp^{-1}\mutp <0}}{	|\Sigt|^{-1/2}\over (2\pi)^{p/2}}
	\min\left\{
	\exp\left(-{1\over 2}x^\T \Sigt^{-1} x\right), ~ \exp\left(-{1\over 2}x^\T \Sigtp^{-1} x\right)
	\right\}{\rm d} x.
	\]
	Recall from (\ref{def_alpha_B_j}) and (\ref{eq_Sigma_inv_2}) that 
	\begin{align*}
		\Sigt^{-1} &= \bI_p - {\lambda \over 1+\lambda }B_{-1}B_{-1}^\T  - {\lambda \over 1+\lambda }B_1B_1^\T,\\
		\Sigtp^{-1} &= \bI_p - {\lambda \over 1+\lambda }B_{-1}B_{-1}^\T  - {\lambda \over 1+\lambda }B'_1{B'_1}^\T.
	\end{align*} 
	Further note from part (ii) of Lemma \ref{lem_theta} that
	\begin{align*}
		\Sigt^{-1}\mut = {\sqrt{\lambda}\over 1+\lambda} {\Dt\over 2}B_1,\qquad 	\Sigtp^{-1}\mutp = {\sqrt \lambda\over 1+\lambda} {\Dt\over 2}B_1'.
	\end{align*}
	Plugging these expressions in $T$ yields 
	\begin{align*}
		T &=  \int_{\substack{x^\T B_1 \ge 0\\
				x^\T B_1' <0}}{	|\Sigt|^{-1/2}\over (2\pi)^{p/2}}
		\exp\left(
		-{1\over 2}x^\T \left(\bI_p - {\lambda\over \lambda+1}B_{-1}B_{-1}^\T\right) x\right)\\
		&\hspace{1cm}\min\left\{
		\exp\left({1\over 2}x^\T{\lambda \over 1+\lambda}B_1B_1^\T x\right), ~ \exp\left({1\over 2}x^\T  {\lambda \over 1+\lambda}B'_1{B'_1}^\T x\right)
		\right\}{\rm d} x.
	\end{align*}
	Let $H\in\cO_{p\times p}$ such that 
	\begin{equation}\label{def_H}
		H B_1 = \begin{bmatrix}
			a \\ b \\ \0_{p-2}
		\end{bmatrix} := \begin{bmatrix}
			u \\ \0_{p-2}
		\end{bmatrix},\qquad HB_1' = \begin{bmatrix}
			a \\ -b \\ \0_{p-2}
		\end{bmatrix} := \begin{bmatrix}
			v \\ \0_{p-2}
		\end{bmatrix},\qquad a > 0.
	\end{equation}		
	Such an $H$ exists since $[B_1~  B_1'] \in \RR^{p\times 2}$ has rank $2$ and $\|B_1\|_2 = \|B_1'\|_2 = 1$. 
	By changing variables $y = Hx$ and by writing $y_I^\T = (y_1, y_2)$, we obtain
	\begin{align*}
		T &=  \int_{\substack{y_I^\T u\ge 0\\
				y_I^\T v<0}}{|\Sigt|^{-1/2}\over (2\pi)^{p/2}}
		\exp\left(
		-{1\over 2}y^\T H\left(\bI_p - {\lambda\over \lambda+1}B_{-1}B_{-1}^\T\right)H^\T y\right)\\
		&\hspace{1cm}\min\left\{
		\exp\left({\lambda(y_I^\T u)^2 \over 2(1+\lambda)}\right), ~ \exp\left( {\lambda(y_I^\T v)^2 \over 2(1+\lambda)}\right)
		\right\} {\rm d}y.
	\end{align*}
	Denote 
	\begin{equation}\label{def_Q}
		Q := H\left(\bI_p - {\lambda\over \lambda+1}B_{-1}B_{-1}^\T\right)^{-1}H^\T = H(\lambda B_{-1}B_{-1}^\T + \bI_p)H^\T. 
	\end{equation}
	Notice that $|Q| = (\lambda+1)^{K-1} = |\Sigt| / (\lambda+1)$ by (\ref{eq_deter_Sigma}). We further have 
	\begin{align*}
		T &=  {1\over \sqrt{\lambda+1}}\int_{\substack{y_I^\T u\ge 0\\
				y_I^\T v<0}}{|Q|^{-1/2}\over (2\pi)^{p/2}}
		\exp\left(
		-{1\over 2}y^\T Q^{-1} y\right)\\
		&\hspace{1cm}\min\left\{
		\exp\left({\lambda(y_I^\T u)^2 \over 2(1+\lambda)}\right), ~ \exp\left({\lambda(y_I^\T v)^2 \over 2(1+\lambda)}\right)
		\right\} {\rm d}y\\
		&= {1\over \sqrt{\lambda+1}}\int_{\substack{ay_1 + by_2\ge 0\\
				ay_1 - by_2<0}}{	|Q_{II}|^{-1/2}\over 2\pi}
		\exp\left(
		-{1\over 2}y_I^\T (Q_{II})^{-1} y_I\right)\\
		&\hspace{1cm}\min\left\{
		\exp\left({\lambda(ay_1 + by_2)^2 \over 2(1+\lambda)}\right), ~ \exp\left( {\lambda(ay_1 - by_2)^2 \over 2(1+\lambda)}\right)
		\right\}{\rm d}y_I
	\end{align*}
	where $Q_{II}$ is the first $2\times 2$ submatrix of $Q$. Recall that $a>0$ and on the area of integration
	$\{ay_1 + by_2\ge 0,	ay_1 - by_2<0\}$ we have 
	\[
	\exp\left({\lambda(ay_1 + by_2)^2 \over 2(1+\lambda)}\right) \ge \exp\left({\lambda(ay_1 - by_2)^2 \over 2(1+\lambda)}\right) \quad \iff  \quad y_1 \ge 0. 
	\]
	Splitting $T$ into two parts further gives 
	\begin{align*}
		T &= {1\over \sqrt{\lambda+1}}\int_{\substack{ay_1 + by_2\ge 0\\
				ay_1 - by_2<0\\ y_1 \ge 0}}{	|Q_{II}|^{-1/2}\over 2\pi}
		\exp\left[
		-{1\over 2}y_I^\T \left(Q_{II}^{-1} - {\lambda\over 1+\lambda}vv^\T\right) y_I\right] {\rm d} y_I\\
		& ~ + {1\over \sqrt{\lambda+1}}\int_{\substack{ay_1 + by_2\ge 0\\
				ay_1 - by_2<0\\ y_1 < 0}}{	|Q_{II}|^{-1/2}\over 2\pi}
		\exp\left[
		-{1\over 2}y_I^\T \left(Q_{II}^{-1} - {\lambda\over 1+\lambda}uu^\T\right) y_I\right] {\rm d}y_I\\
		& := T_1 + T_2.
	\end{align*}
	
	\smallskip 
	
	\paragraph{Step (c)}
	We bound from below $T_1$ first. Denote 
	\begin{align}\label{def_G}
		G = \left(Q_{II}^{-1} - {\lambda\over 1+\lambda}vv^\T\right) ^{-1}
		&= Q_{II} +{ {\lambda \over 1+\lambda}Q_{II}vv^\T Q_{II} \over 1 - {\lambda\over 1+\lambda}v^\T Q_{II}v} 
		= Q_{II} + \lambda Q_{II}vv^\T Q_{II}
	\end{align}
	where the second equality uses the Sherman-Morrison formula and the third equality is due to the fact that 
	\begin{align}\label{eq_vQIIv}\nonumber
		v^\T Q_{II} v &= {B'_1}^\T H^\T H (\lambda B_{-1}B_{-1}^\T + \bI_p) H^\T H B'_1  &&\textrm{by (\ref{def_H}) and (\ref{def_Q})}\\\nonumber
		&=  \lambda{B'_1}^\T B_{-1}B_{-1}^\T B_1' + 1 && \textrm{by }H\in \cO_{p\times p}\\
		& = 1 &&\textrm{by (\ref{def_alpha_B_j}).}
	\end{align}
	Further observe that 
	\[
	|G| = |Q_{II}| \left|
	\bI_2 + \lambda Q_{II}^{1/2}vv^\T Q_{II}
	\right| =  |Q_{II}|  ( 1 + \lambda v^\T Q_{II}v) =  |Q_{II}| (1 +\lambda).
	\]
	We obtain
	\begin{align*}
		T_1 &=  \int_{\substack{ay_1 + by_2\ge 0\\
				ay_1 - by_2<0\\ y_1 \ge 0}}{	|G|^{-1/2}\over 2\pi}
		\exp\left[
		-{1\over 2}y_I^\T G^{-1} y_I\right]{\rm d}y_I \\
		& = \int_{\substack{ay_1 - by_2<0\\ ay_1 \ge 0}}{	|G|^{-1/2}\over 2\pi}
		\exp\left[
		-{1\over 2}y_I^\T G^{-1} y_I\right]{\rm d}y_I.
	\end{align*}
	By changing of variables $z = G^{-1/2}y_I$ again and writing
	\[
	\zeta_1 = G^{1/2}v,\qquad \zeta_2 = G^{1/2}\begin{bmatrix}
		a \\ 0
	\end{bmatrix}
	\]
	for simplicity, 
	one has 
	\begin{align*}
		T_1 = \int_{\substack{z^\T \zeta_1<0\\ z^\T \zeta_2 \ge 0}}{1\over 2\pi}
		e^{-{1\over 2}z^\T z}{\rm d}z & =
		{1\over \pi}\int_{\substack{\zeta_{11}\cos\theta + \zeta_{12}\sin\theta<0\\  \zeta_{21}\cos\theta + \zeta_{22}\sin\theta \ge 0}} {\rm d}\theta.
	\end{align*}
	Note that, the integral is simply the area within the half unit circle $\{(x,y): x^2 + y^2 \le 1, y \ge 0\}$ intersected by vectors $\zeta_1$ and $\zeta_2$. We thus conclude 
	\[
	T_1  = {1\over 2\pi}\textrm{arc}(\wt\zeta_1, \wt\zeta_2) \ge  {1\over 2\pi}\left\|\wt \zeta_1 - \wt \zeta_2\right\|_2
	\]
	where $\wt\zeta_1 = \zeta_1 / \|\zeta_1\|_2$, $\wt \zeta_2 = \zeta_2 / \|\zeta_2\|_2$ and $\textrm{arc}(\wt\zeta_1, \wt\zeta_2)$ denotes the length of the arc between $\wt\zeta_1$ and $\wt\zeta_2$.
	
	We proceed to calculate $\|\wt \zeta_1 - \wt \zeta_2\|_2$. First note that 
	\begin{align*}
		\|\zeta_1\|_2^2 = v^\T G v \overset{ (\ref{def_G}) }{=} v^\T \left(Q_{II} + \lambda Q_{II}vv^\T Q_{II}\right)v \overset{(\ref{eq_vQIIv})}{=} 1 + \lambda. 
	\end{align*}
	Since
	\[
	Q_{II}v \overset{(\ref{def_Q})}{=} H_I(\lambda B_{-1}B_{-1}^\T + \bI_p)H_I^\T v \overset{(\ref{def_H})}{=} H_I(\lambda B_{-1}B_{-1}^\T + \bI_p)H^T HB_1' = H_IB_1',
	\]
	we obtain
	\begin{align*}
		\|\zeta_2\|_2^2 &=  {1\over 4}(u+v)^\T G(u+v)\\
		& = {1\over 4}(B_1+B_1')^\T H_I^\T \left(Q_{II} + \lambda Q_{II}vv^\T Q_{II}\right)  H_I(B_1+B_1')\\
		&=  {1\over 4}(B_1+B_1')^\T H_I^\T \left[H_I(\lambda B_{-1}B_{-1}^\T + \bI_p)H_I^\T+ \lambda H_I B_1'{B_1'}^\T H_I^\T\right] H_I(B_1+B_1')\\
		& = {1\over 4}(B_1+B_1')^\T \left(\bI_p+ \lambda  B_1'{B_1'}^\T \right) (B_1+B_1')\\
		& = {1 \over 4}\left[
		\lambda+2 + 2(\lambda+1)B_1^\T B_1' + \lambda (B_1^\T B_1')^2
		\right].
	\end{align*}
	The penultimate step uses the orthogonality between $B_{-1}$ and $B_1 + B_1'$.
	Since
	\[
	1 - B_1^\T B_1' = {1\over 2}\|B_1 - B_1'\|_2^2 = {\eps^2\over 2}\|J^{(i)} - J^{(j)}\|_2^2 \le 2 \eps^2
	\]
	which can be bounded by a sufficiently small constant, 
	we have $B_1^\T B_1' \asymp 1$ hence $\|\zeta_2\|_2^2 \asymp \lambda + 1$. Finally, similar arguments yield 
	\begin{align*}
		\zeta_1^\T \zeta_2  &= {1\over 2}v^\T G (u+v)\\
		& = {1\over 2}(B_1')^\T \left(\bI_p + \lambda B_1'{B_1'}^\T \right)(B_1+B_1')\\
		& = {1\over 2}(1+\lambda)(1+B_1^\T B_1')\\
		&\asymp 1+\lambda.
	\end{align*}
	We thus have, after a bit algebra, 
	\[
	\|\zeta_1\|_2^2\|\zeta_2\|_2^2 - (\zeta_1^\T \zeta_2)^2 = {1\over 4}(1+\lambda)(1+B_1^\T B_1')(1-B_1^\T B_1') \asymp (1+\lambda) \eps^2,
	\]
	hence
	\begin{align*}
		{1\over 2} \left\|\wt \zeta_1 - \wt \zeta_2\right\|_2^2 &=   {\|\zeta_1\|_2\|\zeta_2\|_2-\zeta_1^\T \zeta_2 \over \|\zeta_1\|_2\|\zeta_2\|_2}\\
		&= {\|\zeta_1\|_2^2\|\zeta_2\|_2^2-(\zeta_1^\T \zeta_2)^2 \over \|\zeta_1\|_2\|\zeta_2\|_2+\zeta_1^\T \zeta_2}{1\over \|\zeta_1\|_2\|\zeta_2\|_2}\\
		&\asymp {\eps^2 \over 1+\lambda}
	\end{align*}
	implying that
	\[
	T_1 \gtrsim \sqrt{\eps^2 \over \lambda}.
	\]
	
	Following the same line of reasoning, we can derive the same lower bound for $T_2$.  We conclude that
	\[
	L_{\theta} + L_{\theta'}\gtrsim e^{-{\Dt_x^2/ 8}}\sqrt{\eps^2 \over \lambda},
	\]
	which completes the proof. 
\end{proof}

\subsection{Proof of (\ref{lb_1})}

The proof of (\ref{lb_1}) follows the same lines of reasoning as the proof of (\ref{lb_2}).  To construct hypotheses of $\Theta(\lambda, \sigma = 1, \Dt)$, we consider 
\begin{equation}\label{def_thetas_prime}
	\theta^{(j)} = \left(\sqrt{\lambda}~ B, \bI_K, \bI_p, \a^{(j)}, \a^{(j)}, {1\over 2}, {1\over 2}\right), \quad \text{for } j = 1,\ldots, M',
\end{equation}
with $B\in \cO_{p\times K}$ and  
\begin{equation}\label{def_alphas}
	\a^{(j)} =  {\Dt\over 2}\begin{bmatrix}
		\sqrt{1-(\eps')^2} \vspace{2mm}\\ 
		\eps' J^{(j)}
	\end{bmatrix}.
\end{equation}
Here $J^{(j)}$ for $j = 1,\ldots, M'$ are again chosen according to Lemma \ref{lem_J} with $m = K-1$ and 
\begin{equation}\label{def_epsilon_prime}
	(\eps')^2 =  {c_0c_1 (K-1) \over n\Dt^2}.
\end{equation}
for some constant $c_0\in (0,1/8]$ and $c_1>0$.  Notice that $\|\a^{(j)}\|_2^2  = {\Dt^2 / 4}$
for all $j\in \{0, 1,\ldots, M'\}$, so that  $\theta^{(j)}\in \Theta(\lambda, \sigma =1, \Dt)$. From part (iii) of Lemma \ref{lem_theta}, we also have 
\[
{\Dt_x^2 \over 4} \overset{(\ref{def_Dt_x})}{=}\mu_{\theta^{(j)}}^\T \Sigma_{\theta^{(j)}}^{-1}\mu_{\theta^{(j)}} = {\lambda \over 1+\lambda} \|\a^{(j)}\|_2^2 = {\lambda \over 1+\lambda}{\Dt^2 \over 4},\qquad \forall j\in \{1,\ldots, M'\}.
\]

Next, to invoke Lemma \ref{lem_tsybakov}, it remains to verify 
\begin{enumerate}
	\item[(1)] $\KL(\PP^{(\D)}_{\theta^{(0)}}, \PP^{(\D)}_{\theta^{(i)}})\le c_0 \log M'$ for all $1\le i\le M'$;
	\item [(2)] $L_{\theta^{(i)}}(\wh g) + L_{\theta^{(j)}}(\wh g) \ge 2\gamma$, for all $1\le i\ne j\le M'$ and any $\wh g$, with 
	\[
	\gamma \asymp{1\over \Dt_x} e^{- {\Dt_x^2/ 8}}\sqrt{K\over n},\qquad \Dt_x^2 = {\lambda \over 1+\lambda}\Dt^2.
	\]
\end{enumerate}
To prove (1), note that the distribution of $(Y,X)$ parametrized by $\theta^{(i)}$ is 
\[
\PP_{\theta^{(i)}} = {1\over 2}N_p(\mu_{\theta^{(i)}}, \Sigma_{\theta^{(i)}}) + {1\over 2}N_p(-\mu_{\theta^{(i)}}, \Sigma_{\theta^{(i)}})
\]
with 
$
\mu_{\theta^{(i)}} = \sqrt{\lambda} B \a^{(i)}
$
and $\Sigma_{\theta^{(i)}}  = \lambda BB^\T + \bI_p$. Following the arguments in the proof of Lemma \ref{lem_KL} yields
\begin{align}\label{eqn_KL_alphas}\nonumber
	\KL(\PP_{\theta^{(1)}}, \PP_{\theta^{(i)}}) =  & ~  {1\over 2}\left(\mu_{\theta^{(i)}} - \mu_{\theta^{(1)}}\right)^\T  \left(\lambda BB^\T + \bI_p\right)^{-1} \left(\mu_{\theta^{(i)}} - \mu_{\theta^{(1)}}\right)\\\nonumber
	= & ~ {\lambda\over 2}(\a^{(i)}- \a^{(1)})^\T B^\T {1\over \lambda+1} B B^\T B (\a^{(i)}-\a^{(1)}) &&\textrm{by (\ref{eq_Sigma_inv_1}),}\\
	= & ~ {\lambda\Dt^2 \over 8(1+\lambda)} (\eps')^2 \|J^{(i)} - J^{(1)}\|_2^2\\\nonumber
	\le & ~ ~ {c_0c_1 (K-1)\over 2n} &&\textrm{by }\|J^{(i)} - J^{(1)}\|_2^2\le 4.
\end{align}
Claim (1) then follows from $\log M' \ge c K$  by using Lemma \ref{lem_J} and the additivity of KL divergence among independent distributions.  Since claim  (2) is proved in Lemma \ref{lem_lb_L_prime}, the proof is complete.\qed 

\medskip

\begin{lemma}\label{lem_lb_L_prime}
	Let $\theta^{(i)}$ for $1\le i\le M'$ be constructed as (\ref{def_thetas_prime}) -- (\ref{def_alphas}). Under $K/ n \le c_1$ and $1/\lambda \le c_2$, for any measurable $\wh g$, one has  
	\[
	L_{\theta^{(i)}}(\wh g) + L_{\theta^{(j)}}(\wh g) ~  \gtrsim  ~ {1\over \Dt_x}  e^{-{\Dt_x^2 / 8}} \sqrt{K\over n}.
	\]
	with $\Dt_x^2 = \lambda \Dt^2/(1+\lambda)$.
\end{lemma}
\begin{proof}
	The proof uses the same reasoning for proving Lemma \ref{lem_lb_L}.  Pick any $i\ne j \in \{0,\ldots, M'\}$ and write $L_\theta = L_{\theta^{(i)}}(\wh g)$ and $L_{\theta'}= L_{\theta^{(j)}}(\wh g)$.  From (\ref{lb_L_theta_theta_prime}), one has 
	\begin{align*}
		L_{\theta^{(i)}} + L_{\theta^{(j)}}  \ge e^{- {\Dt_x^2/ 8}} \int_{\substack{x^\T \Sigma^{-1}\mut \ge 0\\
				x^\T \Sigma^{-1}\mutp <0}}{	|\Sigma|^{-1/2}\over (2\pi)^{p/2}}
		\exp\left(-{1\over 2}x^\T \Sigma^{-1} x\right){\rm d} x 
	\end{align*}
	where $\Sigma := \Sigt = \Sigtp =\lambda BB^\T + \bI_p$.  Let $H\in \cO_{p\times p}$ such that 
	\[
	H \Sigma^{-1}\mut  = \begin{bmatrix}
		a \\ b \\ \0_{p-2}
	\end{bmatrix} := \begin{bmatrix}
		u \\ \0_{p-2}
	\end{bmatrix},\qquad H\Sigma^{-1}\mutp = \begin{bmatrix}
		a \\ -b \\ \0_{p-2}
	\end{bmatrix} := \begin{bmatrix}
		v \\ \0_{p-2}
	\end{bmatrix},\qquad a > 0.
	\]
	By changing variable $y = Hx$ and writing $y_I^\T = (y_1, y_2)$,  we find 
	\begin{align*}
		L_{\theta^{(i)}} + L_{\theta^{(j)}}  &\ge e^{- {\Dt_x^2\over 8}} \int_{\substack{y_I^\T u \ge 0\\
				y_I^\T v <0}}{|H\Sigma H^\T| \over (2\pi)^{p/2}}
		\exp\left(-{1\over 2}y^\T H\Sigma^{-1}H^\T  y\right)dy \\
		&= e^{- {\Dt_x^2\over 8}} \int_{\substack{y_I^\T u \ge 0\\
				y_I^\T v <0}}{|Q_{II}| \over 2\pi}
		\exp\left(-{1\over 2}y_I^\T Q_{II}^{-1} y\right)dy_I
	\end{align*}
	where $Q_{II}$ is the first $2\times 2$ matrix of 
	\[
	Q = H\Sigma H^\T. 
	\]
	By another change of variable and the same reasoning in the proof of Lemma \ref{lem_lb_L}, 
	\begin{align*}
		L_{\theta^{(i)}} + L_{\theta^{(j)}}  &\ge e^{- {\Dt_x^2\over 8}} \int_{\substack{z^\T Q_{II}^{1/2}u \ge 0\\
				z^\T Q_{II}^{1/2}v <0}}{1\over 2\pi}
		\exp\left(-{1\over 2}z^\T z\right)dz\\
		&\ge e^{- {\Dt_x^2\over 8}} {1\over 2\pi} \|\wt \zeta_1 - \wt \zeta_2\|_2,
	\end{align*}
	where 
	\[
	\wt\zeta_1 = {Q_{II}^{1/2}u \over 
		\sqrt{u^\T Q_{II}u}},\qquad \wt \zeta_2 = {Q_{II}^{1/2}v \over 
		\sqrt{v^\T Q_{II}v}}.
	\]
	Since 
	\[
	u^\T Q_{II}u = \mu_{\theta}^\T \Sigma^{-1}H^\T H\Sigma H^\T H \Sigma^{-1}\mut =  \mut^\T \Sigma^{-1}\mut = {\Dt_x^2 \over 4} = v^\T Q_{II}v
	\]
	and 
	\begin{align*}
		\|Q_{II}^{1/2}(u-v)\|_2^2 &= (\mut - \mutp)\Sigma^{-1}(\mut - \mutp)\\
		& = {\lambda \Dt^2\over 4(1+\lambda)} (\eps')^2 \|J^{(j)} - J^{(i)}\|_2^2 &&\textrm{by (\ref{eqn_KL_alphas})}\\
		&\asymp {\lambda K \over (1+\lambda) n} = o(1) &&\textrm{by (\ref{def_epsilon_prime})},
	\end{align*}
	we conclude
	\[
	L_{\theta^{(i)}} + L_{\theta^{(j)}} \gtrsim e^{- {\Dt_x^2/ 8}}  {\|Q_{II}^{1/2}(u-v)\|_2 \over \Dt_x}\asymp 
	{1\over \Dt_x}e^{- {\Dt_x^2/ 8}} \sqrt{{\lambda  \over 1+\lambda}}\sqrt{K\over n}.
	\]
	Using $\lambda \ge c$ completes the proof. 
\end{proof}

\section{Technical lemmas}\label{app_tech}

Consider $\pi_0 + \pi_1 = 1$. This section contains some basic relations between $\a_0$ and $\a_1$, collected in Lemma \ref{fact}, as well as some useful technical lemmas.

\begin{lemma}\label{fact}
	Let  $\ba  := \pi_0 \a_0 + \pi_1\a_1$. We have 
	\[
	\pi_0\a_0\a_0^\T + \pi_1\a_1 \a_1^\T - \ba \ba^\T  = \pi_0\pi_1(\a_1-\a_0) (\a_1 - \a_0)^\T.
	\]
	Additionally, for any $M\in \RR^{K\times K}$,  we have
	\begin{align*}
		&\pi_0\a_0^\T M \a_0 + \pi_1\a_1^\T M \a_1 - \ba^\T M \ba = \pi_0\pi_1(\a_1-\a_0)^\T M (\a_1 - \a_0).
	\end{align*}
	As a result, 
	\[
	\a_0^\T M \a_0 + \a_1^\T M \a_1 - \ba^\T M \ba  ~\le ~  \max\{\pi_0, \pi_1\}\cdot (\a_1-\a_0)^\T M (\a_1 - \a_0).
	\]
\end{lemma}

\medskip

The following lemma provides concentration inequalities of $\wh\pi_k - \pi_k$.  

\begin{lemma}\label{lem_pi_hat}
	For any $k\in \{0,1\}$ and all $t>0$, 
	\[
	\PP\left\{
	|\wh\pi_k - \pi_k| > \sqrt{\pi_k(1-\pi_k)t\over n}+ {t\over n}
	\right\}\le 2e^{-t/2}.
	\]
	In particular, if $\pi_0\pi_1 \ge 2\log n/ n$, then for any $k\in \{0,1\}$, 
	\[
	\PP\left\{
	|\wh\pi_k - \pi_k| < \sqrt{8\pi_0\pi_1\log n\over n}
	\right\}\ge 1-2n^{-1}.
	\]
	Furthermore, if $\pi_0\pi_1 \ge C\log n/ n$ for some sufficiently large constant $C$, then 
	\[
	\PP\left\{
	c\pi_k \le \wh \pi_k \le c'\pi_k
	\right\} \ge 1-2n^{-1}.
	\]
\end{lemma}
\begin{proof}
	The first result follows from an application of the Bernstein inequality for bounded random variables. The second one follows by choosing $t = 2\log n$ and the last one can be readily seen from the second display.  
\end{proof}

\subsection{Deviation inequalities of quantities related with $Z$}

Recall that $\ba = \EE[Z]$,  $\sz=\Cov(Z)$ and $ \wt \Z = \Z \sz^{-1/2}$. Let the centered $\wt \Z$ be defined as 
\[
\R = (R_1, \ldots, R_n)^\T,\qquad \textrm{with}\quad  R_i = \wt Z_i - \sz^{-1/2}\ba.
\]
The following lemma provides concentration inequalities 
of $\wh\a_k - \a_k$ and some useful bounds related with the random matrices
$\R$ 
and $\wt\Z^\T \Pi_n \wt \Z$.

\begin{lemma}\label{lem_deviation}
	Under assumption ${\rm (iv)}$, the following results hold.
	\begin{enumerate}
		\item[(i)]  For any deterministic vector $u\in \RR^K$,  for all $t>0$,
		\[
		\PP\left\{
		\left|u^\T (\wh \a_k-\a_k)\right| \ge t\sqrt{u^\T  \szy u \over n_k} ~
		\right\} \le 2e^{-{t^2/ 2}}.
		\]
		\item [(ii)]
		\[
		\PP\left\{ \left\|
		\sz^{-1/2}(\wh \a_k - \a_k) 
		\right\|_2 \le 2\sqrt{K\log n \over n_k}
		\right\} \ge 1-2K/n^2.
		\]
		\item [(iii)] With probability $1-4Kn^{-2}-4n^{-1}$, 
		\[
		{1\over n}\left\|
		\sum_{i=1}^n R_i\right\|_2 \le  2(2+\sqrt{2})\sqrt{K\log n \over n}.
		\]
		\item [(iv)] For any deterministic vector $u,v\in \RR^K$, with probability $1- 4n^{-1}-8Kn^{-2}$, 
		\begin{align*}
			\left|u^\T\left(
			{1\over n}\sum_{i=1}^n R_i R_i^\T -\bI_K \right)v^\T \right| &\lesssim  \|u\|_2\|v\|_2\sqrt{\log n \over n}\left(1+ \|\a_1-\a_0\|_{\sz}\right)
		\end{align*}

		\item [(v)] With probability $1- \cO(1/n)$, 
		\begin{align*}
			\left\|
			{1\over n}\R^\T \R -\bI_K \right\|_{\op} &\lesssim \sqrt{K\log n\over n} + {K\log n\over n} 
			+ \|\a_1-\a_0\|_{\sz}\sqrt{\log n\over n} .
		\end{align*}

		\item [(vi)]  Assume $K\log n\le c_0 n$ for some sufficiently small constant $c_0>0$. With probability $1- \cO(1/n)$, the inequalities
		\[
		c \le {1\over n}\lambda_{K}(\R^\T  \R)\le  {1\over n}\lambda_{1}( \R^\T  \R) \le C
		\]
		hold for some constants $0<c\le C<\i$ depending on $c_0$ only.
		
		\item [(vii)] Assume $K\log n\le c_0 n$ for some sufficiently small constant $c_0>0$. There exists some absolute constant $C>0$ such that, with probability $1- \cO(1/n)$,  
		\[
		\left\|
		{1\over n}\wt\Z^\T \Pi_n\wt \Z -\bI_K \right\|_{\op} \le C \sqrt{K\log n\over n}  
		\]
		and 
		\[
		{1\over 2} \le {1\over n}\lambda_{K}(\wt \Z^\T \Pi_n \wt \Z)\le  {1\over n}\lambda_{1}(\wt \Z^\T \Pi_n \wt \Z) \le 2.
		\]
	\end{enumerate}
\end{lemma}

\begin{proof}
	Without loss of generality, we assume $\ba = 0_K$ so that $\wt \Z = \R$.
	
	To prove (i), we first condition on $Y_i$ and  use the fact that $Z_i\mid Y_i=k$ are independent $N(\a_k, \szy)$, to conclude that, for all $t>0$ and any deterministic $u\in \RR^K$, 
	\[
	\PP\left\{
	\left|u^\T (\wh \a_k-\a_k)\right| \ge t \sqrt{u^\T  \szy u \over n_k} ~~ \bigg | ~ \Y 
	\right\} \le 2\exp\left(-{t^2 \over 2}\right).
	\]
	After we take the expectation of this bound over $\Y$, we immediately obtain (i).\\

	To show part (ii), 
	we observe that, using part (i),
	\begin{align*}
		\|
		\sz^{-1/2}(\wh \a_k - \a_k) 
		\|_2^2  &=\sum_{j=1}^K  \left( \e_j^\T \sz^{-1/2}(\wh \a_k - \a_k) \right)^2
		\\
		&\le \sum_{j=1}^K t^2 \frac{1}{n_k} \e_j^\T \sz^{-1/2} \szy \sz^{-1/2} \e_j\\
		&\le  {K t^2 \over n_k}
	\end{align*}
	The last inequality  uses $\|\sz^{-1/2}\szy\sz^{-1/2}\|_{\op}\le 1$, which we  deduce in turn from (\ref{eq_sz_C}).
	Next, we take $t = 2\sqrt{\log n}$ and we conclude 
	\begin{align*}
		\PP\left\{	\|
		\sz^{-1/2}(\wh \a_k - \a_k) 
		\|_2 \le  2  \sqrt{K\log n \over n_k}~ \right\} \ge 1-\frac{2K}{n^2}.
	\end{align*}

	To prove part (iii),
	we find, after adding and subtracting terms and using 
	\begin{equation}\label{eq_alphas}
		\EE[Z] = \ba = \0 _K= \pi_1 \a_1 + \pi_0\a_0, 
	\end{equation}
	the identity 
	\begin{align*}
		\sum_{i=1}^n Z_i &= \sum_{i: Y_i = 1} Z_i+\sum_{i: Y_i = 0} Z_i\\
		&= \sum_{i: Y_i = 1}( Z_i - \a_1)+\sum_{i: Y_i = 0}(Z_i-\a_0) + (n_1-n\pi_1)\a_1+ (n_0-n\pi_0)\a_0\\
		&= \sum_{i: Y_i = 1}(Z_i - \a_1)+\sum_{i: Y_i = 0}(Z_i-\a_0) + (n\pi_0-n_0)\a_1+ (n_0-n\pi_0)\a_0\\
		&= \sum_{i: Y_i = 1}(Z_i - \a_1)+\sum_{i: Y_i = 0}(Z_i-\a_0) + (n\pi_0-n_0)(\a_1-\a_0)
	\end{align*}
	In the third equality we used $n_0 + n_1 = n$ and $\pi_0 + \pi_1 = 1$. 
	From this identity, using the definitions (\ref{def_alpha_hat}) of $\a_k$ and (\ref{def_pi_hat}) of $n_k$, we find that
	\begin{align*}
		{1\over n}\left\|
		\sum_{i=1}^n\wt R_i\right\|_2 
		&= {1\over n}\left\|
		\sum_{i=1}^n\wt Z_i\right\|_2 \\
		&\le 	\sqrt{n_1\over n}\left\|\sz^{-1/2}(\wh\a_1 - \a_1)\right\|_2 + \sqrt{n_0\over n}\left\|\sz^{-1/2}(\wh\a_0 - \a_0)\right\|_2 \\
		&\quad + |\wh\pi_0 - \pi_0|\cdot \|\a_1 - \a_0\|_{\sz}.
	\end{align*}
	We invoke part (ii), Lemma \ref{lem_pi_hat} and the inequality
	\begin{equation}\label{eq_Delta}
		\pi_0\pi_1 \|\a_1 - \a_0\|_{\sz}^2 \le 
		{1\over 4} \min(1,\Dt^2)\le 1 \qquad \text{ using (\ref{eq_Deltas})}
	\end{equation}
	to 
	complete the proof of (iii).\\
	
	To prove (iv), observe that 
	\begin{align*}
		\sum_{i=1}^n Z_i Z_i^\T &=  \sum_{i: Y_i = 1}Z_iZ_i^\T + \sum_{i: Y_i = 0}Z_iZ_i^\T\\
		&=  \sum_{k\in \{0,1\}}\left[\sum_{i: Y_i = k}(Z_i-\a_k)(Z_i-\a_k)^\T + n_k (\wh \a_k \a_k^\T + \a_k \wh \a_k^\T) - n_k\a_k\a_k^\T\right]\\
		&= \sum_{k\in \{0,1\}}\left[\sum_{i: Y_i = k}(Z_i-\a_k)(Z_i-\a_k)^\T + n_k (\wh \a_k-\a_k) \a_k^\T + n_k\a_k (\wh \a_k-\a_k)^\T\right]\\
		&\qquad  + \sum_{k\in \{0,1\}}n_k\a_k\a_k^\T.
	\end{align*}
	Since (\ref{eq_sz_C}), (\ref{eq_alphas})  and Lemma \ref{fact} imply
	\[
	\sz = \szy + \sum_{k\in \{0,1\}}\pi_k\a_k \a_k^\T,
	\]
	we obtain, for any $u,v\in \RR^K$,
	\begin{align}\label{eq_decomp_quad_Z}\nonumber
		u^\T \left({1\over n}\sum_{i=1}^n Z_i Z_i^\T  - \sz\right) v
		&=
		\sum_{k\in \{0,1\}}{n_k\over n}u^\T  \left[{1\over n_k}\sum_{i: Y_i = k}(Z_i-\a_k)(Z_i-\a_k)^\T-\szy\right] v^\T\\\nonumber
		&\quad  + \sum_{k\in \{0,1\}}{n_k\over n}
		v^\T (\wh \a_k - \a_k)\a_k^\T u  + \sum_{k\in \{0,1\}}{n_k\over n}
		u^\T (\wh \a_k - \a_k)\a_k^\T v  \\
		&\quad+ \sum_{k\in \{0,1\}}(\wh \pi_k - \pi_k)u^\T \a_k \a_k^\T v.
	\end{align}
	Notice that 
	\[
	u^\T \left({1\over n}\sum_{i=1}^n \wt Z_i \wt Z_i^\T  - \bI_K\right) v = \wt u^\T \left({1\over n}\sum_{i=1}^n Z_i Z_i^\T  - \sz\right) \wt v
	\]
	with $\wt u = \sz^{-1/2}u$ and $\wt v = \sz^{-1/2}v$. By conditioning on $\Y$, standard Gaussian concentration inequalities give
	\begin{align*}
		&\left|\wt u^\T \left({1\over n_k} \sum_{i: Y_i = k}(Z_i-\a_k)(Z_i-\a_k)^\T- \szy \right)\wt v\right|\\
		&\qquad \lesssim \sqrt{\wt u^\T \szy \wt u}\sqrt{\wt v^\T \szy \wt v}\left(\sqrt{\log n\over n_k} + {\log n\over n_k}\right)
	\end{align*}
	with probability $1-\cO(n^{-1})$. By further invoking Lemma \ref{lem_pi_hat} and part (i), we 
	conclude 
	\begin{align*}
		\left|\wt u^\T \left({1\over n}\sum_{i=1}^n Z_i Z_i^\T  - \sz\right) \wt v\right|
		&\lesssim
		\sqrt{\wt u^\T \szy \wt u}\sqrt{\wt v^\T \szy \wt v}\sum_{k\in \{0,1\}}{n_k\over n}\left(\sqrt{\log n\over n_k} + {\log n\over n_k}\right)\\ 
		& \quad + \sqrt{\wt v^\T \szy\wt v}\sum_{k\in \{0,1\}}\sqrt{n_k\log n\over n^2}|\wt u^\T \a_k| \\
		&  \quad + \sqrt{\wt u^\T \szy\wt u} \sum_{k\in \{0,1\}}\sqrt{n_k\log n\over n^2}|\wt v^\T \a_k|\\ 
		& \quad + 
		\sqrt{\pi_0\pi_1\log n\over n}\sum_{k\in \{0,1\}}|\wt u^\T \a_k|^2.
	\end{align*}
	with probability $1- 4n^{-c''} - 4n^{-1}-8Kn^{-2}$. Since 
	\[
	|\wt u^\T \a_k| \le \|u\|_2 \|\a_k\|_{\sz}
	\]
	from the Cauchy-Schwarz inequality, by noting that 
	\[
	\wt u^\T \szy \wt u \le \|u\|_2^2 \|\sz^{-1/2}\szy\sz^{-1/2}\|_{\op}\le \|u\|_2^2
	\] 
	and invoking Lemma \ref{fact} for 
	\[
	\sum_{k\in \{0,1\}}\|\a_k\|_{\sz} \le \sqrt{2}\|\a_1-\a_0\|_{\sz},\quad 	\sum_{k\in \{0,1\}}\|\a_k\|_{\sz}^2 \le \|\a_1-\a_0\|_{\sz}^2,
	\]
	we conclude, with the same probability,
	\begin{align*}
		&\left|\wt u^\T \left({1\over n}\sum_{i=1}^n Z_i Z_i^\T  - \sz\right) \wt v\right|\\ 
		&\lesssim \|u\|_2\|v\|_2\sqrt{\log n \over n}\left(1+ \|\a_1-\a_0\|_{\sz}+
		\sqrt{\pi_0\pi_1}\|\a_1-\a_0\|_{\sz}^2\right)\\
		&\lesssim \|u\|_2\|v\|_2\sqrt{\log n \over n}\left(1+ \|\a_1-\a_0\|_{\sz}\right)
	\end{align*}
	where we used (\ref{eq_Delta}) in the last line.

	Next, we prove (v) by bounding from above 
	\[
	\sup_{u\in \RR^K}u^\T\left({1\over n}\sum_{i=1}^n Z_i Z_i^\T  - \sz\right)u.
	\]
	Recalling that (\ref{eq_decomp_quad_Z}), an application of Lemma  \ref{lem_op_diff} yields 
	\[
	\left\|{1\over n_k}\sum_{i: Y_i = k}\szy^{-1/2}(Z_i-\a_k)(Z_i-\a_k)^\T \szy^{-1/2}-\bI_K\right\|_{\op} \le c'\left(\sqrt{K\log n\over n_k} + {K\log n\over n_k}\right)
	\]
	with probability $1-2n^{-c''K}$. The result follows by the same arguments of proving (iv) and also by noting that the other terms are bounded uniformly over $u\in \RR^K$.
	
	As a result of (v), part (vi) follows from the bound (\ref{quad_beta_sz}) and Weyl's inequality. 
	
	Finally, to prove (vii), observe that 
	\begin{align*}
		{1\over n}\wt\Z^\T \Pi_n \wt \Z &= {1\over n}\sum_{i=1}^n \wt Z_i \wt Z_i^\T - \sz^{-1/2} \bar Z\bar Z^\T \sz^{-1/2}
	\end{align*}
	with $\bar Z = \sum_{i=1}^n Z_i / n$. Consequently, 
	\[
	\left\|{1\over n}\wt\Z^\T \Pi_n \wt \Z - \bI_K\right\|_\op \le 
	\left\|{1\over n}\wt\Z^\T \wt \Z - \bI_K\right\|_\op  + \left\|{1\over n}\sum_{i=1}^n \wt Z_i\right\|_2^2.
	\]
	Invoking (iii) and (v) gives the desired result. The bounds on the  eigenvalues of $\wt \Z^\T \Pi_n \wt \Z$ follow from Weyl's inequality. 
\end{proof}

\subsection{Deviation inequalities of quantities related with $W$}

The following lemma provides deviation inequalities for various quantities related with $\W$. Recall that 
\[
\bar W_{(k)} = {1\over n_k}\sum_{i=1}^n W_i \1\{Y_i = k\},\qquad \forall~  k \in \{0, 1\}.
\]
Further recall that $\E_z$ is defined in (\ref{def_event_z}).

\begin{lemma}\label{lem_W}
	Under assumptions {\rm (i) -- (vi)} and $K \le  n$,
	the following results hold. 
	\begin{align*}
		&  
		\PP\left\{{1\over n} \|\W\|_F^2 \le 
		6\gamma^2 \tr(\sw) \right\} \ge 1-e^{-n}, 
		\\
		&  
		\PP\left\{{1\over \sqrt n}\left\|\W A^{+\T}\sz^{-1/2}\right\|_{\op} \le 
		12\gamma^2 \sqrt{\|\sw\|_{\op} \over \lambda_K} \right\} \ge 1-e^{-n}, 
		\\
		& 
		\PP\left\{{1\over \sqrt n}\left\|\W P_A\right\|_{\op} \le  12\gamma^2 \sqrt{\|\sw\|_\op}\right\} \ge 1-e^{-n},
		\\
		&  
		\PP\left\{\left\|P_A \bar W_{(k)}\right\|_2 \lesssim \sqrt{\|\sw\|_\op}\sqrt{K\log n \over n}\right\} \ge 1-n^{-K}, \quad \text{for }k=0,1\\
		& 
		%%		\left\{
		\PP\left\{{1\over n}\left\|\wt \Z^\T  \Pi_n \W P_A\right\|_{\op} \lesssim \sqrt{\|\sw\|_\op} \sqrt{K\log n\over n}
		\right\} = 1-\cO(n^{-1}),
		\\
		& 
		\PP\left\{	 {1\over n}\left\|
		P_{A} \W^\T\Pi_n \Y
		\right\|_2  \lesssim \sqrt{\|\sw\|_\op} \sqrt{K\log n\over n}		  \right\} \ge 1-2n^{-K}.
	\end{align*}
\end{lemma}
\begin{proof}
	Recall that $\W = \wt\W \sw^{1/2}$. Observe that 
	$\|\W\|_F^2 = \vec(\wt \W)^\T M ~ \vec(\wt \W)$ where $\vec(\wt\W)$ is the vectorized form (by rows) of $\wt\W$ and $M = \bI_n \otimes \sw$. Since 
	$\vec(\wt \W)$ is subGaussian with subGaussian parameter $\gamma$, applying Lemma \ref{lem_quad} with $\xi = \vec(\wt \W)$ and $H = M$ yields, for all $t\ge 0$,
	\[
	\PP\left\{
	\|\W\|_F^2 > 2\gamma^2\left(
	\tr(M) + 2t\|M\|_\op 
	\right)
	\right\}\le e^{-t}.
	\]
	Since $\tr(M) = n\tr(\sw)$ and $\|M\|_\op \le \|\sw\|_\op \le \tr(\sw)$, the first result follows by taking $t = n$. 
	
	Invoke Lemma \ref{lem_op_norm} with $\G = \wt\W$ and $H =\sw^{1/2}A^{+\T}\sz^{-1}A^+\sw^{1/2}$  together with $\tr(H)\le K\|H\|_{\op}$,
	$\|H\|_{\op} \le \|\sw\|_{\op} / \lambda_K$ and $K \le n$  to obtain 
	\[
	\PP\left\{{1\over \sqrt n}\left\|\W A^{+\T}\sz^{-1/2}\right\|_{\op} \le  12\gamma^2\sqrt{\|\sw\|_{\op} \over \lambda_K}\right\} \ge 1-e^{-n}.
	\]
	
	Similarly, by invoking Lemma \ref{lem_op_norm} and using $K\le n$, the second result follows from 
	\begin{align}\label{bd_W_P_diff}
		{1\over n}\left\|\W P_A\right\|_{\op}^2  &\le \gamma^2\left( \sqrt{6\|P_A\sw P_A\|_\op}+ \sqrt{\tr(P_A\sw P_A)\over  n}\right)^2  \le 12\gamma^2 \|\sw\|_\op  
	\end{align}
	with probability at least $1-e^{-n}$.

	Regarding the third result, since  $\sw^{-1/2}\bar W_{(k)}$ given $\Y$ is $\sqrt{\gamma^2/n_k}$-subGaussian, Lemma \ref{lem_quad} gives
	\begin{align}\label{bd_WPA2}\nonumber
		\left\|P_A\bar W_{(k)}\right\|_2   &\lesssim    \sqrt{{1\over n}\Bigl[\tr(P_A\sw P_A) + \|P_A\sw P_A\|_{\op} K\log n\Bigr]} \\
		&\le  \sqrt{{K +K\log n\over n}\|\sw\|_{\op}},
	\end{align}
	with probability $1- n^{-K}$. The last inequality in (\ref{bd_WPA2}) uses $\tr(P_A\sw P_A) \le K \|\sw\|_{\op}$.\\ 
	
	To prove the fourth result, let $P_A = \U_A \U_A^\T$ with $\U_A \in \cO_{p\times K}$. 
	Further let $\cN_K(1/4)$ be the $(1/4)$-net of $\cS^{K}$. By the properties of $\cN_K(1/4)$, we have 
	\begin{align*}
		{1\over n}\|\wt \Z^\T \Pi_n \W P_A\|_{\op} = {1\over n}\|\wt \Z^\T \Pi_n \W \U_A\|_{\op}
		& = \sup_{u \in \cS^K, v\in \cS^K} u^\T\wt \Z^\T \Pi_n \W \U_A v \\
		&\le 2\max_{u \in \cN_K(1/4), v\in \cN_K(1/4)} u^\T\wt \Z^\T \Pi_n \W \U_A v.
	\end{align*}
	Furthermore,  
	\begin{align}\label{eq_quad_ZW}\nonumber
		u^\T\wt \Z^\T \Pi_n \W \U_A v &= {1\over n}\sum_{i=1}^n u^\T \left(
		\wt Z_i - {1\over n}\sum_{i=1}^n \wt Z_i
		\right) (W_i - \bar W)^\T \U_Av\\\nonumber
		&= {1\over n}\sum_{i=1}^n u^\T \left(
		\wt Z_i -\sz^{-1/2}\ba
		\right) (W_i - \bar W)^\T \U_Av\\
		&= {1\over n}\sum_{i=1}^n u^\T \left(
		\wt Z_i -\sz^{-1/2}\ba
		\right) W_i^\T  \U_A  v -  u^\T {1\over n}\sum_{i=1}^n\left(
		\wt Z_i -\sz^{-1/2}\ba
		\right)  \bar W^\T   \U_A v.
	\end{align}
	By (iii) of Lemma \ref{lem_deviation} and (\ref{bd_WPA2}), 
	the second term can be bounded from above, uniformly over $u,v\in \cN_K(1/4)$, as 
	\[
	\left\|{1\over n}\sum_{i=1}^n\left(
	\wt Z_i -\sz^{-1/2}\ba
	\right)\right\|_2 \left\|  \U_A \bar W\right\|_2 \lesssim  \sqrt{\|\sw\|_\op {K\log n\over n}}
	\]
	with probability $1-cn^{-1}$.

	It remains to show that the same bound holds for the first term in (\ref{eq_quad_ZW}). Since $\Z$ and $\W$ are independent, conditioning on $\wt \Z$, we know $u^\T (\wt Z_i - \sz^{-1/2}\ba) W_{i}^\T \U_A v$ is sub-Gaussian with sub-Gaussian constant equal to 
	\[
	\sqrt{v^\T \U_A^\T \sw \U_A v}\sqrt{u^\T (\wt Z_i-\sz^{-1/2}\ba)(\wt Z_i-\sz^{-1/2}\ba)^\T u} \le \sqrt{\|\sw\|_\op}\sqrt{u^\T R_iR_i^\T u},
	\]
	recalling that $R_i = \wt Z_i-\sz^{-1/2}\ba$. Thus,
	$n^{-1}\sum_{i=1}^n u^\T (\wt Z_i -\sz^{-1/2}\ba) W_{i}^\T \U_A v$ is  sub-Gaussian with sub-Gaussian constant equal to
	\[
	{1\over n}\sqrt{\|\sw\|_\op\sum_{i=1}^n u^\T  R_i R_i^\T u} \le \sqrt{{\|\sw\|_\op \over n}\left\|{1\over n}\R^\T\R\right\|_\op}.
	\]
	We conclude that, for each $u,v\in \cN_K(1/4)$, 
	\[
	\PP\left\{
	{1\over n}\sum_{i=1}^n u^\T (\wt Z_i-\sz^{-1/2}\ba) W_{i}^\T \U_A v \ge t\sqrt{\|\sw\|_\op \over n}\left\|{1\over n}\R^\T\R\right\|_\op
	\right\} \le e^{-t^2/2}.
	\]
	The result follows by choosing $t = C\sqrt{K\log n}$ for some sufficiently large constant $C>0$, taking a union bounds over $\cN_K(1/4)$ together with $|\cN_K(1/4)| \le 9^K$, and invoking  (v) of Lemma \ref{lem_deviation}.\\
	
	Finally, to prove the last claim,  recall from (\ref{def_W_bar}) that $$
	\W^\T\Pi_n  \Y = \W^\T \Y - {1\over n}\W^\T \b1_n \b1_n^\T \Y =  n_1 (\bar W_{(1)}  - \bar W),
	$$
	with $\bar W = \sum_{i=1}^n \W/n$. We thus find that, with probability $1-2n^{-K}$, 
	\begin{align}\label{bd_intermediate_ZPXA_comp}
		{1\over n}\left\|
		P_A \W^\T\Pi_n \Y
		\right\|_2 
		&\le \left\|
		P_A\bar W_{(1)}\right\|_2 +\left\|
		P_A\bar W\right\|_2  \lesssim \sqrt{K \log n\over n}\sqrt{\|\sw\|_{\op}}
	\end{align}
	where the last step uses the bound in (\ref{bd_WPA2}).
\end{proof}

\section{Auxiliary lemmas}\label{app_aux}

The following lemma is the tail inequality for a quadratic form of sub-Gaussian random vectors. We refer to \cite[Lemma 16]{bing2020prediction} for its proof.  Also, see Lemma 30 in \cite{Hsu2014}. 
\begin{lemma} \label{lem_quad}
	Let $\xi\in \RR^d$ be a $\gamma_\xi$ sub-Gaussian random vector. Then, for all symmetric positive semi-definite matrices $H$, and all $t\ge 0$, 
	\[
	\PP\left\{
	\xi^\T H\xi > \gamma_\xi^2\left(
	\sqrt{{\rm tr}(H)}+ \sqrt{2 t \|H\|_{\rm op} }
	\right)^2
	\right\} \le e^{-t}.
	\] 
\end{lemma}

The following lemma provides an upper bound on the operator norm of $\G H \G^\T$ where  $\G\in \RR^{n\times d}$ is a random matrix and its rows are independent sub-Gaussian random vectors. It is proved in Lemma 22 of \cite{bing2020prediction}.
\begin{lemma}\label{lem_op_norm}
	Let $\G$ be a $n\times d$ matrix with rows that are independent $\gamma$ sub-Gaussian  random vectors with identity covariance matrix. Then, for all symmetric positive semi-definite matrices $H$, 
	\[
	\PP\left\{{1\over n}\| \G H \G^\T \|_{{\rm op}} \le \gamma^2\left( \sqrt{{\rm tr}(H) \over n} + \sqrt{6\|H\|_{\op}}
	\right)^2\right\} \ge  1 -  e^{-n}
	\]
\end{lemma}

Another useful concentration inequality of the operator norm of the random matrices with i.i.d. sub-Gaussian rows is stated in the following lemma \cite[Lemma 16]{bing2020prediction}. This is an immediate result of \cite[Remark 5.40]{vershynin_2012}.

\begin{lemma}\label{lem_op_diff} 
	Let $\G$ be $n$ by $d$ matrix whose rows are i.i.d. $\gamma$ sub-Gaussian  random vectors with covariance matrix $\Sigma_Y$. Then, for every $t\ge 0$, with probability at least  $1-2e^{-ct^2}$,
	\[
	\left\|	{1\over n}\G^\T \G - \Sigma_Y\right\|_{{\rm op}}\le \max\left\{\delta, \delta^2\right\} \left\|\Sigma_Y\right\|_{{\rm op}},
	\]
	with $\delta = C\sqrt{d/n}+ t/\sqrt n$ where $c = c(\gamma)$ and $C=C(\gamma)$ are positive constants depending on $\gamma$.
\end{lemma}

\section{Additional simulation results}\label{app_sim}

\subsection{Performance of PCLDA when $K$ cannot be estimated consistently}\label{app_sim_K}
In this section, we report our findings of a simulation study on the performance of the PCLDA classifier in situations when $K$ cannot be estimated consistently. We used the same generating mechanism as Section \ref{sec_sim}, except for the way of generating the matrix $A$. Here, for $k = 1,\ldots, K$, the entries of the column $A_{\cdot k}$ are generated   independently from a normal  $N(0, \sigma_{A,k}^2)$ distribution with variance parameter
\[
\sigma_{A,k}^2 = {2\over p}p^{K-k \over K-1}.
\] 
For $K = o(p)$, standard concentration inequalities on the singular values of $A$ give
\[
\lambda_k(A^\T A) \asymp p^{K-k \over K-1},\quad \text{for }k = 1,\ldots, K,
\]
with high probability. 
Since the matrix $\szy$ has bounded eigenvalues, the first $K$ eigenvalues of $A\szy A^\T$ follow  the same rates as above. In particular, we have 
$\lambda_K:= \lambda_K(A\szy A^\T) \asymp 1$, whence the condition $\xi \ge C$ on the signal-to-noise ratio in Theorem \ref{thm_PCR_K_hat} fails to hold for $p > n$.
In this case, we should not expect that $\wh K$ consistently estimates $K$. 

We fix $K = 10$, $p=500$ and vary $n\in \{50, 100, 200, 300, 500\}$. Each setting is repeated $100$ times and the number of data points in the test set is increased to $300$. 

Figure \ref{fig_K_1} depicts the performance of  PCLDA-$\wh K$ and PCLDA-$K$ as well as other methods mentioned in Section \ref{sec_sim}. 
We see that (i)  PCLDA-$\wh K$ performs as well as PCLDA-$K$ even though the selected $\wh K$ is $\{8, 9, 12, 17, 27\}$ (the true $K$ is $10$), corresponding to each choice of $n$; (ii)  As $n$ increases, $\wh K$ tends to overestimate $K$, which, however, does not lead to higher misclassification rates.

\begin{figure}[ht]
	\centering
	\includegraphics[width=.5\textwidth]{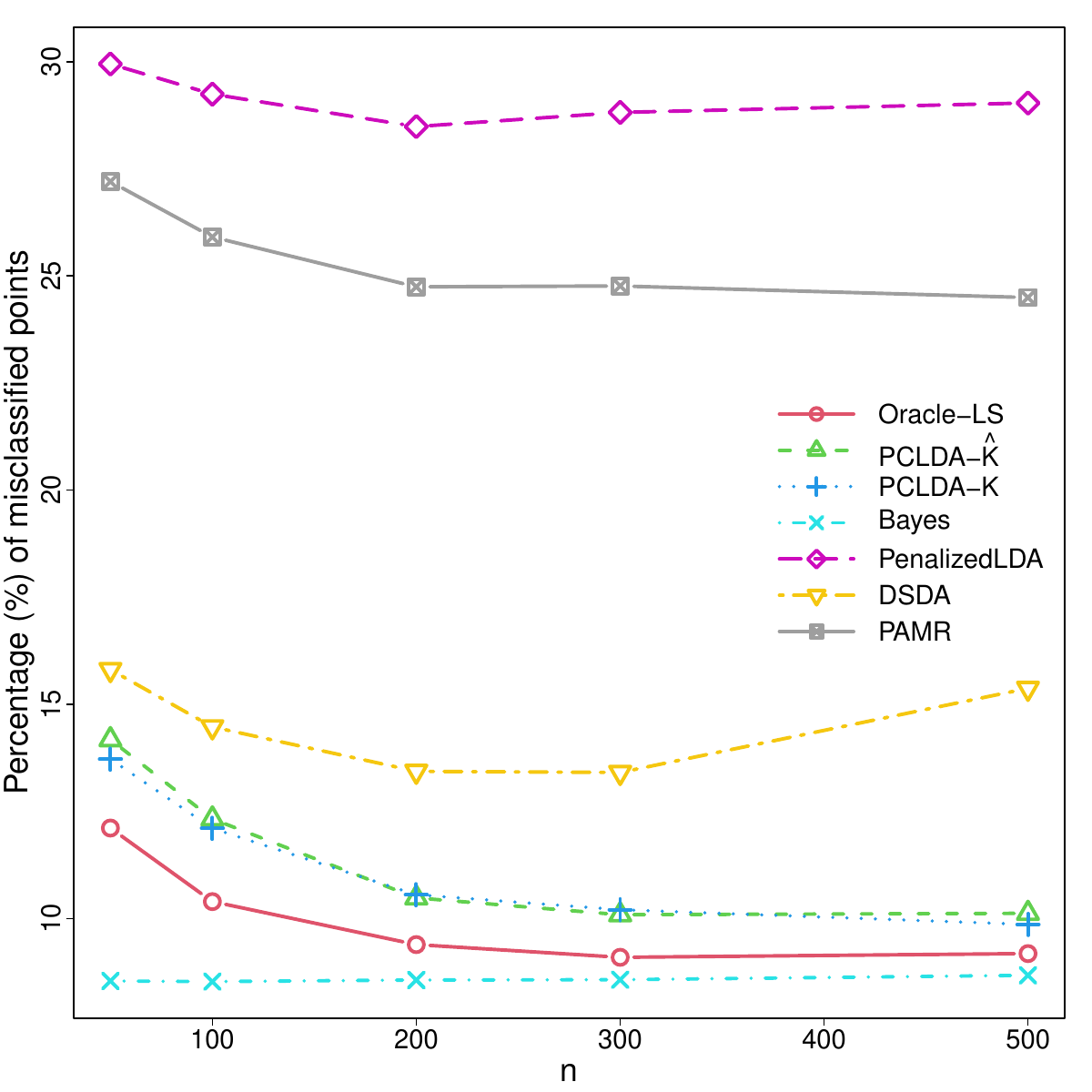}
	\caption{The averaged misclassification errors of each algorithm for various choices of $n$}
	\label{fig_K_1}
\end{figure}

To further examine the robustness of PCLDA-$s$ by using different $s$, we chose $s$ within $\{6, 8, 10, 15, 20, 30\}$ and compared the corresponding PCLDA-$s$ with the Bayes error and the Oracle-LS. Recall that the true $K$ is $10$.  Figure \ref{fig_K_s} shows that PCLDA has  robust performance across a wide range of $s$, and this range gets wider as the sample size increases. One extreme choice is $s = p$ in which case $\wh\theta$ reduces to the minimum-norm interpolator $(\Pi_n\X)^+\Y$, which, as analyzed in \cite{BW22}, has promising performance when $p\gg n$.

\begin{figure}[ht]
	\centering
	\includegraphics[width=\textwidth]{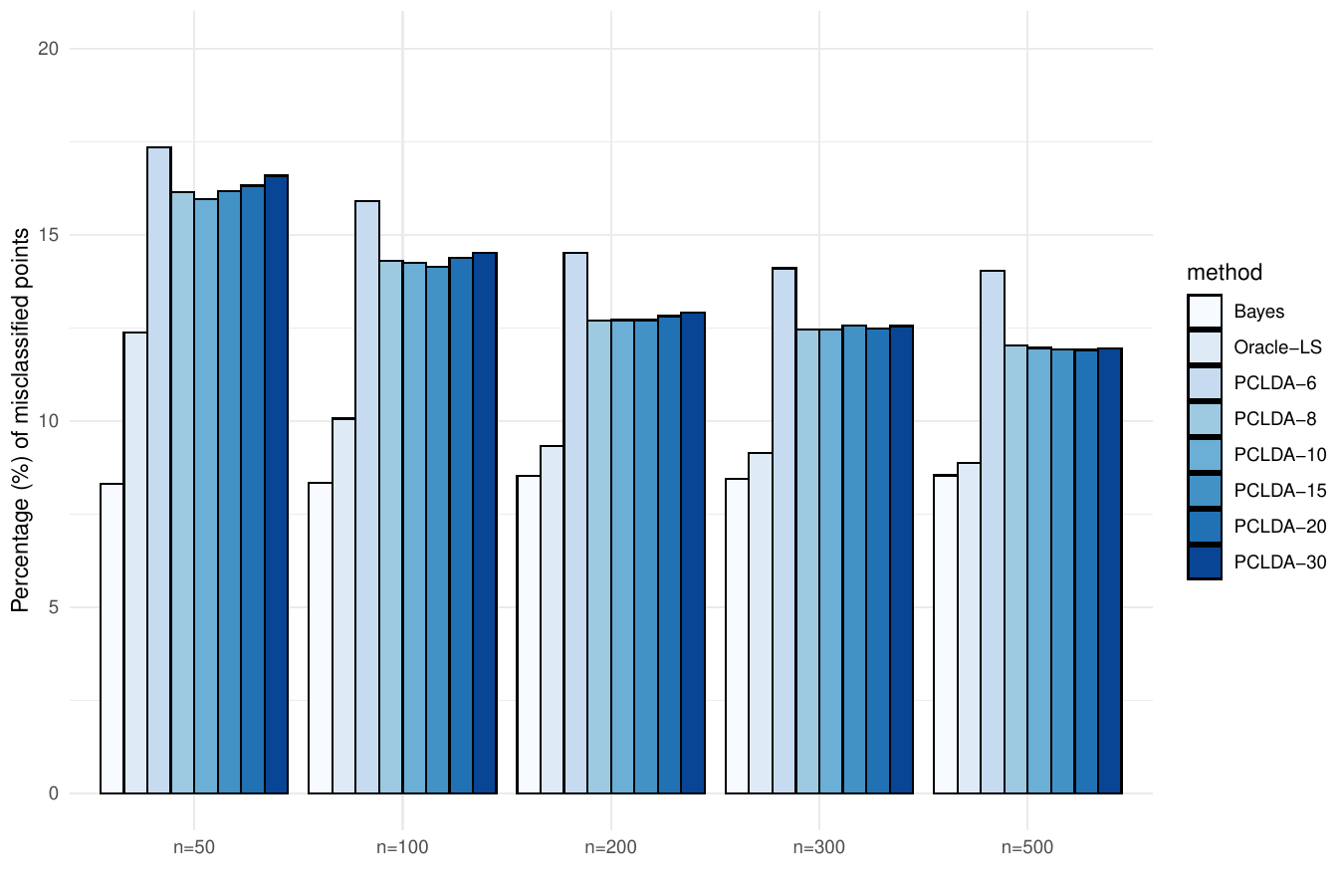}
	\caption{The averaged misclassification errors of each PCLDA-$s$ for various choices of $s$}
	\label{fig_K_s}
\end{figure}

\subsection{Benefit of using an auxiliary feature data set}\label{app_sim_aux}

In this section we conduct a simulation study to examine the benefit of using an auxiliary data set to construct $\wt\U_K$, and to investigate how many  auxiliary data points are required to estimate $P_A$ accurately enough to yield an improvement over the classifier entirely based on the training data $\D$.

We consider $K = 10$, $p = 300$ and $n \in \{50, 100, 200, 300\}$.
We adopt the same data generating mechanism used in our simulation study of Section \ref{sec_sim} and increase the number of repetitions in each setting to $300$ and the number of data points in the test data to $500$. We denote by PCLDA-split-$n'$ the proposed method that uses an independent copy of $\X$ with $n'$ data points to compute $\wt \U_K$. We consider $n'\in \{20, 30, 50, 100, 300, 500, 700\}$. In addition to Oracle-LS, Bayes and PCLDA-$K$ (the procedure only using the training data), we choose the method of using the true $A$, denoted by PCLDA-split-inf, as another benchmark. 

Figure \ref{fig_aux} depicts the performance of various methods in the strong signal-to-noise ratio (SNR) setting where $\lambda_1\asymp \lambda_K \asymp p$. From Figure \ref{fig_aux} we can see that one needs $n'\ge 100$ for PCLDA-split-$n'$ to have nearly the same performance as PCLDA-split-inf, though $n'=50$ already yields similar performance. Since improvement over $n\ge 300$ is small, we exclude the results for $n \in \{500,700\}$. Comparing to PCLDA-$K$, PCLDA-split-$n'$ starts showing small advantage for $n'\ge 100$. Since we have strong SNR in this setting, the advantage of using auxiliary data set is not considerable, in line with our discussion in Remark \ref{rem_cond_PCR}. 

We further consider in Figure \ref{fig_aux_weak} the weak SNR setting where entries of $A$ are generated as described in Appendix \ref{app_sim_K}. 
As we can see, the advantage of using auxiliary data becomes more visible in the weak SNR setting.  PCLDA-split-$n'$ seems to start outperforming PCLDA-$K$ when $n'\ge n$, suggesting that the same amount of auxiliary data points is needed for PCLDA-split-$n'$ to show improvement  over PCLDA-$K$.

\begin{figure}[ht]
	\centering
	\includegraphics[width=\textwidth]{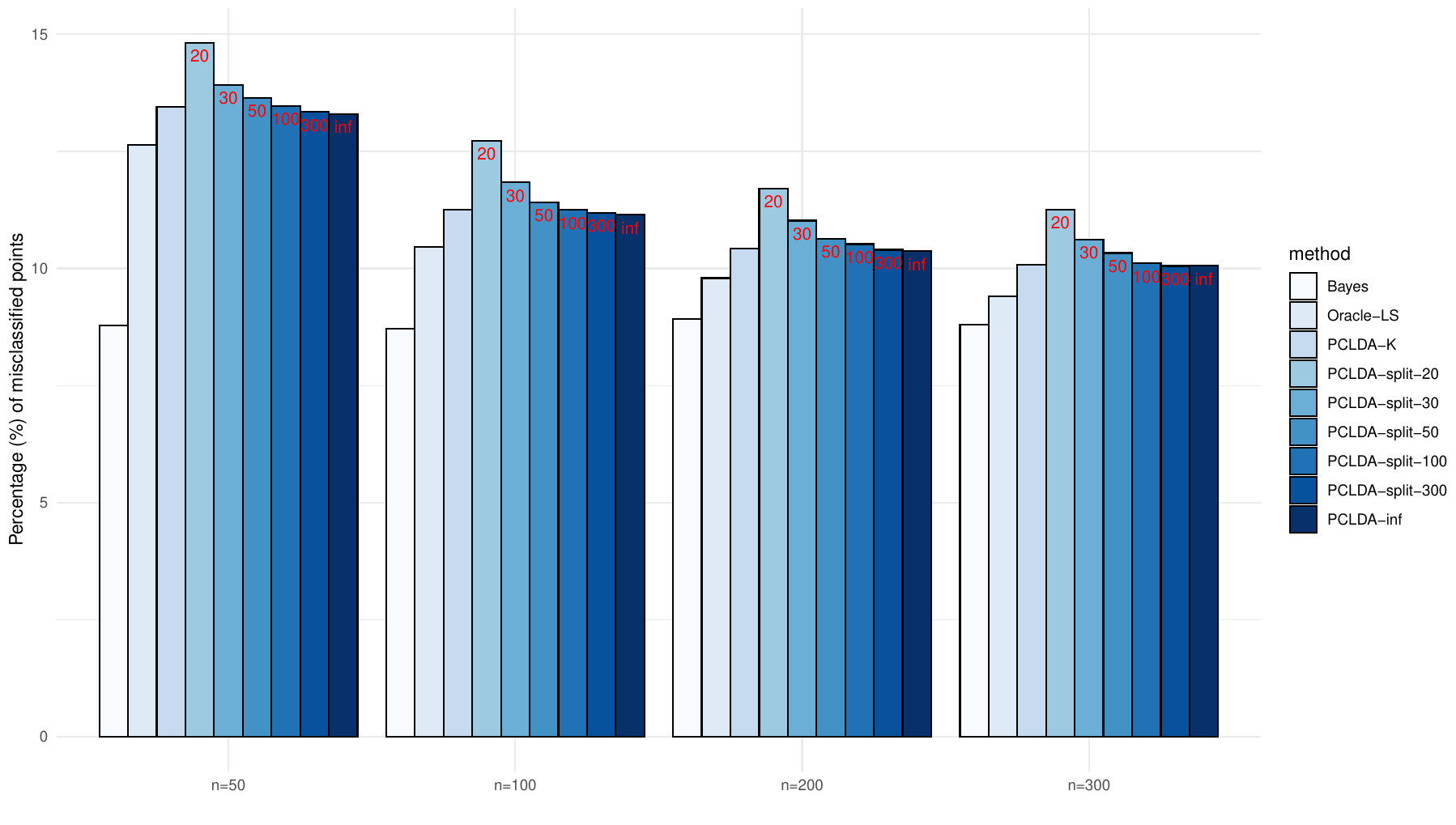}
	\caption{The averaged misclassification errors of PCLDA-split-$n'$ for various choices of $n'$ in the strong SNR setting}
	\label{fig_aux}
\end{figure}

\begin{figure}[ht]
	\centering
	\includegraphics[width=\textwidth]{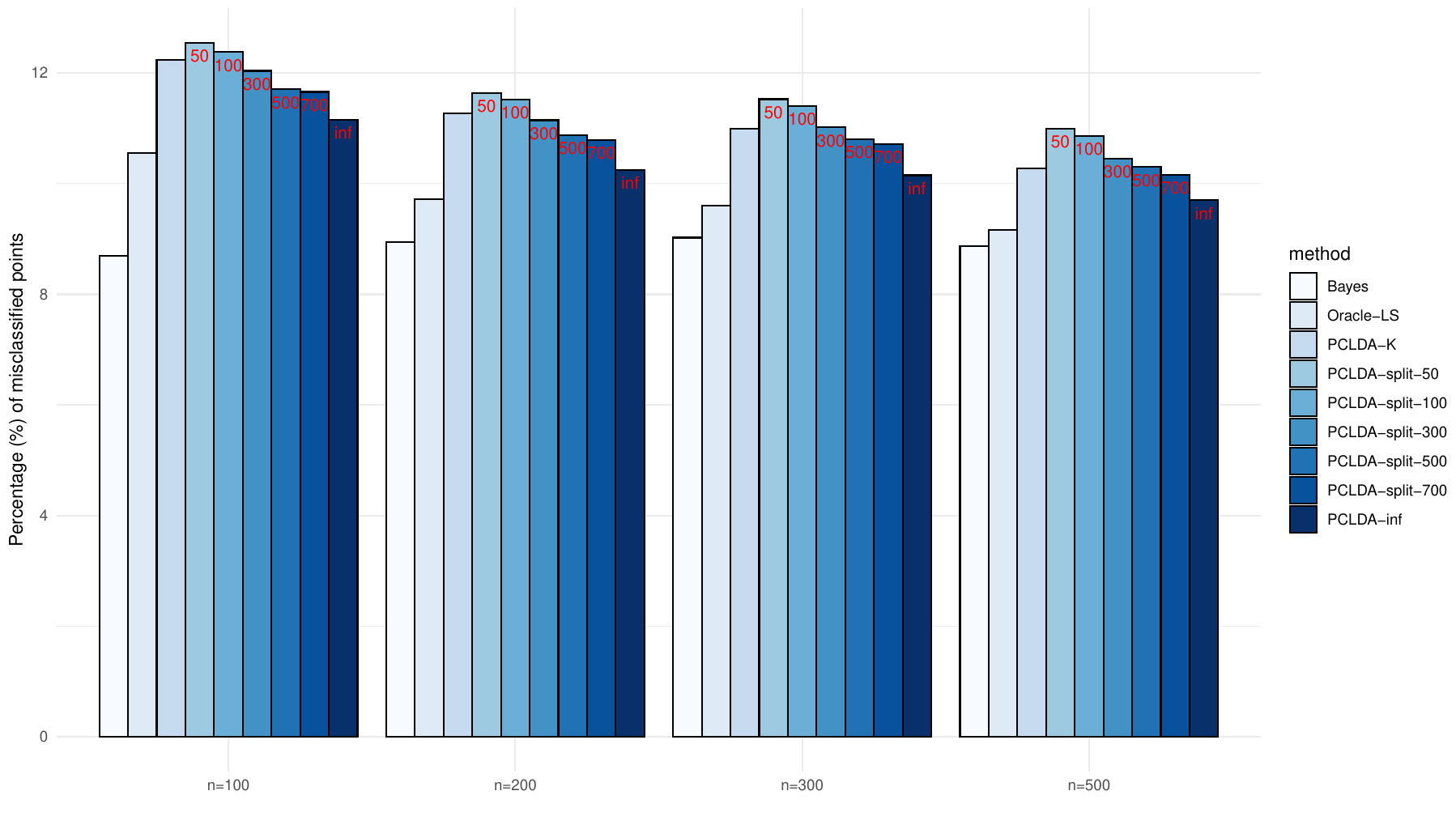}
	\caption{The averaged misclassification errors of PCLDA-split-$n'$ for various choices of $n'$ in the weak SNR setting}
	\label{fig_aux_weak}
\end{figure}

\subsection{Performance of the proposed procedure for multi-class classification}\label{app_sim_multi}

In this section we evaluate the proposed approach for multi-class classification. We take the same data generating mechanism with the exception that the centers $\a_\ell$ for $\ell \in \cL$ are generated as i.i.d. realizations of $N(0, 2/K)$ and the priors are set to $\pi_\ell = 1/L$. For ease of presentation, we only consider PCLDA-$K$ and its averaged version, PCLDA-$K$-avg, given by Remark \ref{rem_multi}. We also consider PCLDA-$K$-plugin, the classical LDA rule by using the projections $\wh\Z := \X \U_K$ in place of the unobserved $\Z$.  
For comparison,  we include the PenalizedLDA and PAMR classifiers as well.  

We first examine the effect of the number of total classes, $L$, on the proposed approach. Fix $K= 10$, $p = 300$ and $n = 500$ with $L$ varying within $\{2, 3, 4, 5, 6\}$. Each setting is repeated $100$ times with $300$ test data points.  Figure \ref{fig_multi_L} reveals that PCLDA-$K$, PCLDA-$K$-avg and PCLDA-$K$-plugin have similar performance. As $L$ increases, the misclassification errors of all three methods   increase, in line with Theorem \ref{thm_risk_multi} and Corollary \ref{cor_risk_multi}, meanwhile
PCLDA-$K$-plugin and PCLDA-$K$-avg tend  to have an advantage over PCLDA-$K$. 

We further vary $n\in \{100, 200, 400, 600, 800\}$ with fixed $K = 10$, $p = 500$ and $L = 4$. As shown in Figure \ref{fig_multi_n}, all methods have smaller misclassification errors as $n$ increases while the advantages of PCLDA-$K$-plugin and PCLDA-$K$-avg over PCLDA-$K$ become more visible for smaller sample sizes. 
We also see that PCLDA-$K$-plugin has slightly better performance than PCLDA-$K$-avg for small $n$. On the other hand, the proposed multi-class classification, such as PCLDA-$K$-avg, is based on the regression formulation, hence more amenable to structural estimation of the discriminant direction $\beta$. For instance, in the high-dimensional LDA setting, the regression based approach \citep{mai2012} has  net computational advantage over the procedure based on the plug-in rule \citep{caizhang2019}. The regression formulation also transfers related notions of regression methods to discriminant analysis, such as the degrees of freedom, which can be used for selecting tuning parameters in penalized  discriminant analysis (see \cite{hastie1995penalized} for details). 
A regression-based approach for multi-class classification that performs as well as PCLDA-$K$-plugin deserves a full separate investigation. We leave this for future research.

\begin{figure}[ht]
	\centering
	\begin{subfigure}{.48\textwidth}
		\includegraphics[width=\textwidth]{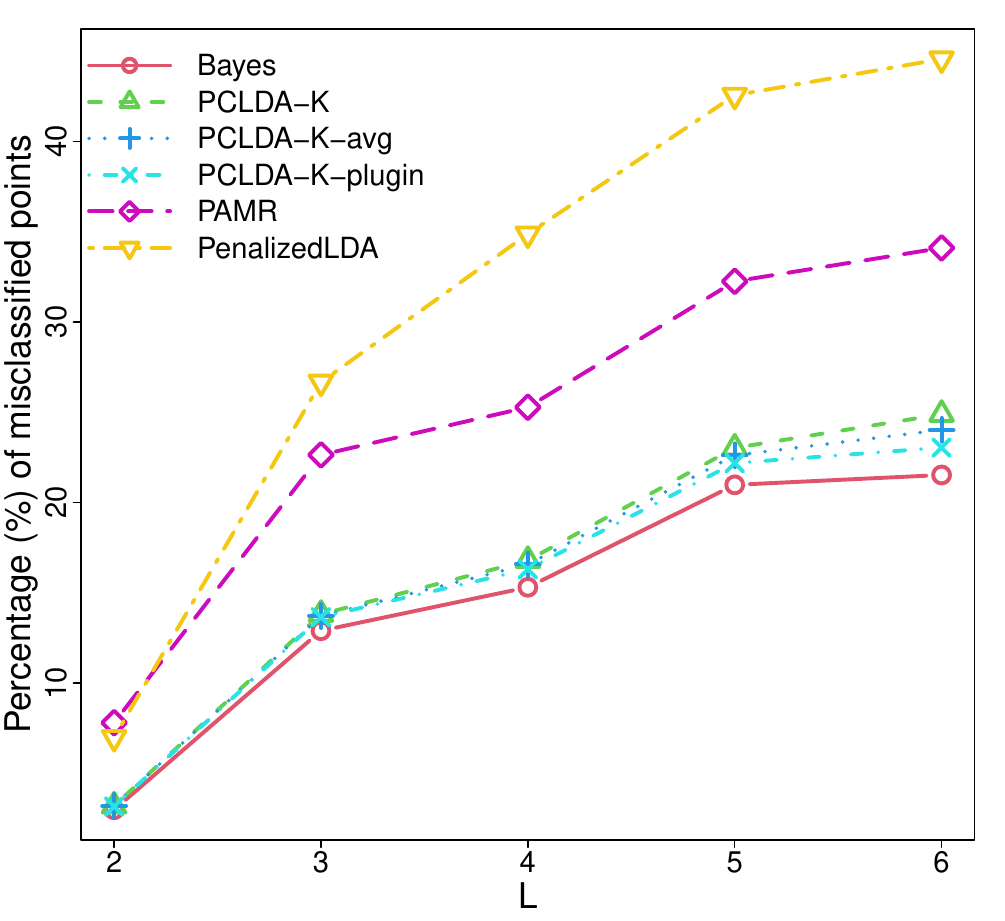}
		\vspace{-3mm}
		\caption{$K= 10$, $p = 300$, $n = 500$}
		\label{fig_multi_L} 
	\end{subfigure}
	\hfill 
	\begin{subfigure}{.48\textwidth}
		\includegraphics[width=\textwidth]{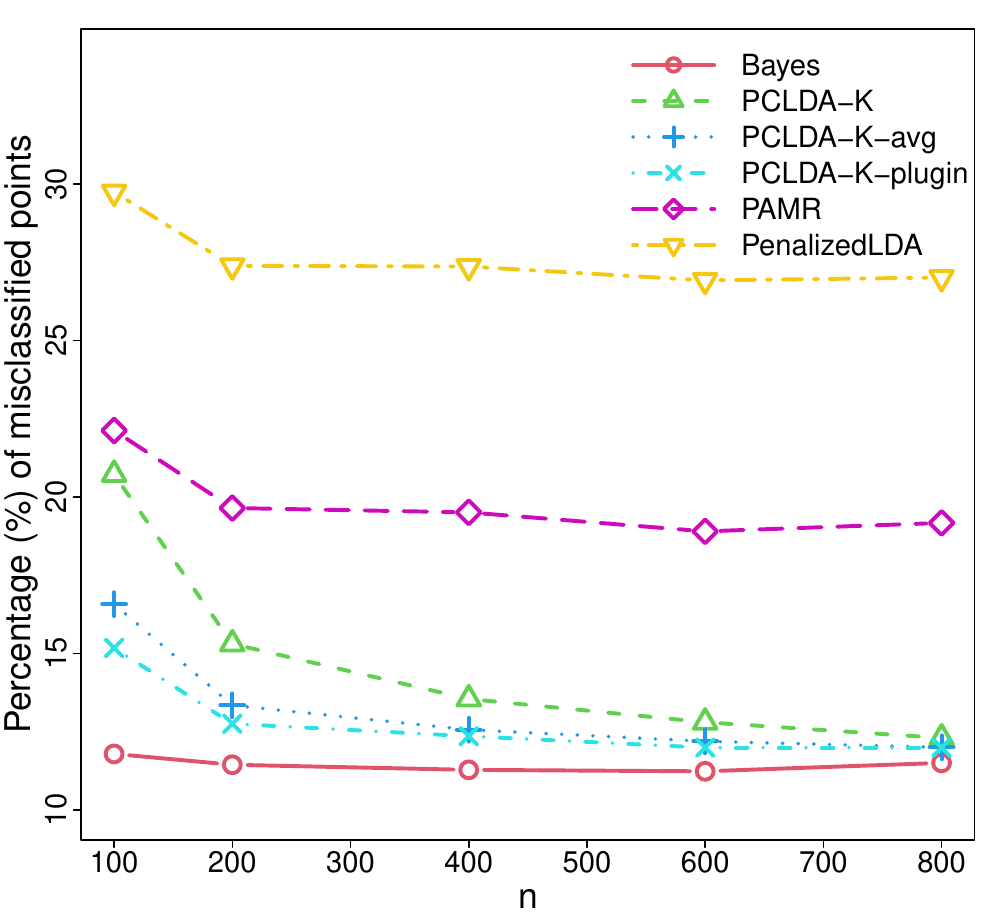}
		\vspace{-3mm}
		\caption{$K= 10$, $p = 300$, $L = 4$}    
		\label{fig_multi_n}
	\end{subfigure}
	\caption{The averaged misclassification errors of multi-class classification procedures}
	\label{fig_multi}
\end{figure}

\end{document}